\documentclass[12pt,oneside]{amsart}

\usepackage{graphicx}
\usepackage{amsfonts}
\usepackage{epsf}
\usepackage{amssymb}
\usepackage{amsmath}
\usepackage{amscd}
\usepackage{tikz}
\usepackage{pdfpages}
\usepackage{fancyhdr}
\usepackage{setspace}
\usepackage{hyperref}
\usepackage[all]{xy}
\usetikzlibrary{matrix}
\usepackage{verbatim}

\theoremstyle{definition}
\newtheorem{theorem}{Theorem}[section]
\newtheorem{proposition}[theorem]{Proposition}
\newtheorem{lemma}[theorem]{Lemma}

\newtheorem{question}[theorem]{Question}
\newtheorem{corollary}[theorem]{Corollary}

\newcommand{\calc}{\mathcal{C}}

\newcommand{\Z}{\mathbb{Z}}

\newcommand{\Q}{\mathbb{Q}}

\newcommand{\CP}{\mathbb{CP}}

\newcommand{\bi}{\begin{itemize}}
\newcommand{\ei}{\end{itemize}}
\newcommand{\be}{\begin{enumerate}}
\newcommand{\ee}{\end{enumerate}}

\newcommand{\n}{\beta}

\newcommand{\emp}{\emptyset}
\newcommand{\X}{\times}

\newcommand{\eps}{\epsilon}

\newcommand{\A}{\alpha}
\newcommand{\pd}{\partial}

\newcommand{\g}{\gamma}
\newcommand{\io}{\iota}

\makeatletter
\newtheorem*{rep@theorem}{\rep@title}
\newcommand{\newreptheorem}[2]{%
\newenvironment{rep#1}[1]{%
 \def\rep@title{#2 \ref{##1}}%
 \begin{rep@theorem}}%
 {\end{rep@theorem}}}
\makeatother

\newreptheorem{theorem}{Theorem}
\newreptheorem{lemma}{Lemma}
\newreptheorem{question}{Question}
\newreptheorem{corollary}{Corollary}

\begin{document}

\rhead{\thepage}
\lhead{\author}
\thispagestyle{empty}


\raggedbottom
\pagenumbering{arabic}
\setcounter{section}{0}



\title{Genus two trisections are standard}
\date{\today}
\author{Jeffrey Meier \and Alexander Zupan}

\begin{abstract}
We show that the only closed 4--manifolds admitting genus two trisections are $S^2 \X S^2$ and connected sums of $S^1 \X S^3$, $\CP^2$, and $\overline{\CP}^2$ with two summands.  Moreover, each of these manifolds admits a unique genus two trisection up to diffeomorphism.  The proof relies heavily on the combinatorics of genus two Heegaard diagrams of $S^3$. As a corollary, we classify two-component links contained in a genus two Heegaard surface for $S^3$ with a surface-sloped cosmetic Dehn surgery.
\end{abstract}
\maketitle

\section{Introduction}

A trisection of a smooth 4--manifold $X$ is a decomposition of $X$ into three pieces, each of which is diffeomorphic to a 4--dimensional handlebody.  Trisections have been introduced by Gay and Kirby as a 4--dimensional analogue to Heegaard splittings of 3--manifolds \cite{gay-kirby:trisections}.  An explicit goal of the theory of trisections is to provide a vehicle to apply 3--manifold techniques to problems in 4--dimensional topology.  In their seminal paper, they prove that every closed, orientable 4--manifold admits a trisection, and any two trisections for the same 4--manifold become isotopic after some number of iterations of a natural stabilization operation, akin to the Reidemeister-Singer Theorem in dimension three.

While a Heegaard splitting of a 3--manifold has only one complexity parameter (the genus of the splitting), a trisection comes equipped with two such parameters.  A \emph{$(g,k)$--trisection} of a closed, orientable, smooth 4--manifold $X$ is a decomposition
\[ X = X_1 \cup X_2 \cup X_3,\]
where each $X_i$ is a 4--dimensional 1--handlebody obtained by attaching $k$ 1--handles to one 0--handle, each $Y_{ij} = X_i \cap X_j$ is 3--dimensional genus $g$ handlebody, and $\Sigma = X_1 \cap X_2 \cap X_3$ is a closed genus $g$ surface, which we call a \emph{trisection surface}.  Observe that $\Sigma$ is a Heegaard surface for each 3--manifold $\pd X_i = \#^k (S^1 \X S^2)$; hence, the definition implies that $g \geq k$.

Since the field is still in its infancy, there are many basic questions about trisections which have not yet been answered.  In terms of 3--manifolds, the only manifold $Y$ with a genus zero Heegaard splitting is $S^3$, and the manifolds with genus one splittings (lens spaces, $S^1 \X S^2$, and $S^3$) are neatly parameterized by the extended rational numbers.  This leads to the following general question.

\begin{question}
To what extent can we enumerate $(g,k)$--trisections for small values of $g$?
\end{question}

In addition to proving the existence of trisections, Gay and Kirby relate them to handle decompositions.  They show that a 4--manifold admitting a $(g,k)$--trisection has a handle decomposition with one 0--handle, $k$ 1--handles, $(g-k)$ 2--handles, $k$ 3--handles, and one 4--handle.  As such, it is easy to classify $(g,g)$--trisections.  If $X$ has a $(g,g)$--trisection, then $X$ has a handle decomposition with no 2--handles; therefore, by \cite{laudenbach-poenaru}, we have that $X$ is diffeomorphic to $\#^g (S^1 \X S^3)$ (where $\#^0 (S^1 \X S^3) = S^4$).  It follows that there is a unique genus zero trisection: the standard $(0,0)$--trisection of $S^4$ into three 4--balls, two of which are glued to each other along one hemisphere of their 3--sphere boundaries, and the third of which is attached along the resulting 3--sphere boundary.

Note that every Heegaard splitting $Y = H_1 \cup_{\Sigma} H_2$ can be presented by a Heegaard diagram (in fact, infinitely many Heegaard diagrams), which is a triple denoted $(\Sigma,\A,\n)$.  This notion has an analogue in terms of trisections: a \emph{$(g,k)$--trisection diagram} is a quadruple $(\Sigma,\A,\n,\g)$ such that each Heegaard diagram $(\Sigma,\A,\n)$, $(\Sigma,\n,\g)$, and $(\Sigma,\A,\g)$ presents a genus $g$ splitting of $\#^k(S^1 \X S^2)$.  Every trisection diagram yields a straightforward construction of a trisected 4--manifold $X$ (see Section \ref{sec:prelims}).

The manifolds that admit genus one trisections are not difficult to classify; they are described in detail in \cite{gay-kirby:trisections} and diagrams of these trisections are pictured in Figure \ref{fig:G1Diags}.  There are precisely two manifolds which admit $(1,0)$--trisections: $\CP^2$ and $\overline{\CP}^2$.  Thus, the situation is somewhat different here than in dimension three, suggesting another basic question.

\begin{figure}[h!]
\centering
\includegraphics[scale = 1.75]{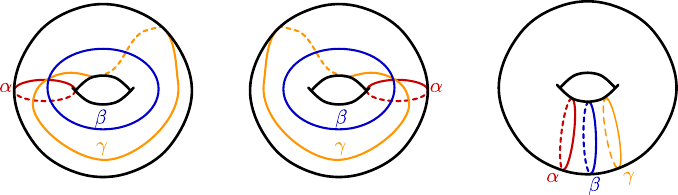}
\put(-298,-20){(a)}
\put(-178,-20){(b)}
\put(-53,-20){(c)}
\caption{The three possible genus one trisection diagrams: the (1,0)--trisection diagrams for (a) $\CP^2$ and (b) $\overline{\CP}^2$, and the (1,1)--trisection diagram for (c) $S^1\times S^3$.}
\label{fig:G1Diags}
\end{figure}

\begin{question}\label{infinite}
What is the smallest value of $g$ for which there are infinitely many 4--manifolds $X$ which admit a $(g,k)$--trisection for some $k$?
\end{question}

Work in progress by the authors and David Gay reveals that the answer to Question \ref{infinite} satisfies $g \leq 3$. It is shown that if $X$ is the double cover of $S^4$ branched over a $k$--twist spun 2--bridge knot where $k \neq \pm 1$, then $X$ admits a minimal genus $(3,1)$--trisection \cite{gmz}.  Hence, Question \ref{infinite} reduces to understanding genus two trisections.  This brings us to the main result of the present paper.

\begin{theorem}\label{thm:main}
If $X$ admits a genus two trisection, then $X$ is either $S^2 \X S^2$ or a connected sum of $S^1 \X S^3$, $\CP$, and $\overline{\CP}^2$ with two summands.  Moreover, each of these 4--manifolds has a unique genus two trisection up to diffeomorphism.
\end{theorem}
\begin{proof}
Suppose $X$ admits a $(2,k)$--trisection $X = X_1 \cup X_2 \cup X_3$.  If $k = 2$, then trivially $X$ is diffeomorphic to $(S^1 \X S^3) \# (S^1 \X S^3)$ and has a trisection diagram $(\A,\n,\g)$ such that $\A = \n = \g$.  If $k = 1$, then by Corollary \ref{21class}, the trisection is reducible and can be written as the connected sum of one of the standard $(1,0)$--trisections and the standard $(1,1)$--trisection shown in Figure \ref{fig:G1Diags}.  It follows that $X$ is diffeomorphic to $(S^1 \X S^3) \# \CP$ or $(S^1 \X S^3) \# \overline{\CP}^2$.  Finally, if $k = 2$, then by Theorem \ref{claim2}, the trisection has a diagram homeomorphic to one of the standard diagrams pictured in Figure \ref{fig:StdDiags}, and thus $X$ is diffeomorphic to $S^2 \X S^2$, $\CP^2 \# \CP^2$, or $\CP^2 \# \overline{\CP}^2$ by \cite{gay-kirby:trisections}.
\end{proof}

\begin{figure}[h!]
\centering
\includegraphics[scale = 1.75]{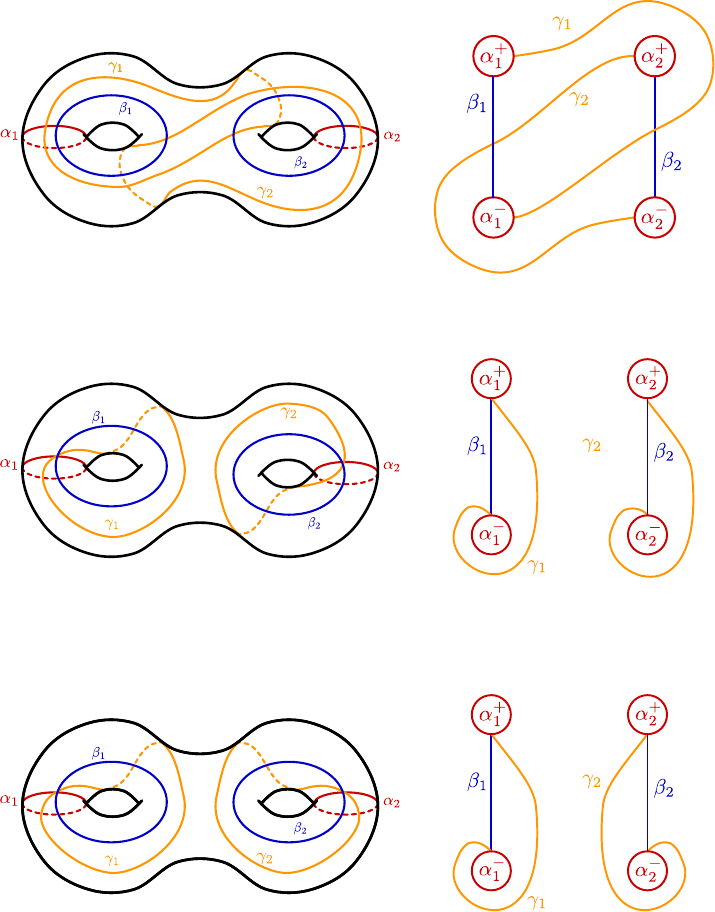}
\put(-190,325){(a)}
\put(-190,160){(b)}
\put(-190,-10){(c)}
\caption{The three standard  (2,0)--trisection diagrams: (a) $S^2\times S^2$, (b) $\CP^2\#\CP^2$, and (c) $\CP^2\#\overline{\CP}^2$.  By Theorem \ref{thm:main}, every $(2,0)$--trisection is represented by one of these three diagrams.}
\label{fig:StdDiags}
\end{figure}

The bulk of this paper is devoted to proving Theorem \ref{claim2}, the classification of $(2,0)$--trisections, by understanding $(2,0)$--trisection diagrams $(\Sigma, \A, \n, \g)$.  In this case, each of $(\Sigma, \A,\n)$, $(\Sigma, \n,\g)$, and $(\Sigma, \g,\A)$ is a genus two Heegaard diagram for $S^3$, and these are well-studied objects in 3--manifold topology.  A result of Homma-Ochiai-Takahashi \cite{hot} states that every genus two Heegaard splitting of $S^3$ can be reduced algorithmically (see Section \ref{sec:prelims}), and this is the main input into our proof of Theorem \ref{claim2}.

As a corollary to the main theorem, we obtain the following result (see Section \ref{sec:surgery} for definitions).

\begin{corollary}\label{cor:surgery}
Suppose that $L$ is a two-component link contained in a genus two Heegaard surface $\Sigma$ for $S^3$, and $L$ admits an $S^3$ surgery with slope given by the surface slope of $L$ in $\Sigma$.  Then there is a series of handle slides on $L$ contained in the surface $\Sigma$ that converts $L$ to a $(\pm 1)$--framed unlink or a $0$--framed Hopf link.
\end{corollary}


In Section \ref{sec:surgery}, we  discuss how Corollary \ref{cor:surgery} relates to the search for exotic, simply connected 4--manifolds with $b_2=2$.  We also remark on the similarity between Corollary \ref{cor:surgery} and the Generalized Property R Conjecture \cite{gst}. 

\subsection{Organization}\ 

In Section \ref{sec:prelims}, we introduce notation and terminology which will be used in the rest of the paper.   In Section \ref{sec:start}, we define a complexity measure for $(2,0)$--trisection diagrams and prove some preliminary results about this complexity.  Section \ref{sec:digress} includes a digression into a special type of Heegaard diagram for $S^3$, which we will use to further reduce the complexity of an arbitrary trisection diagram in Section \ref{sec:hard_case}.  We prove the main theorem about $(2,0)$--trisections in Section \ref{sec:main} and conclude by classifying $(2,1)$--trisections and establishing Corollary \ref{cor:surgery} in Section \ref{sec:surgery}.

\subsection{Acknowledgements}\ 

The authors would like to thank Ken Baker and David Gay for their interest in this project and for their insightful comments at the outset, as well as Cameron Gordon, for helpful discussions related to the intricacies of the arguments presented below.  The first author is supported by NSF grant DMS-1400543, and the second author is supported by NSF grant DMS-1203988.

\section{Preliminary Set-up}\label{sec:prelims}

Manifolds are orientable and connected unless otherwise specified.  For the remainder of the paper, we fix a genus two surface $\Sigma$ and let $\nu( \cdot )$ denote an open regular neighborhood in $\Sigma$.  We use the term \emph{curve} to mean an isotopy class of simple closed curves in $\Sigma$.  Given two distinct curves $c_1$ and $c_2$ in $\Sigma$, we assume that $c_1$ and $c_2$ have been isotoped so that $|c_1 \cap c_2|$ is minimal, and we let $\io(c_1,c_2)$ denote this geometric intersection number.  By orienting $c_1$ and $c_2$, we can also compute the algebraic intersection number, denoted $c_1 \cdot c_2$, by counting the signed intersection number of $c_1$ and $c_2$.

A \emph{cut system} $\A$ for $\Sigma$ is a pair of disjoint curves $\A_1$ and $\A_2$ that cut $\Sigma$ into a four-punctured sphere, which we will denote $\Sigma_{\A}$.  In a slight abuse of language, we refer to $\Sigma_\A$ as both a four-punctured sphere and a sphere with four boundary components.  A \emph{Heegaard diagram} $(\Sigma,\A,\n)$ consists of two cut systems $\A$ and $\n$ for $\Sigma$, whereas a \emph{trisection diagram} $(\Sigma,\A,\n,\g)$ consists of three cut systems $\A$, $\n$, and $\g$ for $\Sigma$ with the additional restrictions described above.  We often suppress $\Sigma$ and simply write $(\A,\n)$ and $(\A,\n,\g)$ for Heegaard and trisection diagrams, respectively.  Given a Heegaard diagram $(\A,\n)$, we may construct a 3--manifold by gluing 3--dimensional 2--handles along $\A$ and $\n$ on opposite sides of a product neighborhood $\Sigma \X I$ of $\Sigma$ and capping off the resulting 2--sphere boundary components with 3--balls.  Similarly, given a trisection diagram $(\Sigma,\A,\n,\g)$, we may construct a 4--manifold by taking a 4--dimensional regular neighborhood $\Sigma \X D$ of $\Sigma$ and attaching 4--dimensional 2--handles along $\A \X \{1\}$, $\n \X \{e^{2\pi i/3}\}$, and $\g \X \{e^{4\pi i/3}\}$, where $D$ is viewed as the complex unit disk and the framings of the handle attachments are given by $\Sigma \X \{p\}$.  Finally, we cap off boundary components with 3-- and 4--handles, which can be done in a unique way by \cite{laudenbach-poenaru}.

We begin by stating several standard lemmata dealing with genus two Heegaard diagrams of $S^3$.  We say that a Heegaard diagram $(\A,\n)$ for $S^3$, where $\A = \{\A_1,\A_2\}$ and $\n = \{\n_1,\n_2\}$, is \emph{trivial} if $\io(\A_i,\n_j) = \delta_{ij}$.  
Trivial Heegaard diagrams are unique (up to homeomorphism).  Given a Heegaard diagram $(\A,\n)$, we define the \emph{intersection matrix} $M(\A,\n)$ of $(\A,\n)$ to be
\[ M(\A,\n) = \begin{pmatrix}
\A_1 \cdot \n_1 & \A_1 \cdot \n_2 \\
\A_2 \cdot \n_1 & \A_2 \cdot \n_2
\end{pmatrix}.\]

\begin{lemma}\label{determ}
If $(\A,\n)$ is a Heegaard diagram for $S^3$, then $|\text{det}(M(\A,\n))| = 1$.
\end{lemma}

It follows immediately that if $(\A,\n)$ is a Heegaard diagram for $S^3$ satisfying $\io(\A,\n) = 2$, then $(\A,\n)$ is the trivial diagram.  Now, suppose that $\A_1,\A_2,$ and $\A_3$ are pairwise disjoint curves in $\Sigma$, and let $\A = \{\A_1,\A_2\}$ and $\A' = \{\A_2,\A_3\}$.  We say that $\A$ and $\A'$ are related by a \emph{handle slide}.

\begin{lemma}\cite{johannson}
Two Heegaard diagrams $(\A,\n)$ and $(\A',\n')$ determine the same Heegaard splitting if and only if there are sequences of handle slides taking $\A$ to $\A'$ and $\n$ to $\n'$.
\end{lemma}

We now turn our attention to trisection diagrams.  The first lemma is due to Gay and Kirby.

\begin{lemma}\cite{gay-kirby:trisections} 
Two trisection diagrams $(\A,\n,\g)$ and $(\A',\n',\g')$ determine the same trisection if and only if there are sequences of handle slides taking $\A$ to $\A$, $\n$ to $\n'$, and $\g$ to $\g'$.
\end{lemma}

We say that a $(2,0)$--trisecton diagram $(\A,\n,\g)$ is \emph{standard} if $\io(\A,\n) = \io(\A,\g) = \io(\n,\g) = 2$.  In Figure \ref{fig:StdDiags}, $\Sigma_\A$ is shown for each of the three standard $(2,0)$--trisection diagrams.  The main theorem presented in this paper asserts that for any $(2,0)$--trisection diagram $(\A,\n,\g)$, there is a sequence of handle slides taking $(\A,\n,\g)$ to one of these three standard diagrams.


For $\A = \{\A_1,\A_2\}$ and $\mathcal C$ a collection of curves in $\Sigma$,  intersection $\mathcal C \cap \Sigma_{\A}$ is a collection of properly embedded essential arcs.  Viewing the boundary components of $\Sigma_\A$ as four fat vertices and the essential arcs as edges, we may consider $\mathcal C \cap \Sigma_\A$ as a (topological) graph, the \emph{Whitehead graph} of $\mathcal C$ with respect to $\A$, denoted $\Sigma_{\A}(\mathcal C)$.  By collapsing each boundary component of $\Sigma_{\A}$ to a vertex and collections of parallel edges in $\Sigma_{\A}(\mathcal C)$ to a single edge, we construct the \emph{reduced Whitehead graph} $G_{\A}(\mathcal C)$.  To avoid confusion, we will refer to edges in $\Sigma_{\A}(\mathcal C)$ as \emph{arcs} and edges in $G_{\A}(\mathcal C)$ as \emph{edges}.  If an arc $a$ in $\Sigma_\A(\mathcal C)$ corresponds to an edge $e$ in $G_{\A}(\mathcal C)$, we say that $a$ is \emph{in the edge} $e$.  The \emph{weight} of an edge $e \in G_{\A}(\mathcal C)$, denoted $w_{\mathcal C}(e)$, is the number of arcs in $e$.  If $\mathcal C' \subset \mathcal C$, then $\Sigma_{\A}(\mathcal C')$ is a subgraph of $\Sigma_\A(\mathcal C)$ and $G_{\A}(\mathcal C')$ is a subgraph of $G_{\A}(\mathcal C)$.  In this case we let $w_{\mathcal C'}(e)$ denote the weight of $e$ as an edge in $G_{\A}(\mathcal C')$.  We use $\A_i^{\pm}$ ($i= 1,2$) to denote the vertices of both $\Sigma_\A(\mathcal C)$ and $G_{\A}(\mathcal C)$.

\begin{figure}[h!]
\centering
\includegraphics[scale = .9]{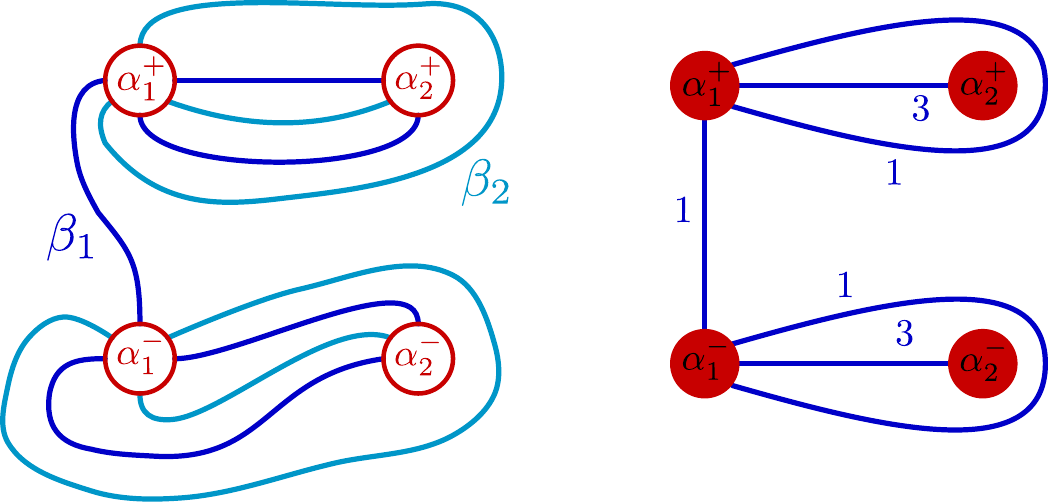}
\put(-207,-15){(a)}
\put(-57,-15){(b)}
\caption{An example of (a) the Whitehead graph $\Sigma_{\A}(\n)$ and (b) the corresponding reduced Whitehead graph $G_\A(\n)$ obtained from a Heegaard splitting $(\A,\n)$.}
\label{fig:WhGraphEx}
\end{figure}

Non-isotopic arcs $a$ and $b$ in $\Sigma_\A$ can intersect in two fundamentally different ways.  If an intersection point $p \in a \cap b$ is a vertex of a triangle in $\Sigma_\A$ cobounded by subarcs of $a$, $b$, and a boundary component $\A_i^{\pm}$, we call $p$ an \emph{inessential point of intersection based in $\A_i^{\pm}$}.  Note that if $a$ and $b$ are arcs of $c \cap \Sigma_{\A}$ and $d \cap \Sigma_{\A}$ for curves $c,d \subset \Sigma$, then there is an isotopy of $c$ and $d$ in $\Sigma$ which pushes $p$ through $\A_i^{\pm}$, converting it to an inessential point of intersection based in $\A_i^{\mp}$.  A point $p \in a \cap b$ which is not inessential is called an \emph{essential point of intersection}.  See Figure \ref{fig:InessentialInt}.  Observe that essential points of intersection are present in $G_\A(c \cup d)$ but we lose all information about inessential points of intersection when passing from $\Sigma_\A(c \cup d)$ to $G_{\A}(c \cup d)$.

\begin{figure}[h!]
\centering
\includegraphics[scale = .7]{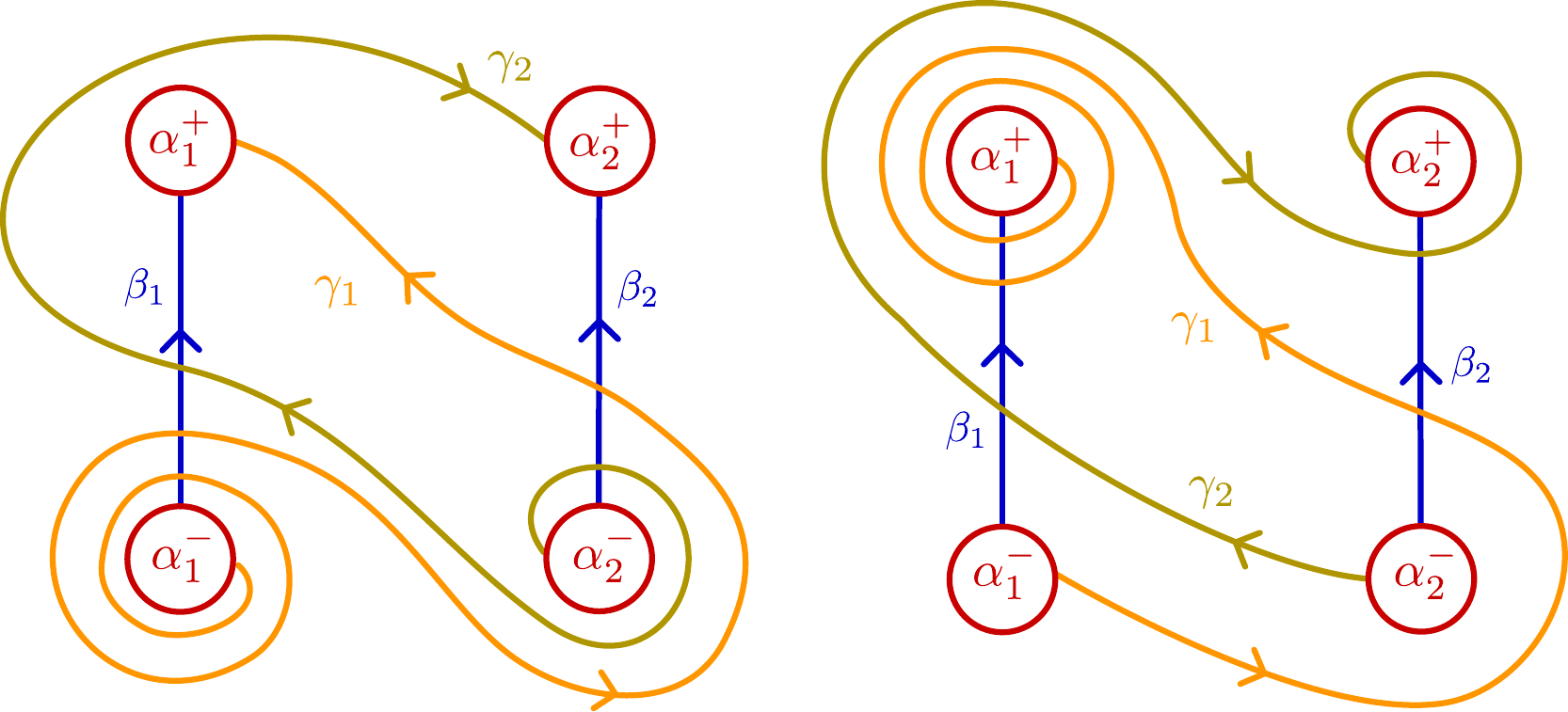}
\put(-260,-15){(a)}
\put(-90,-15){(b)}
\caption{An example (a) of $\Sigma_{\A}(\n,\g)$ such that $\n \cap \g$ contains a total of three inessential intersection points, all based in the vertices $\A_i^-$.  In (b), curves in $\g$ have been isotoped so that inessential intersections are based in the vertices $\A_i^+$.}
\label{fig:InessentialInt}
\end{figure}

An arc $a$ in $\Sigma_\A$ with endpoints in disjoint boundary components is called a \emph{seam}.  Otherwise $a$ is called a \emph{wave}.  Seams in $\Sigma_{\A}$ may be parameterized by the extended rational numbers $\overline{\Q}$ in the following way:  Choose four pairwise disjoint seams which cut $\Sigma$ into two squares with their vertices removed.  Assign the left and right sides of the square the slope $\frac{1}{0}$ and the top and bottom the slope $\frac{0}{1}$.  Any other seam may be isotoped to have constant rational slope in each of the squares, and we assign this slope to the seam.

We call a seam in $\Sigma_{\A}$ that connects either $\A_1^+$ or $\A_1^-$ to either $\A_2^+$ or $\A_2^-$ a \emph{cross-seam}, since it travels across $\Sigma_{\A}$ from one pair of boundary components to the other.  Cross-seams come in pairs:  There is a hyperelliptic involution $J$ of $\Sigma$ that leaves every curve invariant \cite{haas-susskind}.  Thus, the restriction of $J$ to $\Sigma_{\A}$ is a homeomorphism of $\Sigma_\A$ which maps $\A_i^{\pm}$ to $\A_i^{\mp}$.  It follows that if $c$ is a curve in $\Sigma$ and $a_+$ is a cross-seam of $c \cap \Sigma_{\A}$ connecting to $\A_1^+$ to $\A_2^{\pm}$, then since $J(c) = c$, we have that $a_- = J(a_+)$ is a cross-seam of $c \cap \Sigma_{\A}$ connecting $\A_1^-$ to $\A_2^{\mp}$. See Figure \ref{fig:Involutions}.

Next, we show that the slopes of seams determine the number of times they intersect essentially.  For each essential arc $a \subset \Sigma_{\A}$, there is a unique essential curve in $\Sigma_\A$, which we will henceforth call $\Delta_a$, such that $a \cap \Delta_a = \emp$.  See Figure \ref{fig:WaveSurgIntPf}.

\begin{figure}[h!]
\centering
\includegraphics[scale = 1.2]{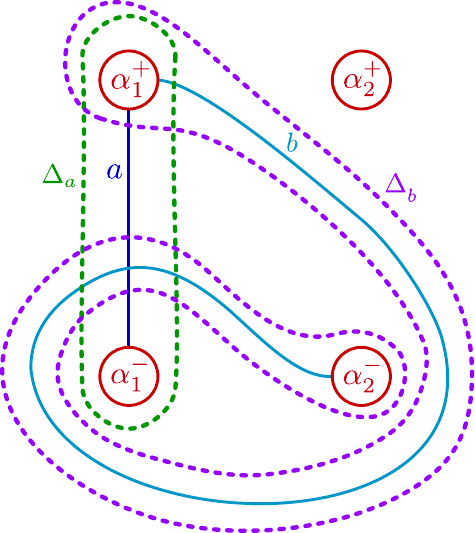}
\caption{An example depicting the relationship between $\io(a,b)$ and $\io(\Delta_a,\Delta_b)$ for arcs $a$ of slope $\frac{1}{0}$ and $b$ of slope $-\frac{1}{3}$, as discussed in the proof of Lemma \ref{inter}.}
\label{fig:WaveSurgIntPf}
\end{figure}

\begin{lemma}\label{inter}
If $a$ and $b$ are seams in $\Sigma_{\A}$ parameterized by slopes $\frac{w}{x}$ and $\frac{y}{z}$, respectively, and $a$ and $b$ have endpoints on $k$ common boundary components of $\pd \Sigma_{\A}$, (where $k\in\{0,1,2\})$, then $a$ and $b$ intersect essentially in $\frac{1}{2}\left(|wz-xy| - k\right)$ points.
\begin{proof}
It is well-known that $i(\Delta_a,\Delta_b) = 2|wz-xy|$.  Each common boundary component of $a$ and $b$ contributes two points of intersection to $\Delta_a \cap \Delta_b$, and each essential point of intersection of $a$ and $b$ contributes four points of intersection to $\Delta_a \cap \Delta_b$.  See Figure \ref{fig:WaveSurgIntPf}.  If there are $q$ essential points of $a \cap b$, it follows that
\[ 2|wz - xy| = 2k + 4q,\]
which easily yields the desired statement.
\end{proof}
\end{lemma}

Given a Heegaard diagram $(\A,\n)$, an $\A$--\emph{wave} for $(\alpha, \beta)$ is wave $\omega$ properly embedded in $\Sigma_\A$ such that $\omega \cap \n = \emp$.  A $\n$--\emph{wave} is defined similarly.  If such an arc exists, we say that $(\A,\n)$  contains a wave \emph{based in} $\A$ (resp., $\n$) and that $\Sigma_\A(\n)$ (resp., $\Sigma_\n(\A)$) contains a wave.  Let $\A = \{\A_1,\A_2\}$ and $\A_3 = \Delta_{\omega}$, and suppose that $\pd \omega \subset \A_1^{\pm}$.  Then $\A' = \{\A_2,\A_3\}$ is related to $\A$ by a handle slide, and we say that the diagram $(\A',\n)$ is the  result of \emph{wave surgery} (or just \emph{surgery}, for short) on $\alpha$ along $\omega$.

\begin{proposition}\cite{hot}\label{wavesurgery}
Let $(\A,\n)$ be a Heegaard diagram for $S^3$ that contains a wave based in $\A$, and let $(\A', \n)$ be the diagram obtained by wave surgery on $\alpha$.  Then $\io(\A',\n)<\io(\A,\n)$.
\end{proposition}

Thus, it is desirable for a nontrivial Heegaard diagram to contain a wave, as it allows for a straightforward simplification of the diagram.  Unfortunately, this is not the case for Heegaard diagrams of $S^3$ of arbitrary genera \cite{ochiai:Whitehead}.  However, the situation in genus two is much more manageable, as is shown by the following theorem of Homma, Ochiai, and Takahashi.

\begin{theorem}\cite{hot}\label{wave}
Every nontrivial genus two Heegaard diagram for $S^3$ contains a wave.
\end{theorem}



Wave surgery has a dual notion, known as \emph{banding}.  Let $\A' = \{\A_2,\A_3\}$ and suppose $\upsilon$ is an arc with one endpoint on each of $\A_2$ and $\A_3$ and whose interior is disjoint from $\A'$ (we call $\upsilon$ a \emph{band}).  Then one component $\A_1$ of $\pd \overline{\nu(\A' \cup \upsilon)}$ is not isotopic to $\A_2$ or $\A_3$, and we that say $\A_1$ is the result of banding $\A_2$ to $\A_3$.  For every wave $\omega$ such that wave surgery on $\A = \{\A_1,\A_2\}$ yields $\A' = \{\A_2,\A_3\}$, there is a dual band $\upsilon$ such that $\A_1$ is the result of banding $\A_2$ to $\A_3$; thus, banding and wave surgery are inverse processes.

We now state several general facts about the reduced Whitehead graphs corresponding to genus two Heegaard diagrams.  First, every such graph is isomorphic to one of the graphs shown in Figure \ref{fig:WhGraphs} \cite{ochiai:Whitehead}, which we henceforth refer to as Type I, Type II, and Type III, as labeled in the figure.  We will automatically assume that a graph of one of these types inherits the edge labelings shown.  Note that if all edge weights are nonzero, then the three families are mutually exclusive; however, if some edge weights are zero there may be graphs that fall into more than one family.  In addition, the existence of the involution $J$ implies that $w_{\n}(a_+) = w_{\n}(a_-)$ and $w_{\n}(b_+) = w_{\n}(b_-)$ and so we may denote these weights $w_{\n}(a)$ and $w_{\n}(b)$ without ambiguity.

\begin{figure}[h!]
\centering
\includegraphics[scale = .9]{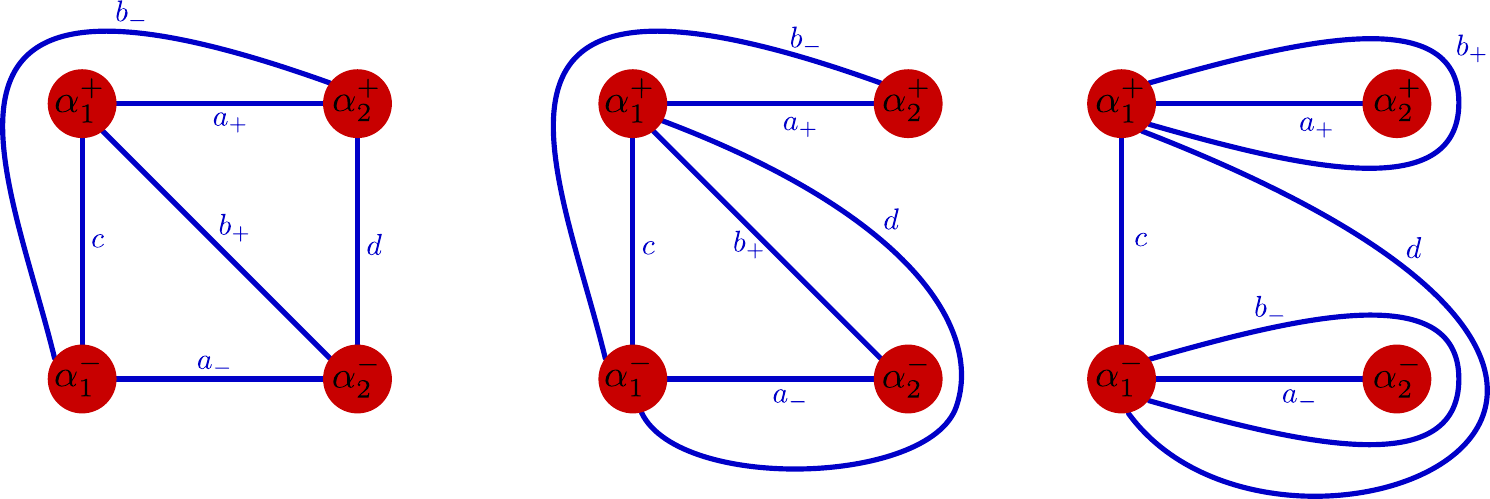}
\put(-345,-15){Type I}
\put(-205,-15){Type II}
\put(-78,-15){Type III}
\caption{The three possible reduced Whitehead graphs $G_\A(\n)$ corresponding to a genus two Heegaard splitting.}
\label{fig:WhGraphs}
\end{figure}

If a reduced Whitehead graph $G_{\A}(\n)$ contains exactly four edges and one 4-cycle, we will call $G_{\A}(\n)$ a \emph{square graph}.  Equivalently, a square graph is a Type I graph such that $w_{\n}(b) = 0$ and all other edges have nonzero weights.  If $G_{\A}(\n)$ is a graph of any type such that $w_{\n}(a),w_{\n}(c) \neq 0$ while $w_{\n}(b)=w_{\n}(d) = 0$, we say that $G_{\A}(\n)$ is a \emph{c-graph} (as its shape looks like the letter $c$). \\

\begin{figure}[h!]
\centering
\includegraphics[scale = 1.8]{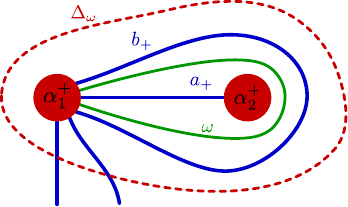}
\caption{An $\A$-wave $\omega$ in a Type III graph.  If the weight of $a$ or $b$ is nonzero, then $\io(\Delta_{\omega},\n) < \io(\A_1,\n)$.}
\label{fig:TypeIIIWave}
\end{figure}

The next lemma provides basic information about some of the weights in $G_{\A}(\n)$ when $(\A,\n)$ is a diagram for $S^3$.

\begin{lemma}[\cite{hot}, Lemma 9]\label{crossseam}
If $(\A,\n)$ is a nontrivial diagram for $S^3$, then $\Sigma_{\A}(\n)$ contains a cross-seam.
\end{lemma}

Observe that if $(\A,\n)$ contains a wave based in $\A$, then the graph $G_{\A}(\n)$ must be of Type III. A graph of Type I or II that admits a wave must also fall under the Type III category.  Figure \ref{fig:TypeIIIWave} shows the obvious wave in this case.  Equivalently, $(\A,\n)$ contains a wave based in $\A$ if and only if a vertex in $G_{\A}(\n)$ has valence one (recall that $G_{\A}(\n)$ is a reduced graph).  Note that Theorem \ref{wave} implies that if $(\A,\n)$ is a Heegaard diagram for $S^3$, then there is a wave based in \emph{either} $\A$ or $\n$.  However, if $G_{\A}(\n)$ is a c-graph, we can say something stronger:


\begin{lemma}\label{12gon}
If $(\A,\n)$ is a diagram for $S^3$ and $G_{\A}(\n)$ is a c-graph, then $(\A,\n)$ contains a wave based in $\A$ and a wave based in $\n$.
\begin{proof}
In this case, $\Sigma \setminus \nu(\A \cup \n)$ contains some number of rectangular regions and one 12--sided region, which we will call $U$, where the boundary of $U$ alternates between a set of six arcs in $\A$ and a set of six arcs in $\n$.  See Figure \ref{fig:cGraphWave}.  A choice of orientation $\A$ and $\n$ orients each arc in $\partial U$ clockwise or counter-clockwise.  Moreover, each $\A$--arc comes from either $\A_1$ or $\A_2$.  Therefore, there are four labeling options for each $\alpha$--arc.  However, since there are six  $\A$--arcs, at least two have a common labeling, and an arc $\omega_\A$ in $U$ connecting two such arcs is a wave based in $\A$.  See Figure \ref{fig:cGraphWave}(b).  A parallel argument shows that there is a wave $\omega_\n$ in $U$ based in $\n$.
\end{proof}
\end{lemma}

\begin{figure}[h!]
\centering
\includegraphics[scale = .7]{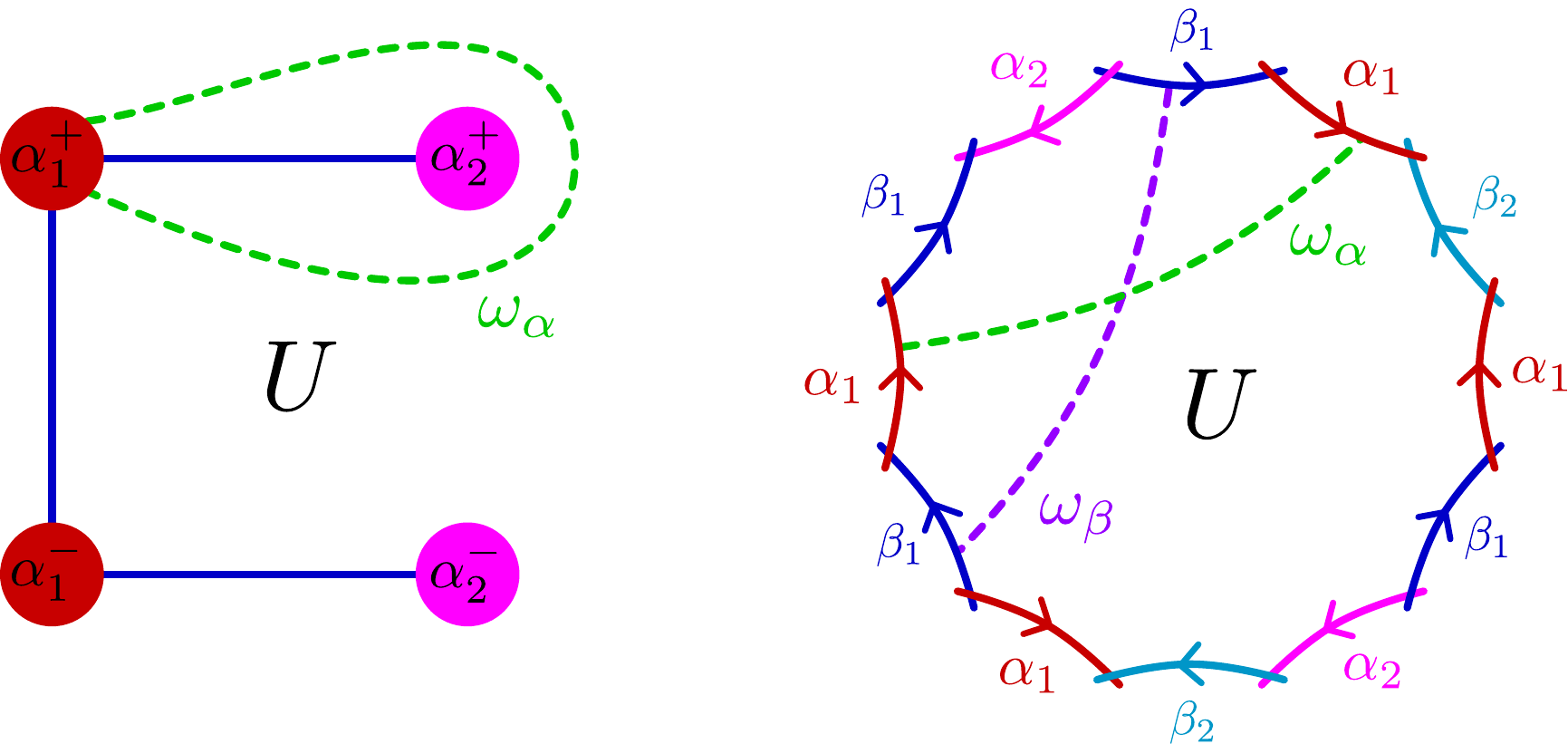}
\put(-300,-20){ (a)}
\put(-90,-20){ (b)}
\caption{When $G(\A,\n)$ is a c-graph, there are waves in $\A$ and $\n$.  The obvious wave $\omega_\A$ is shown in (a), and the rendering of the 12--gon $U$ shown in (b) purports the existence of $\omega_\n$.}
\label{fig:cGraphWave}
\end{figure}

In the case-by-case arguments that follow, we attempt to rule out the existence of certain waves for a Heegaard diagram, so we lay the groundwork for that strategy here.  If two oriented curves meet in two points, we say that these points are \emph{coherently oriented} if they have the same sign and \emph{oppositely oriented} if they have opposite sign.  Note that this definition is independent of the choice of orientations on the curves.  If $(\A,\n)$ is a Heegaard diagram and there is a choice of orientations on curves in $\A$ and $\n$ so that all points of $\A \cap \n$ are coherently oriented, we call $(\A,\n)$ a \emph{positive Heegaard diagram}. 

Suppose that $(\A,\n)$ is a Heegaard diagram and that $\n_1$ contains an arc $a$ that meets $\A_1$ only in its oppositely oriented endpoints and meets $\A_2$ in one point in its interior.  We call an arc $a$ satisfying these conditions a \emph{wave-busting arc}.  Similarly, if $\n$ contains two arcs $c$ and $d$ such that $c$ meets $\A_1$ only in its coherently oriented endpoints and avoids $\A_2$, while $d$ meets $\A_2$ only in its coherently oriented endpoints and avoids $\A_1$, we say that the pair $(c,d)$ is a \emph{wave-busting pair}.  The next lemma makes clear our use of the term ``wave-busting."
\begin{lemma}\label{sat}
If $(\A,\n)$ is a nontrivial Heegaard diagram such that $\n$ contains a wave-busting arc $a$ or a wave-busting pair $(c,d)$, then $(\A,\n)$ does not contain a wave based in $\A$.
\begin{proof}
A wave-busting arc $a$ contributes two cross-seams to $G_{\A}(\n)$, one from $\A_1^{\pm}$ to $\A_2^+$ and one from $\A_1^{\pm}$ to $\A_2^-$.  It follows by the symmetry of $G_{\A}(\n)$ that each vertex has valence at least two, and thus $G_{\A}(\n)$ cannot be a graph of Type III.  On the other hand, a wave-busting pair contributes seams connecting $\A_1^+$ to $\A_1^-$ and $\A_2^+$ to $\A_2^-$, again implying that $G_{\A}(\n)$ is not a Type III graph.
\end{proof}
\end{lemma}


To conclude this section, we discuss how homeomorphisms of $\Sigma_{\A}$ can be used to reduce the number of cases in the arguments that follow.  The five homeomorphisms most relevant to our proof are shown in Figure \ref{fig:Involutions}.   The mapping class group of a four-punctured sphere is generated by $PGL_2(\Z)$ and two involutions  \cite{farb-margalit:primer}, the first  of which is the hyperelliptic involution $J$ discussed above, and the second of which we will call $J'$ and is depicted in Figure \ref{fig:Involutions}(a).  Note that $J$ and $J'$ preserve the slopes of arcs in $\Sigma_\A$, so to determine the behavior of slopes of arcs with respect to a given homeomorphism $\sigma$, we need only specify an element of $PGL_2(\Z)$, which we will call $M_{\sigma}$.

\begin{figure}[h!]
\centering
\includegraphics[scale = .7]{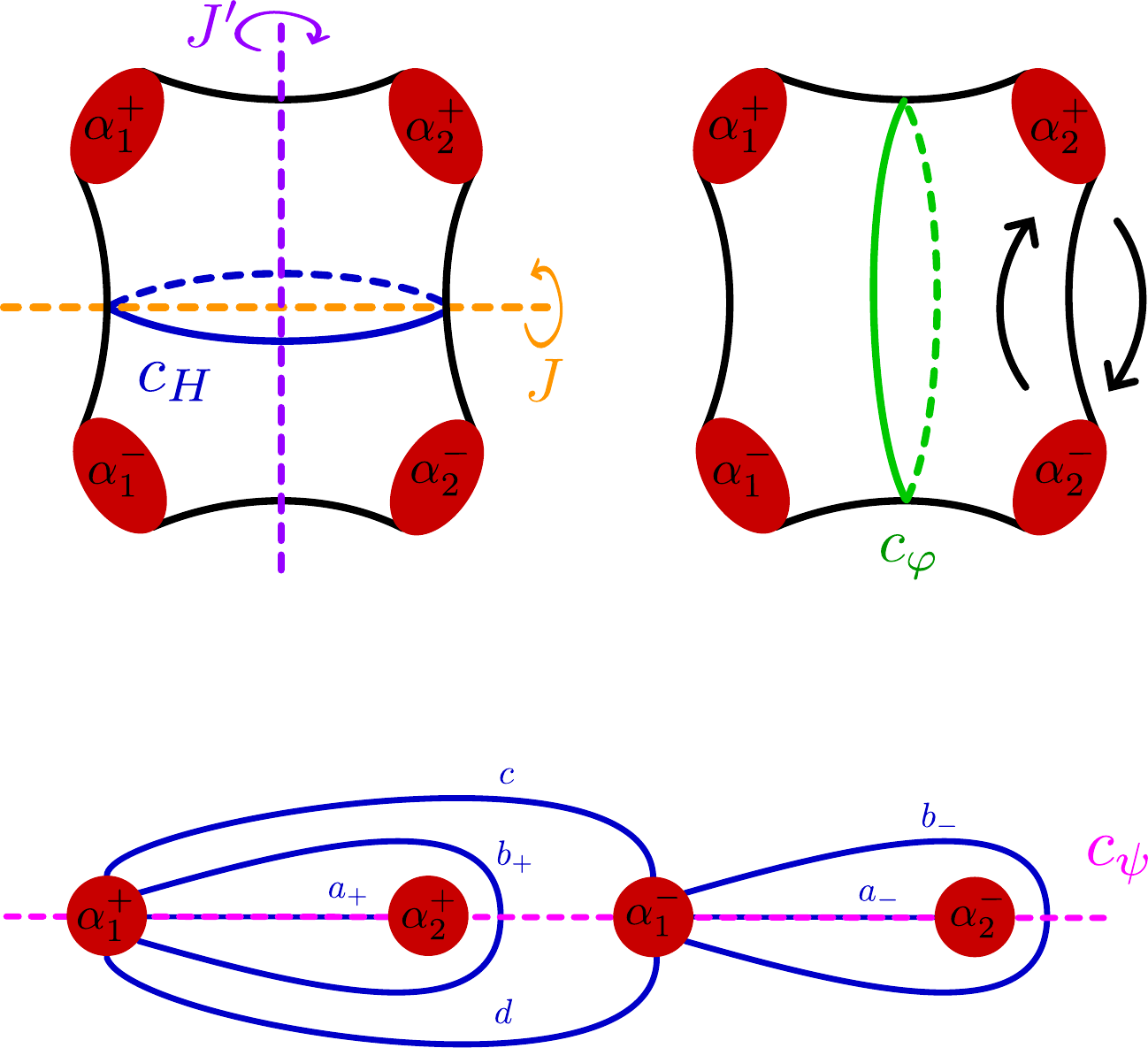}
\put(-205,90){(a)}
\put(-62,90){(b)}
\put(-140,-25){(c)}
\caption{The five homeomorphisms of the four-punctured sphere $\Sigma_\A$ that will be useful throughout the paper.  First, $J$ and $J'$ are rotations through $\pi$ radians about the axes shown in (a), and  $H$ is reflection across the plane that intersects the sphere in $c_H$. The homeomorphism $\varphi$ is a half-twist fixing $c_\varphi$ and exchanging $\A_2^+$ and $\A_2^-$, as shown in (b), and (c) shows  the action of $\psi$ as reflection across the plane intersecting the page in the line $c_\psi$.}
\label{fig:Involutions}
\end{figure}

Let $c_{\varphi} = \Delta_a$ for a seam $a$ of slope $\frac{1}{0}$, and let the homeomorphism $\varphi$ be given by fixing $\A_1^{\pm}$ and $c_{\varphi}$ and performing a half-twist on $\A_2^{\pm}$ which exchanges these two punctures, as in Figure \ref{fig:Involutions}(b).  The matrices for $\varphi$ and $\varphi^{-1}$ are given by
\[ M_{\varphi} = \begin{pmatrix}
1 & 1 \\ 0& 1 \end{pmatrix} \quad \text{ and } \quad
M_{\varphi^{-1}} = \begin{pmatrix}
1 & -1 \\ 0 & 1 \end{pmatrix}.\]

The remaining two homeomorphisms  reverse the orientation of $\Sigma_\A$.  Let $H$ be the horizontal reflection fixing the curve $c_H$ shown in Figure \ref{fig:Involutions}(a), and let $\psi$ be a reflection of $\Sigma_\A$ fixing the four arcs denoted $c_{\psi}$ in Figure \ref{fig:Involutions}(c).  Note that $\psi$ preserves a Type III graph; edges $a_{\pm}$ and $b_{\pm}$ are fixed, while edges $c$ and $d$ are exchanged.  The matrices for $H$ and $\psi$ are given by

\[ M_{H} = \begin{pmatrix}
-1 & 0 \\ 0& 1 \end{pmatrix} \quad \text{ and } \quad
M_{\psi} = \begin{pmatrix}
-1 & 0 \\ 2 & 1 \end{pmatrix}\]

\section{The complexity of a $(2,0)$--trisection diagram}\label{sec:start}

To classify $(2,0)$--trisections, we will show that a trisection diagram which minimizes a prescribed complexity over all diagrams corresponding to a fixed $(2,0)$--trisection must be a standard diagram. 
Define the \emph{complexity} $\calc(\A,\n,\g)$ of a trisection diagram $(\A,\n,\g)$ to be the triple given by
\[ \calc(\A,\n,\g) = (\io(\A_2,\g),\io(\A_1,\g),\io(\n,\g)).\]
For a $(2,0)$--trisection $X = X_1 \cup X_2 \cup X_3$, 
we call a trisection diagram \emph{minimal} if it minimizes the complexity $\mathcal C$ (with the dictionary ordering) among diagrams satisfying $\io(\A,\n)=2$.
A plurality of this paper (Sections \ref{sec:start}--\ref{sec:hard_case}) is devoted to proving the following proposition, which we will use later to classify $(2,0)$--trisections in Section \ref{sec:main}.

\begin{proposition}\label{claim1}
If $(\A,\n,\g)$ is minimal, then $\io(\A,\g) = 2$.
\end{proposition}

To prove the proposition, we will suppose by way of contradiction that $(\A,\n,\g)$ is minimal, but $(\A,\g)$ is not trivial, so that $\io(\A,\g) > 2$.  We may also suppose that $(\n,\g)$ is not trivial; otherwise $\calc(\n,\g,\A) < \calc(\A,\n,\g)$ and $(\A,\n,\g)$ is not minimal.  One case is quite easy to rule out.

\begin{lemma}\label{gwave}
If $G_{\g}(\A)$ is a Type III graph, then $(\A,\n,\g)$ is not minimal.
\begin{proof}
In this case, $\Sigma_{\g}(\A)$ contains a wave, and we let $\g'$ be the result of surgery on $\g$ along this wave, so that $\g$ and $\g'$ are related by a handle slide. With Figure \ref{fig:TypeIIIWave} in mind, we compute
\begin{eqnarray}
\io(\A_2,\g) &=& 2w_{\A_2}(a) + 2w_{\A_2}(b) + w_{\A_2}(c) + w_{\A_2}(d) \label{eq1} \\
\io(\A_2,\g') &=& w_{\A_2}(a) + w_{\A_2}(c) + w_{\A_2}(d). \label{eq2}
\end{eqnarray}

Thus, if $w_{\A_2}(a)$ or $w_{\A_2}(b)$ is nonzero in $G_{\g}(\A)$, then $\io(\A_2,\g') < \io(\A_2,\g)$ and $\calc(\A,\n,\g') < \calc(\A,\n,\g)$.  Otherwise, we have $\io(\A_2,\g) = \io(\A_2,\g')$.  However, since $\g'$ is surgery on $\g$ along a wave, $\io(\A,\g') < \io(\A,\g)$, which implies that $\io(\A_1,\g') < \io(\A_1,\g)$, and once again, $\calc(\A,\n,\g') < \calc(\A,\n,\g)$.  In either case, $(\A,\n,\g)$ is not minimal.
\end{proof}
\end{lemma}

Using Equations (\ref{eq1}) and (\ref{eq2}), we also have the next lemma.
\begin{lemma}\label{g2wave}
If $G_{\g}(\A_2)$ is a Type III graph that contains a cross-seam, then $(\A,\n,\g)$ is not minimal.
\begin{proof}
In this case, we have $w_{\A_2}(a) \neq 0$, and so $(\A,\n,\g)$ is not minimal by the proof of Lemma \ref{gwave}.
\end{proof}
\end{lemma}




By Theorem \ref{wave} and the assumption that $(\A,\g)$ is not trivial, we have that $(\A,\gamma)$ contains a wave $\omega$.  If $\omega$ is based in $\gamma$, then $G_{\g}(\A)$ is a Type III graph, so $(\A,\n,\g)$ is not minimal by Lemma \ref{gwave}.  Thus, we may assume that $\omega$ is based in $\A$, and $G_{\A}(\gamma)$ is a Type III graph.  Since $(\A,\n,\g)$ is minimal, we have that $\io(\A_2,\g) \leq \io(\A_1,\g)$; otherwise we could change the labeling of $\A = \{\A_1,\A_2\}$ to reduce complexity. Therefore, we can assume that $\pd \omega \subset \A_1$.

We parameterize arcs in $G_{\A}(\gamma)$ by assigning the seams connecting $\A_1^+$ to $\A_1^-$ the slopes $\frac{1}{0}$ and $-\frac{1}{2}$, and the seams connecting $\A_1^{\pm}$ to $\A_2^{\pm}$ the slope $\frac{0}{1}$.  Since $\io(\A, \n) = 2$, each intersection $\n_i \cap \Sigma_\A$ is a single seam connecting $\A_i^+$ and $\A_i^-$ for $i=1,2$, and these two seams have the same slope.  Let $\frac{m}{n}$ denote the slope of these seams, noting that $m$ must be odd and $n$ must be even, and we suppose that $m > 0$.  We also assume that $\io(\A_i,\n_j) = \delta_{ij}$.  In the case that $m = 1$, there is a curve $\n_3$ such that $\n_3 \cap \Sigma_\A$ consists of two cross-seams of slope $\frac{0}{1}$ and $\n_3 \cap \n = \emp$.  See Figure \ref{fig:Case-m=1}.  We call the curve $\n_3$ the \emph{0--replacement} for $\n$.

\begin{lemma}\label{mone}
If $m = 1$, then $(\A,\n,\g)$ is not minimal.
\begin{proof}
Let $\A' = \{\A_3,\A_2\}$ be the result of surgery on $\A$ along the wave $\omega$, let $\n_3$ be the 0--replacement for $\n$, and let $\n' = \{\n_1,\n_3\}$.  Then $\A$ and $\n$ are related to $\A'$ and $\n'$ by handle slides, and $\io(\A',\n') = 2$.  However, $\io(\A_3,\g) < \io(\A_1,\g)$ and thus $\calc(\A',\n',\g) < \calc(\A,\n,\g)$.  See Figure \ref{fig:Case-m=1}.
\end{proof}
\end{lemma}

\begin{figure}[h!]
\centering
\includegraphics[scale = 2.7]{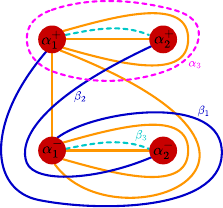}
\caption{The case in which $m=1$, so that arcs of $\beta_i$ have slope $\frac{1}{n}$ (shown with $n=2$).  Replacing $\A_1$ with $\A_3$ and $\n_2$ with the 0--replacement $\n_3$ lowers the complexity of the diagram.}
\label{fig:Case-m=1}
\end{figure}

We suppose for the remainder of this section that $m \geq 3$.  
Our plan of attack is as follows: We show that, in most cases, $(\n,\g)$ contains a wave based in $\g$ and that $(\A,\n,\g)$ is not minimal by Lemma \ref{nogwave}.  In the exceptional case, we will prove in Section \ref{sec:hard_case} that $(\n,\g)$ is not a Heegaard diagram for $S^3$, a contradiction.

Before we proceed, we require a technical lemma, which we state in greatest generality for use in later sections.

\begin{lemma}\label{missedcase}
Suppose that $(\delta,\eps)$ is a Heegaard diagram such that $\io(\delta_1,\eps_2) = 0$ and $G_{\delta}(\eps)$ is a square graph.  Then $(\delta,\eps)$ is not a Heegaard diagram for $S^3$.
\begin{proof}
We label the edges of the Type I graph $G_{\delta}(\eps)$ as in Figure \ref{fig:WhGraphs}, noting that $\eps_2$ contributes to only one edge, $d$, in $G_{\delta}(\eps)$.  Let $k_1 = w_{\eps_2}(d)$.  Since $G_{\delta}(\eps)$ is a square graph, we may orient $\delta$ and $\eps$ so that all intersection points have the same sign; thus, $(\delta,\eps)$ is a positive Heegaard diagram.  It follows that $\io(\delta_i,\eps_j) = \delta_i \cdot \eps_j$.  Since all edges of $G_{\delta}(\eps)$ have nonzero weights, $\eps_1$ intersects $\delta_2$ in $k_2 = w_{\eps}(a) >0$ points and $\delta_1$ in $k_3 = w_{\eps}(a) + w_{\eps}(c)  > 1$ points.  This implies that the intersection matrix $M(\delta,\eps)$ is
\[ M(\delta,\eps) = \begin{pmatrix}
k_3 & 0 \\
k_2 & k_1
\end{pmatrix},\]
where $|\text{det}(M(\delta,\eps))| = k_1k_3 > 1$.  Thus, $(\delta,\eps)$ is not a Heegaard diagram for $S^3$ by Lemma \ref{determ}.
\end{proof}
\end{lemma}

Now, we return to $(\A,\n,\g)$.  If $\n^*$ is a seam in $\Sigma_{\A}(\n)$ connecting $\A_i^+$ to $\A_i^-$, then the slope of $\n^*$ is $\frac{m}{n}$.  A seam $\g^*$ in $\Sigma_\A(\g)$ connecting $\A_1^+$ to $\A_1-$ has slope $\frac{1}{0}$ or $-\frac{1}{2}$, while a cross-seam has slope $\frac{0}{1}$.  It follows from Lemma \ref{inter} that $\n^*$ intersects $\g^*$ essentially.  The same is true if $\g^*$ is a wave arc.  Recall that we have assumed that $(\n,\g)$ is not trivial, so by Theorem \ref{wave}, $(\n,\g)$ contains a wave.

\begin{lemma}\label{nogwave}
If $(\n,\g)$ contains a wave based in $\g$, then  $(\A,\n,\g)$ is not minimal.
\begin{proof}
If $(\n,\g)$ contains a wave based in $\g$, then $G_{\g}(\n)$ is a Type III graph, and thus each vertex $\g_2^{\pm}$ is incident to only one edge $a_{\pm}$ in $G_{\g}(\n)$.  Isotope $\A$ so that any inessential intersections of $\A$ and $\n$ in $\Sigma_\g$ based in $\g_2^{\pm}$ (there are at most two of them) are based in $\g_2^-$.  We note that $\g_2^+$ does not contain two consecutive intersections with $\A$.  Otherwise, the subarc connecting such intersections corresponds to an arc $\g^*$ in $\Sigma_\A(\g)$ which avoids $\n$, contradicting that every such arc $\g^*$ intersects an arc of $\n$ essentially in $\Sigma_\A$.

It follows that in $\Sigma_{\g}(\A,\n)$, all $\A$-arcs meeting $\g_2^+$ with possibly one exception are situated between $\n$-arcs in the edge $a_+$.  In other words, in the reduced graph $G_{\g}(\A)$, the vertex $\g_2^+$ is incident to at most two edges, one of which has the same slope as $a_+$ in $G_{\g}(\n)$, and so we will call this edge $a_+$, and the other of which we will call $e$.  See Figure \ref{fig:LocalC2+}.  Note that $w_{\A}(e) = 1$.  We split the remainder of the argument into four cases, based on possible values of $w_{\A_2}(a)$ and $w_{\A_2}(e)$. \\

\begin{figure}[h!]
\centering
\includegraphics[scale = 1.5]{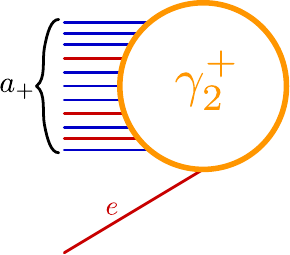}
\caption{A local picture of $\Sigma_\g$ near $\g_2^+$.  All arcs of $\Sigma_\g(\n)$ and all but one arc of $\Sigma_\g(\A)$ meeting $\g_2^+$ are in the edge $a_+$, and the one exceptional arc of $\A$ is in the edge $e$.}
\label{fig:LocalC2+}
\end{figure}

\textbf{Case 1:} $w_{\A_2}(a) \neq 0$ and $w_{\A_2}(e) = 0$.  In this case, the vertex $\g_2^+$ has valence one in $G_{\g}(\A_2)$ and $G_{\g}(\A_2)$ contains the cross-seam $a_+$.  By Lemma \ref{g2wave}, the diagram $(\A,\n,\g)$ is not minimal. \\

\textbf{Case 2:} $w_{\A_2}(a) = 0$ and $w_{\A_2}(e) \neq 0$.  If $e$ is a cross-seam, then by the same argument as in Case 1, $(\A,\n,\g)$ is not minimal.  Otherwise, $e$ connects $\g_2^+$ to $\g_2^-$.  By the argument above, $\g_2^{\pm}$ is incident to exactly two edges $a^{\pm}$ and $e$ in $G_{\g}(\A)$.  If $\g_1^{\pm}$ has valence one in $G_{\g}(\A)$, then $G_{\g}(\A)$ is a c-graph and there is a wave for $(\A,\g)$ based in $\g$ by Lemma \ref{12gon} and thus $(\A,\n,\g)$ is not minimal by Lemma \ref{gwave}.  If $\g_1^+$ has valence greater than one, there is another edge $c$ incident to $\g_1^+$ in $G_{\g}(\A)$, and $c$ must connect $\g_1^+$ to $\g_1^-$.  It follows that $G_{\g}(\A)$ is a Type I graph with only one cross-seam; equivalently, $G_{\g}(\A)$ is a square graph, and the assumption $w_{\A_2}(a) = 0$ implies that $\io(\A_2,\g_1) = 0$.  However, in this case Lemma \ref{missedcase} implies that $(\A,\g)$ is not a Heegaard diagram for $S^3$, a contradiction. \\

\textbf{Case 3:} $w_{\A_2}(a) \neq 0$ and $w_{\A_2}(e) \neq 0$.  Here, $w_{\A_2}(e) = 1$ and thus $w_{\A_1}(e) = 0$.  Let $\g_3 = \Delta_a$ and $\g' = \{\g_2,\g_3\}$, so that $\g$ is related to $\g'$ by a handle slide.  There are two subcases to consider; refer to Figure \ref{fig:a2graphs}.  Suppose first that $e$ is a cross-seam, so that there is another cross-seam $e'$ which meets $\g_2^-$ by the symmetry of $G_{\g}(\A)$.  Since the valence of $\g_2^+$ in $G_{\g}(\A)$ is two, we have $G_{\g}(\A)$ is a Type II graph.  Let $c$ and $d$ denote the two distinct edges connecting $\g_1^+$ to $\g_1^-$.  We compute
\begin{eqnarray*}
\io(\A_2,\gamma_1) &=& w_{\A_2}(a) + w_{\A_2}(c) + w_{\A_2}(d) + 1\\
\io(\A_2,\gamma_3) &=& w_{\A_2}(c) + w_{\A_2}(d) + 2.
\end{eqnarray*}
It follows from $w_{\A_2}(a) \neq 0$ that $\io(\A_2,\g') \leq \io(\A_2,\g)$.  If this inequality is strict, then $(\A,\n,\g)$ is not minimal.  Otherwise, we assume $\io(\A_2,\g') = \io(\A_2,\g)$.  Since $w_{\A_1}(e) = 0$, $G_{\g}(\A_1)$ is a Type III graph.  Thus, if $w_{\A_1}(a) \neq 0$, then $\io(\A_1,\g') < \io(\A_1,\g)$ so that $\calc(\A,\n,\g') < \calc(\A,\n,\g)$.  Otherwise, $w_{\A_1}(a) = 0$ and thus $\io(\A_1,\g') = \io(\A_1,\g)$.  Now, since $\g'$ is the result of surgery on a wave for $(\n,\g)$, we have $\io(\n,\g') < \io(\n,\g)$ in addition to $\io(\A_2,\g') = \io(\A_2,\g)$ and $\io(\A_1,\g') = \io(\A_1,\g)$; hence $\calc(\A,\n,\g') < \calc(\A,\n,\g)$.

\begin{figure}[h!]
\centering
\includegraphics[scale = .5]{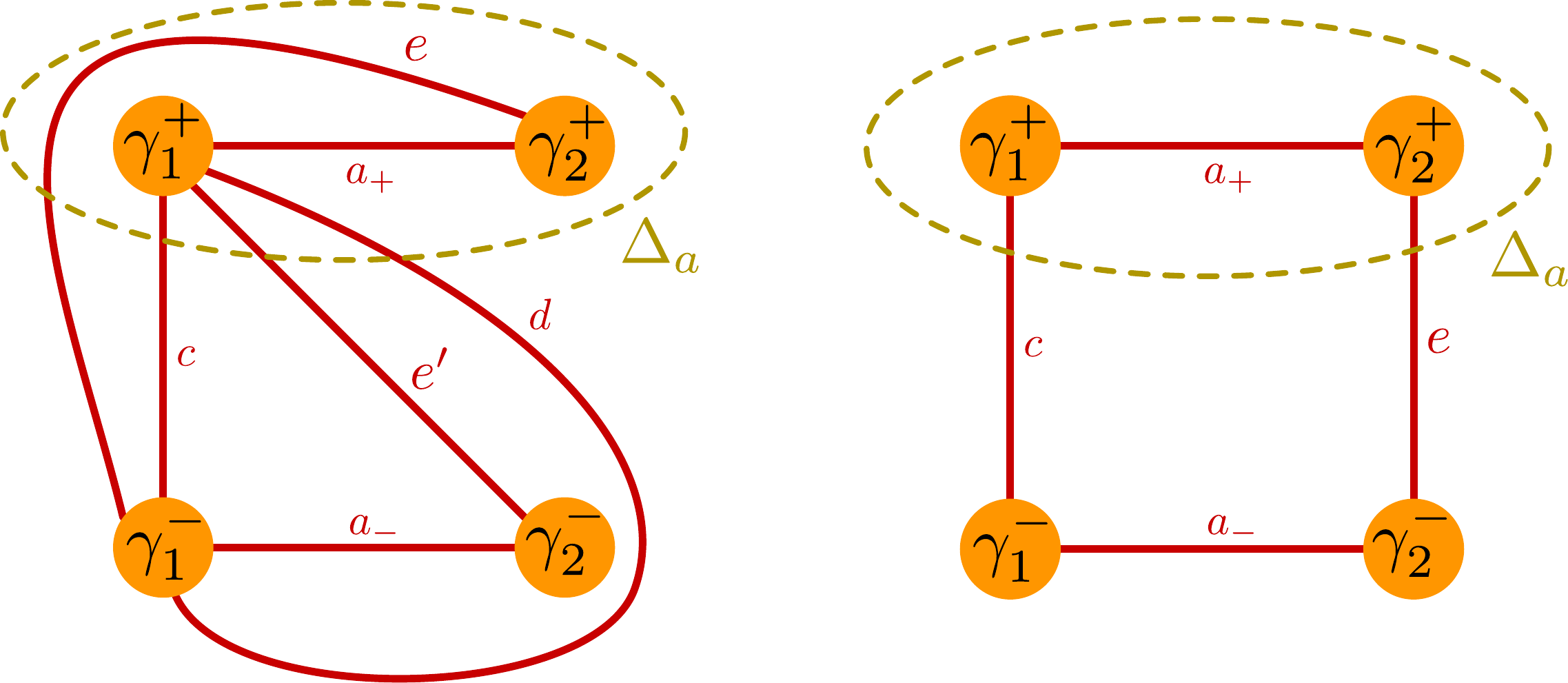}
\put(-260,-20){\large (a)}
\put(-87,-20){\large (b)}
\caption{The two subcases of Case 3 considered in the proof of Lemma \ref{nogwave}: (a) $G_\g(\A_2)$ is a Type II graph, and (b) $G_\g(\A_2)$ is a square graph.  In each case, replacing $\g_1$ by $\g_3=\Delta_a$ reduces complexity.}
\label{fig:a2graphs}
\end{figure}

In the second case, we suppose that $e$ connects $\g_2^+$ to $\g_2^-$.  If the valence of $\g_1^{\pm}$ in $G_{\g}(\A_2)$ is one, then $(\A,\n,\g)$ is not minimal by Lemma \ref{g2wave}.  Otherwise, the valence of $\g_1^{\pm}$ is greater than one in $G_{\g}(\A)$, and since the valence of $\g_2^+$ is two, $G_{\g}(\A)$ is a Type I graph with one cross-seam; equivalently, $G_{\g}(\A)$ is  a square graph.  Here we let $c$ denote the single edge connecting $\g_1^+$ and $\g_1^-$.  As above, we compute
\begin{eqnarray*}
\io(\A_2,\gamma_1) &=& w_{\A_2}(a) + w_{\A_2}(c) \\
\io(\A_2,\gamma_3) &=& w_{\A_2}(c) + 1.
\end{eqnarray*}
If $w_{\A_2}(a) > 1$, then $\io(\A_2,\g') < \io(\_2A,\g)$ and $(\A,\n,\g)$ is not minimal.  On the other hand, if $w_{\A_2}(a) = 1$, then $\io(\A_2,\g) = \io(\A_2,\g')$, while $G_{\g}(\A_1)$ is a Type III graph, which implies $\calc(\A,\n,\g') < \calc(\A,\n,\g)$ as above. \\

\textbf{Case 4}: $w_{\A_2}(a) = w_{\A_2}(e) = 0$.  If $w_{\A}(a) = w_{\A}(e) = 0$, then Lemma \ref{crossseam} implies that $(\A,\g)$ is trivial.  If one of $w_{\A}(a)$ or $w_{\A}(e)$ is zero, then $G_{\g}(\A)$ is a Type III graph and $(\A,\n,\g)$ is not minimal by Lemma \ref{gwave}.  Otherwise, $w_{\A_1}(e) = 1$ and $w_{\A_1}(a) \neq 0$.  Suppose first that $e$ is a cross-seam, so that by the symmetry of $G_{\g}(\A_1)$, there is another weight one cross-seam $e'$ which meets $\g_2^-$.  Since the valence of $\g_2^{\pm}$ is two and $G_{\g}(\A)$ contains two distinct cross-seams, the graph $G_{\g}(\A)$ is a Type II graph (as in Figure \ref{fig:a2graphs}(a)).  Let $c$ and $d$ denote the edges connecting $\g_1^+$ to $\g_1^-$, let $\g_3 = \Delta_a$, and let $\g' = \{\g_2,\g_3\}$ so that $\g'$ is related to $\g$ by a handle slide.  Then $\io(\A_2,\g_3) = \io(\A_2,\g_1) = w_{\A_2}(c) + w_{\A_2}(d)$, and, using $w_{\A_1}(e) = w_{\A_1}(e') = 1$, we compute
\begin{eqnarray*}
\io(\A_1,\g_1) &=& w_{\A_1}(a) + w_{\A_1}(c) + w_{\A_1}(d) +1 \\
\io(\A_1,\g_3) &=& w_{\A_1}(c) + w_{\A_1}(d) + 2.
\end{eqnarray*}
Thus, if $w_{\A_1}(a) > 1$, then $\io(\A_1,\g') < \io(\A_1,\g)$ so that $\calc(\A,\n,\g') < \calc(\A,\n,\g)$.  Otherwise, $w_{\A_1}(a) = 1$ and thus $\io(\A_1,\g') = \io(\A_1,\g)$.  As above, since $\g'$ is the result of surgery on a wave for $(\n,\g)$, we have $\io(\n,\g') < \io(\n,\g)$; hence $\calc(\A,\n,\g') < \calc(\A,\n,\g)$.


In the second subcase, suppose that $e$ is a seam connecting $\g_2^+$ to $\g_2^-$.  Since an $\A_2$--arc contributes an edge, call it $c$, connecting $\g_1^+$ to $\g_1^-$ in $G_{\g}(\A)$, it follows that $G_{\g}(\A)$ is a Type I graph with only one cross-seam (see Figure \ref{fig:a2graphs}(b)); that is, $G_{\g}(\A)$ is a square graph.  Moreover, $w_{\A_2}(a) = w_{\A_2}(e) = 0$ implies $\io(\A_2,\g_2) = 0$, so by Lemma \ref{missedcase}, we have $(\A,\g)$ is not a Heegaard diagram for $S^3$, a contradiction.
\end{proof}
\end{lemma}

In the two lemmas that follow, we return our attention to $G_{\A}(\n,\g)$.  These two lemmas, together with our previous work in this section, prove Proposition \ref{claim1} in all but one exceptional case.  This case is significantly more complicated and requires the arguments in Sections \ref{sec:digress} and \ref{sec:hard_case} to settle.

\begin{lemma}\label{bnotzero}
If $w_{\g}(b) \neq 0$ in $G_{\A}(\g)$, then $(\n,\g)$ does not contain a wave based in $\n$.
\begin{proof}
Observe that there is an arc $\n^*$ of $\n_2$ which connects $\A_2^+$ to a point $p$ contained in an arc $\gamma_0$ in the edge $b_+$ of $G_\A(\g)$. See Figure \ref{fig:wb=0}(a). Since $m \geq 3$, we have that $\n_1$ intersects arcs in the edge $a$ essentially, which implies that $\n_1$ meets $\g_0$ in a nonempty set of pairs of oppositely oriented points.  It follows that $\g_0$ contains a wave-busting arc $\g^*$ containing the point $p$; hence $G_{\n}(\g)$ does not contain a wave by Lemma \ref{sat}.  
\end{proof}
\end{lemma}

\begin{figure}[h!]
\centering
\includegraphics[scale = 1.0]{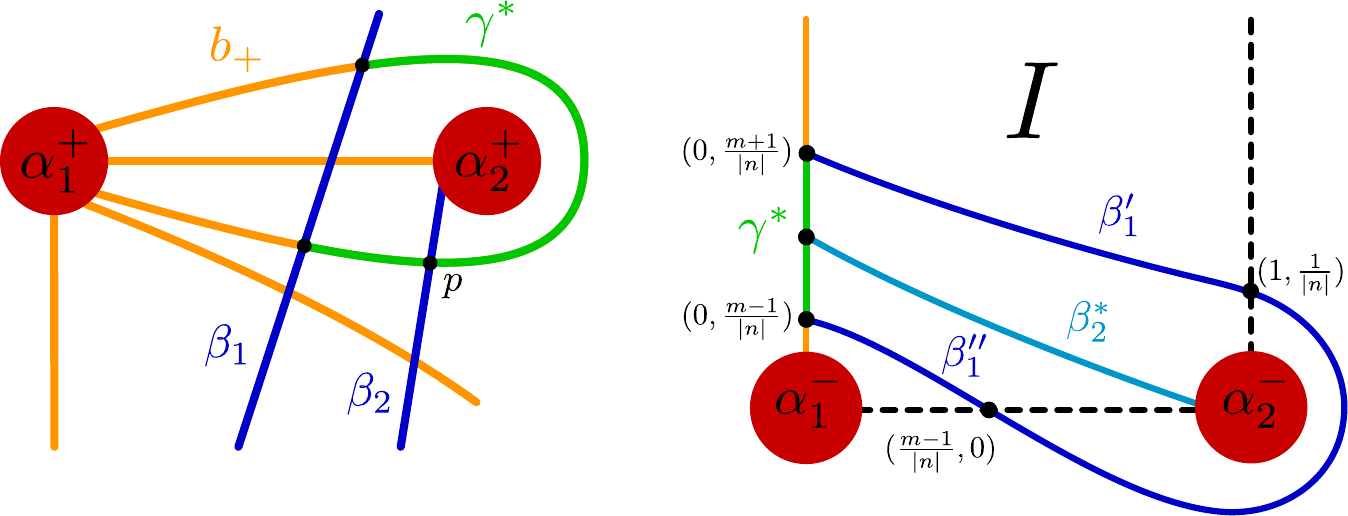}
\put(-320,-15){\large (a)}
\put(-100,-15){\large (b)}
\caption{Wave-busting arcs in $\g^*$ in the case (a) that $w_{\g}(b) \neq 0$ and (b) that $w_\g(b)=0$, but the slope $\frac{m}{n}\not=-\frac{n\pm1}{n}$.}
\label{fig:wb=0}
\end{figure}

We now assume that $w_\g(b)=0$.  Recall that the homeomorphism $\psi$ described in Section \ref{sec:prelims} maps $G_{\A}(\g)$ onto itself.

\begin{lemma}\label{nothard}
If $\frac{m}{n} \neq -\frac{n \pm 1}{n}$, then $(\n,\g)$ does not contain a wave based in $\n$ or $(\A,\n,\g)$ is not minimal.
\begin{proof}
First, we note that $w_{\g}(a)$, $w_{\g}(c)$, and $w_{\g}(d)$ are not zero; otherwise, $G_{\g}(\A)$ is a c-graph and contains a wave by Lemma \ref{12gon}, so $(\A,\n,\g)$ is not minimal by Lemma \ref{gwave}.  To prove the lemma, we will view $\Sigma_{\A}$ as two unit squares $I$ and $I'$ glued together along their boundaries with a neighborhood of their vertices removed.  We describe points in $I$ with Cartesian coordinates, where $\A_1^- = (0,0)$.  First, suppose that $\frac{m}{n} < 0$ and $1 < m < |n| - 1$.  Let $\n_2^*$ denote the arc of $\n_2 \cap I$ which meets the vertex of $I$ at the coordinates $(1,0)$.  Since the slope of $\n_2^*$ is $-\frac{m}{|n|}$, the coordinates of the other endpoint of $\n_2^*$ are $(0,\frac{m}{|n|})$.  In $I$, there are two arcs of $\n_1$, call them $\n_1'$ and $\n_1''$, which are adjacent to $\n_2^*$.  The endpoints of $\n_1'$ and $\n_1''$ are given by
\begin{eqnarray*}
\text{endpoints of } \n_1' &=& \left(0,\frac{m+1}{|n|}\right) \text{ and } \left(1, \frac{1}{|n|}\right) \\
\text{endpoints of } \n_1'' &=& \left(0,\frac{m-1}{|n|} \right) \text{ and } \left(\frac{m-1}{m}, 0\right)
\end{eqnarray*}
See Figure \ref{fig:wb=0}(b).  The assumption that $1 < m < |n|-1$ ensures that none of these four endpoints is a vertex of $I$.  In addition, we note that in $I'$, the endpoints $(1,\frac{1}{|n|})$ and $(\frac{m-1}{m},0)$ are connected by an arc in $\n_1$.  Thus, an arc in $\Sigma_\A(\g)$ which has slope $\frac{1}{0}$ contains a subarc $\g^*$ connecting oppositely oriented points of intersection with $\n_1$, $(0,\frac{m+1}{|n|})$ and $(0,\frac{m-1}{|n|})$.  Further, $\g^*$ meets $\n_2$ precisely once in the point $(0,\frac{m}{|n|})$.  Therefore, $\g^*$ is a wave-busting arc and $(\n,\g)$ does not contain a wave by Lemma \ref{sat}.

For the other cases, we will show that by applying the homeomorphism $\psi$ to $\Sigma_{\A}$, we can reduce the argument to the one above.  Since $\psi$ preserves $G_{\A}(\g)$, we need only determine the action of $\psi$ on the slope of $\frac{m}{n}$ of arcs in $\n$.  Suppose that either $\frac{m}{n} > 0$ or $\frac{m}{n} < 0$ and $m > |n| + 1$.  We have
 \[ \psi \begin{pmatrix} m \\ n \end{pmatrix} = \begin{pmatrix} -m \\ 2m + n \end{pmatrix}.\]
If $\frac{m}{n} > 0$, then certainly $-\frac{m}{2m+n} < 0$ and $m < 2m+n -1$.  Hence, the above argument shows that there exists a wave-busting arc $\g^* \subset \g$, so $(\n,\g)$ does not contain a wave based in $\g$.  Now, suppose that $n < 0$ and $m > |n| + 1$ and let $k = 2m + n = 2m -|n|$, so that $k > m +1$ and thus $m < k -1$.  Once again, we have that $-\frac{m}{2m+n} = -\frac{m}{k}$ is of the form handled by the above argument, and we conclude that $(\n,\g)$ does not contain a wave based in $\n$.
\end{proof}
\end{lemma}

The remaining case to consider in the proof of Proposition \ref{claim1} occurs when $\frac{m}{n} = - \frac{n \pm 1}{n}$, and we postpone our examination of this case until Section \ref{sec:hard_case}.

\section{A brief digression into square graphs}\label{sec:digress}

Recall some of the assumptions and conclusions we have made up to this point.  First, $(\A,\n,\g)$ is minimal, and so neither $G_{\g}(\A)$ nor $G_\g(\n)$ is a Type III graph.  Second, the weight $w_{\gamma}(b)=0$ in $G_{\A}(\g)$, which is a Type III graph.  This implies that $\A$ and $\g$ can be oriented so that all points of intersection occur are coherently oriented; in other words, $(\A,\g)$ is a \emph{positive Heegaard diagram}.  In this case, we consider $G_{\g}(\A)$.

\begin{lemma}\label{lemma:4cycle}
The graph $G_{\g}(\A)$ is a square graph.
\begin{proof}
Since $G_{\g}(\A)$ does not contain a wave by assumption, we know that it falls under Type I or Type II.  Refer to Figure \ref{fig:WhGraphs}.  In addition, $(\A,\g)$ is a positive Heegaard diagram, so all cross-seams in $\Sigma_{\g}(\A)$ originating from $\g_1^+$ and terminating in $\g_2^{\pm}$  must terminate in the same vertex, say $\g_2^+$.  Thus, $G_{\g}(\A)$ contains an edge connecting $\g_1^+$ to $\g_2^+$, its symmetric image connecting $\g_1^-$ to $\g_2^-$, and no edges connecting $\g_1^{\pm}$ to $\g_2^{\mp}$.  It follows that if $G_{\g}(\A)$ is a graph of Type II, then $w_{\A}(b) = 0$ and $G_{\g}(\A)$ contains a wave, a contradiction.  Thus, $G_{\g}(\A)$ is a graph of Type I and $w_{\A}(b) = 0$, as desired.
\end{proof}
\end{lemma}

Recall that we assume that the edges in a square graph or a c-graph have nonzero weights.  By Lemma \ref{crossseam}, a graph for $S^3$ with only two edges is trivial, and such a graph must be Type I with $w(a)=w(b)=0$ and $w(c)=w(d)=1$.  In keeping with Figure \ref{fig:WhGraphs}, we will refer to the edges $a_{\pm}$ of a square graph as \emph{horizontal edges} and the edges $c$ and $d$ as \emph{vertical edges}.  Note that if $(\delta,\eps)$ is a positive Heegaard diagram for $S^3$ such that $G_{\delta}(\eps)$ is a square graph, then $G_{\delta}(\eps)$ is not a Type III graph and as such, there is no wave based in $\delta$ for $(\delta,\eps)$.  By Theorem \ref{wave}, it follows that there must be a wave based in $\eps$.  In this setting, we may keep track of the structure of $G_{\delta}(\eps)$ after doing wave surgery.

\begin{lemma}\label{cgraph}
Suppose that $(\delta,\eps)$ is a positive Heegaard diagram for $S^3$ such that $G_{\delta}(\eps)$ is a square graph, and let $(\delta,\eps')$ be the result of doing wave surgery on $(\delta,\eps)$.  Then $G_{\delta}(\eps')$ is either a square graph or a c-graph.
\begin{proof}
Let $\delta = \{\delta_1,\delta_2\}$, $\eps = \{\eps_1,\eps_2\}$, and $\eps' = \{\eps_2,\eps_3\}$.  Observe that $\eps_3$ can be embedded disjointly from $\eps$ and consider $G_{\g}(\eps_1 \cup \eps_2 \cup \eps_3)$.  The curve $\eps_3$ may introduce at most one additional edge; however, doing wave surgery introduces no new points of intersection and does not change the sign of existing intersection points. Therefore, $(\delta,\eps')$ remains a positive diagram, and $G_{\delta}(\eps')$ cannot have such an edge.  Thus $G_{\delta}(\eps')$ is a subgraph of $G_{\delta}(\eps)$.

We must show that $G_{\delta}(\eps')$ has more than two edges.  Suppose by way of contradiction that $G_{\delta}(\eps')$ has two edges, so that it is the trivial diagram.  It follows that $\io(\delta,\eps_2)=1$, so we may assume that $\io(\delta_1,\eps_2) = 0$.  By Lemma \ref{missedcase}, $(\delta,\eps)$ is not a Heegaard diagram for $S^3$, a contradiction.
\end{proof}
\end{lemma}
As an immediate consequence, we have the next lemma.
\begin{lemma}\label{reducing}
Suppose that $(\delta,\eps)$ is a positive Heegaard diagram for $S^3$ and $G_{\delta}(\eps)$ is a square graph.  Then there is a sequence of diagrams $(\delta,\eps) = (\delta,\eps^l),(\delta,\eps^{l-1}),\dots, (\delta,\eps^0)$ such that $(\delta,\eps^i)$ is the result of wave surgery on $(\delta,\eps^{i+1})$, the graph $G_{\delta}(\eps^i)$ is a square graph for $i > 0$, and $G_{\delta}(\eps^0)$ is a c-graph.
\begin{proof}
By Lemma \ref{cgraph}, $G_{\delta}(\eps^{l-1})$ is either a square graph or a c-graph.  By induction on $(l-i)$, suppose that $G_{\delta}(\eps^{i})$ is square graph.  Observe that $(\delta,\eps^{i})$ always contains a wave, that the wave must be based in $\eps^{i}$, and that $\delta$ is preserved by the wave surgery.  Since $\io(\delta,\eps^{i-1}) < \io(\delta,\eps^{i})$, finitely many applications of Lemma \ref{cgraph} yield the desired result.
\end{proof}
\end{lemma}
Observe that for $i >0 $, $\Sigma \setminus \nu(\delta \cup \eps^i)$ is a collection of rectangles and two octagons whose sides alternate between arcs in $\delta$ and arcs in $\eps^i$.  It follows that $G_{\eps^i}(\delta)$ is a Type III graph with $w_{\delta}(b) = 0$ and all other edge weights nonzero.  As such, there is a unique way to do wave surgery on $\eps^i$, and so we call the sequence $\{\eps = \eps^l,\eps^{l-1},\dots,\eps^0\}$ guaranteed by Lemma \ref{reducing} the \emph{reducing sequence} for $\eps$ (without ambiguity).  See Figure \ref{fig:SquareGraphs}.

The involution of $\Sigma$ maps one of the octagons cut out by $\delta$ and $\eps^i$ to the other octagon, and thus it suffices to examine one of the octagons, which we will call $O_i$.  Since the diagram is positive, any wave or band for $(\delta,\eps^i)$ must connect opposite arcs of $\eps^i$ contained in the boundary of $O_i$.  We will call a wave $\omega$ parallel to a horizontal edge of $G_{\delta}(\eps^i)$ a \emph{horizontal wave}; otherwise $\omega$ is parallel to a vertical edge and we call $\omega$ a \emph{vertical wave}.  The notions of a \emph{horizontal band} and \emph{vertical band} are defined similarly.

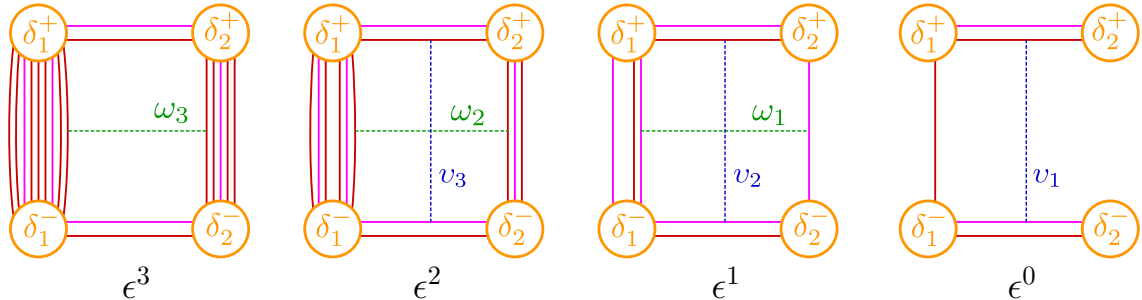
\begin{figure}[h!]
\centering
\includegraphics[scale = .33]{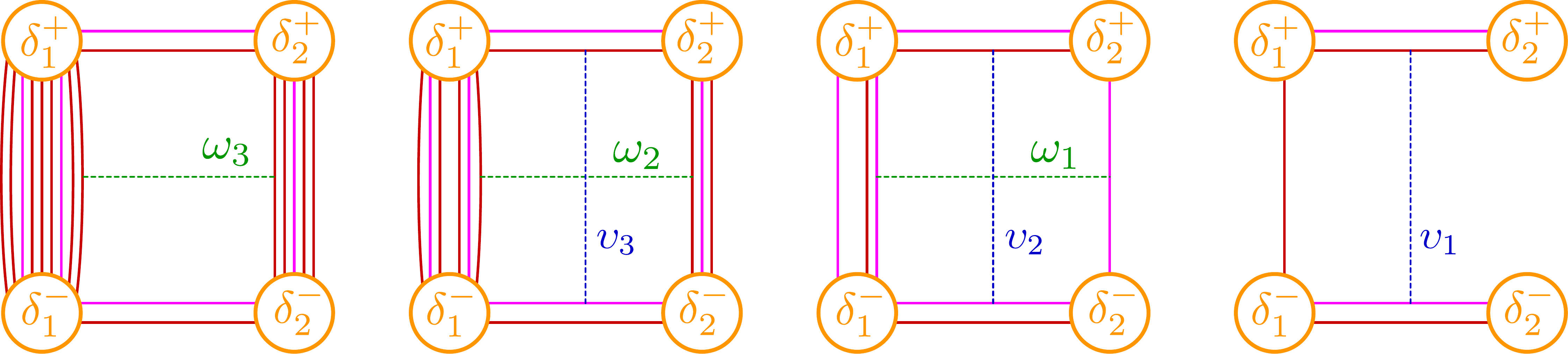}
\put(-385,-15){\large$\eps^3$}
\put(-275,-15){\large$\eps^2$}
\put(-162,-15){\large$\eps^1$}
\put(-50,-15){\large$\eps^0$}
\caption{The reducing sequence for a diagram $(\delta,\eps)$.  All four graphs shown have alternating adjacency.  Note that we have suppressed the gluing information needed to pass from $\Sigma_{\delta}(\eps^i)$ to a Heegaard diagram for $S^3$, although determining it is straightforward.}
\label{fig:SquareGraphs}
\end{figure}

Let $\omega_i$ denote the wave for $(\delta,\eps^i)$ such that surgery on $\omega_i$ yields $(\delta,\eps^{i-1})$, and let $\upsilon_i$ denote the band dual to each $\omega_i$ (so that banding $(\delta,\eps^{i-1})$ along $\upsilon_i$ yields $(\delta,\eps^i)$).  Let $\eps^i = \{\eps^i_1,\eps^i_2\}$.  Note that $\omega_i$ necessarily has endpoints in the same component of $\eps^i$, whereas $\upsilon_{i+1}$ always connects $\eps^i_1$ to $\eps^i_2$.  In addition, $\upsilon_i$ is vertical if and only if $\omega_i$ is horizontal.

\begin{lemma}\label{hor}
Every wave $\omega_i$ is a horizontal wave.
\begin{proof}
The band $\upsilon_1$ is contained in the c-graph $G_{\delta}(\eps^0)$, and thus it must be a vertical band.  It follows that $\omega_1$ is a horizontal wave.  Suppose by way of induction that $\omega_i$ is a horizontal wave in $G_{\delta}(\eps^i)$.  Since the band $\upsilon_i$ is also contained in $G_{\delta}(\eps^i)$, we have that $\upsilon_{i+1}$ cannot be parallel to $\omega_i$ since $\omega_i$ connects the same component of $\eps^i$ while $\upsilon_{i+1}$ connects distinct components of $\eps^i$.  Thus, $\upsilon_{i+1}$ is vertical, which implies that its dual wave $\omega_{i+1}$ is horizontal as well.
\end{proof}
\end{lemma}

Here we introduce some terminology to better understand what happens to arcs of $\Sigma_{\delta}(\eps^i)$ when two curves are banded together.  We say that two arcs in $\Sigma_{\delta}(\eps^i)$ are \emph{adjacent} if they cobound a rectangle in $\Sigma \setminus \nu(\delta \cup \eps^i)$ with arcs in $\delta$.

\begin{lemma}\label{weights}
If $\eps^i_3$ is the result of banding $\eps^i_1$ to $\eps^i_2$ along $\upsilon_{i+i}$, then $\io(\delta,\eps^i_3) = \io(\delta,\eps^i)$.  In addition, in $\Sigma_{\delta}(\eps^i \cup \eps^i_3)$ every arc of $\eps^i$ is adjacent to an arc of $\eps^i_3$.
\begin{proof}
In order to construct $\eps^i_3$ from $\eps^i_1$ and $\eps^i_2$ in $\Sigma_{\delta}$, we begin with two parallel copies of $\eps^i_1$ and $\eps^i_2$ which have been pushed off of $\eps^i$ so that two of their horizontal arcs lie in the octagon $O_i$.  Attaching the band $\upsilon_{i+1}$ to these two arcs transforms them into two vertical arcs and leaves all other arcs of $\eps^i$ unchanged, proving the first statement.  Since the vertical band $\upsilon_{i+1}$ connects $\eps^i_1$ to $\eps^i_2$, we have $w_{\eps^i_j}(a) \geq 1$ for $j = 1,2$.  Thus $w_{\eps^i}(a) \geq 2$, and every arc of $\eps^i$ is adjacent to some arc of its push-off that has not been changed by attaching the band $\upsilon_{i+1}$.  See Figure \ref{fig:AltAdj}.
\end{proof}
\end{lemma}

Now, we return to the trisection diagram $(\A,\n,\g)$.  By Lemma \ref{lemma:4cycle}, $G_{\g}(\A)$ is a square graph, and, as such, $\A$ has a reducing sequence $\{\A = \A^l,\A^{l-1},\dots,\A^0\}$ by Lemma \ref{reducing}.  The next lemma examines this reducing sequence more closely.

\begin{lemma}\label{onearc}
Suppose that $\A^0$ and $\A^l$ have a curve $\A^0_*$ in common.  Then $(\A,\n,\g)$ is not minimal.
\begin{proof}
Let $\A^l = \{\A^l_1,\A^0_*\}$.  If $(\A,\n,\g)$ is minimal, then $\io(\A_1,\g) \geq \io(\A_2,\g)$, and by Lemma \ref{weights}, $\io(\A^l_1,\g) \geq \io(\A^0,\g) > \io(\A^0_*,\g)$.  It follows that $\A_1 = \A^l_1$ and $\A_2 = \A^0_*$.    In addition, we have that $G_{\g}(\A_2)$ is a subgraph of the c-graph $G_{\g}(\A^0)$.  If $G_{\g}(\A_2)$ contains a cross-seam, them $(\A,\n,\g)$ is not minimal by Lemma \ref{g2wave}.  Otherwise, $\io(\A_2,\g_j) = 0$ for either $j =1$ or $j=2$; therefore, $(\A,\g)$ is not a Heegaard diagram for $S^3$ by Lemma \ref{missedcase}.
\end{proof}
\end{lemma}


In order to deal with the intricacy of the remaining case in Section \ref{sec:hard_case}, we will need to keep track of the cyclic ordering of arcs at the vertices of $\Sigma_{\A}(\g)$.  One way to obtain this information is to understand adjacency in $\Sigma_{\g}(\A)$, since the regions of $\Sigma \setminus \nu(\A \cup \g)$ appear in both graphs.  We say that $(\A^i,\g)$ has \emph{upper alternating adjacency} (resp. \emph{lower alternating adjacency}) if the horizontal boundary arc of $O_i$ in the edge $a^+$ (resp. $a^-$)  is adjacent to an arc in the opposite curve of $\A^i$.  See Figure \ref{fig:SquareGraphs}.  If $(\A^i,\g)$ has both upper and lower alternating adjacencies, we say that it has \emph{alternating adjacency}.  The remainder of the section includes several lemmas which establish that $G_{\g}(\A)$ has alternating adjacency.

\begin{lemma}\label{adjpreserve}
If $(\A^i,\g)$ has upper (resp. lower) alternating adjacency, then $(\A^{i+1},\g)$ has upper (resp. lower) alternating adjacency.
\begin{proof}
By Lemma \ref{hor}, the band $\upsilon_{i+1}$ dual to the horizontal wave $\omega_{i+1}$ is vertical.  Let $\A^i_3$ correspond to the curve created by banding $\A^i_1$ to $\A^i_2$ along $\upsilon_{i+1}$, so that $\A^{i+1}$ is either $\{\A^i_1,\A^i_3\}$ or $\{\A^i_2,\A^i_3\}$.

Let $\A^* = \{\A^i_1,\A^i_2,\A^i_3\}$ and consider $\Sigma_\g(\A^*)$, letting $O_*$ denote the octagon in $\Sigma \setminus \nu(\A^* \cup \g)$ such that $O_* \subset O_{i+1}$.  By the proof of Lemma \ref{weights}, both vertical arcs of $\pd O_*$ in $\A^*$ are in $\A^i_3$, and $\pd O_*$ has one horizontal arc in each of $\A^i_1$ and $\A^i_2$; call these arcs $\A^*_1$ and $\A^*_2$, respectively.  Suppose that $\A^*_1$ is in the edge $a_+$ (resp. $a_-$).  Then by Lemma \ref{weights}, $\A^*_1$ is adjacent to an arc in $\A^i_3$, which in turn is adjacent to an arc in $\A^i_2$ by the upper (resp. lower) alternating adjacency of $\{\A^i,\g\}$.  Thus, independently of whether $\A^i_1 \in \A^{i+1}$ or $\A^i_2 \in \A^{i+1}$, the diagram $(\A^{i+1},\g)$ must also have upper (resp. lower) alternating adjacency.  See Figure \ref{fig:AltAdj}.
\end{proof}
\end{lemma}

\begin{figure}[h!]
\centering
\includegraphics[scale = .45]{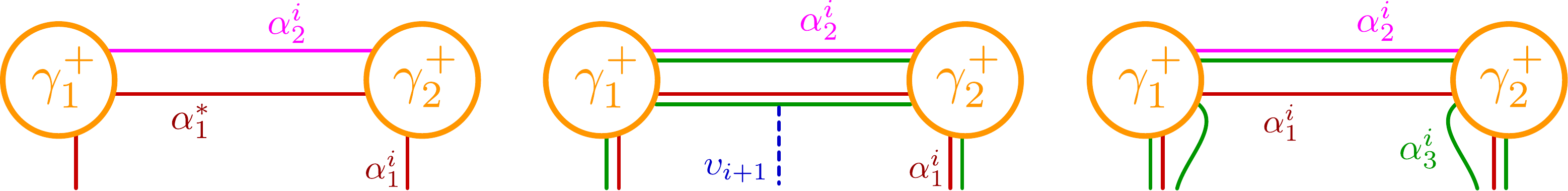}
\put(-350,-15){(a)}
\put(-208,-15){(b)}
\put(-68,-15){(c)}
\caption{The graph $G_\g(\alpha^i)$:  (a) exhibits upper alternating adjacency, (b) shows the parallel push offs of $\alpha_1^i$ and $\alpha_2^i$, and (c) shows $\alpha_3^i$, the result of the banding.}
\label{fig:AltAdj}
\end{figure}

Note that, in the proof of Lemma \ref{adjpreserve}, if $\A^i_1\in \A^{i+1}$, then we need not require that $(\A^i,\g)$ have upper (resp. lower) alternating adjacency.  The next lemma follows immediately.

\begin{lemma}\label{adjnew}
If an arc in $\A^i_j$ is in the upper (resp. lower) horizontal boundary of $O_i$ and $\A^i_j \in \A^{i+1}$, then $(\A^{i+1},\g)$ has upper (resp. lower) alternating adjacency.
\end{lemma}

Now, we combine Lemmas \ref{adjpreserve} and \ref{adjnew} to determine adjacency for $(\A,\g)$ when $(\A,\n,\g)$ is minimal.

\begin{lemma}\label{nocommon}
If $\A^0$ and $\A^l$ do not have a curve in common, then $(\A^l,\g)$ has alternating adjacency.
\begin{proof}
Note that the band $\upsilon_1$ is vertical, so it connects $\A^0_1$ to $\A^0_2$.  Suppose without loss of generality that $\upsilon_1$ connects an upper horizontal arc of $\A^0_1$ to a lower horizontal arc of $\A^0_2$, and suppose further that $\A^0_1 \in \A^1$.  By Lemma \ref{adjnew}, $(\A^1,\g)$ has upper alternating adjacency, and by Lemma \ref{adjpreserve}, $(\A^i,\g)$ has upper alternating adjacency for all $i > 0$.  Let $k$ be the largest index such that $\A^0_1 \in \A^k$, and suppose $\A^k = \{\A^0_1,\A^k_2\}$.  By assumption, $k < l$, and following the proof of Lemma \ref{adjpreserve}, we have that the upper horizontal boundary of $O_k$ is an arc in $\A^0_1$ while the lower boundary is an arc in $\A^k_2$.  Since $\A^0_1 \notin \A^{k+1}$, we have $\A^k_2 \in \A^{k+1}$ and thus $(\A^{k+1},\g)$ has lower alternating adjacency by Lemma \ref{adjnew}.  Finally, Lemma \ref{adjpreserve} implies that $(\A^i,\g)$ has lower alternating adjacency for all $i > k$, and thus $(\A_l,\g)$ has alternating adjacency as desired.
\end{proof}
\end{lemma}

In summary, we have the following.

\begin{lemma}\label{alt}
If $G_{\g}(\A)$ is a square graph, then $(\A,\g)$ has alternating adjacency or $(\A,\n,\g)$ is not minimal.
\begin{proof}
By Lemma \ref{reducing}, there is a reducing sequence $\A = \A^l,\A^{l-1},\dots,\A^0$ for $(\A,\g)$.  If $\A^l$ and $\A^0$ have no curve in common, then $(\A,\g)$ has alternating adjacency by Lemma \ref{nocommon}.  Otherwise, $(\A,\n,\g)$ is not minimal by Lemma \ref{onearc}.
\end{proof}
\end{lemma}

\section{The case in which $m = -n \pm 1$}\label{sec:hard_case}

Up to this point in the paper, our arguments have not required information about the cyclic ordering of the intersections of arcs with the fat vertices in $\Sigma_{\A}(\n, \g)$.  However, this case is significantly more delicate than the preceding ones, and thus we must be more precise about how we recover $(\A,\n,\g)$ from $\Sigma_{\A}(\n,\g)$.  We will assume that $\n$ has been isotoped so that all inessential intersections between $\n$ and $\g$ correspond to triangles based in $\A_1^+$ or $\A_2^+$.

Let us recall our current assumptions.    First, neither diagram $(\A,\g)$ or $(\n,\g)$ is trivial, there is no wave based in $\g$, and the trisection diagram $(\A,\n,\g)$ is minimal.  Second, $G_{\A}(\g)$ is a Type III graph such that $w_b(\g) = 0$ and all other weights are nonzero. This implies that $\Sigma \setminus \nu(\A \cup \g)$ contains some number of rectangular regions and two octagonal regions, which we call $O'$ and $O''$.   Third, $G_\g(\A)$ is a square graph with alternating adjacency and $(\A,\g)$ contains a wave based in $\A_1$.  Fourth, by Lemma \ref{nogwave} we may assume that $(\n,\g)$ does not contain a wave based in $\g$, so that it must contain a wave based in $\n$.  Finally, the slopes of the $\n$--arcs in $\Sigma_\A(\n,\g)$ are both $-\frac{n\pm1}{n}$.  Since the slopes of the $\n$--arcs differ from the slopes of the $\g$--arcs, and since there are no inessential intersection points based at $\A_i^-$, $\n_i$ intersects $\A_i^-$ in an edge of either $O'$ or $O''$ for $i=1,2$.



\begin{lemma}\label{weightone}
If $G_{\g}(\A)$ has an edge of weight one, then  $(\A,\n,\g)$ is not minimal. 
\begin{proof}
Suppose that in the square graph $G_{\g}(\A)$, one of the edges $a_{\pm}$, $c$, or $d$ has weight one (see Figure \ref{fig:WhGraphs}).  Note that each octagon $O'$ and $O''$ contains three boundary components in $\A_1$ and one boundary component in $\A_2$.  By Lemma \ref{reducing}, there is a hortizontal wave contained in $\Sigma_{\g}(\A)$, and so both vertical boundary arcs of $O'$ in $\Sigma_{\g}(\A)$ are in $\A_1$. This means that $O'$ has horizontal boundary arcs in both $\A_1$ and $\A_2$; thus $w_{\A}(a) \geq 2$.

Consequently, we suppose without loss of generality that $w_{\A}(d) = 1$.  If $w_{\A_2}(d) = 0$, then $G_{\g}(\A_2)$ is a Type III graph containing a cross-seam and $(\A,\n,\g)$ is not minimal by Lemma \ref{g2wave}.  If $w_{\A_2}(d) = 1$, let $\gamma_3 = \Delta_a$ and let $\g' = \{\g_2,\g_3\}$.  Then $\g'$ is related to $\g$ by a handle slide, so by the minimality of $(\A,\n,\g)$, we have that $\io(\A_2,\g) \leq \io(\A_2,\g')$ and thus $\io(\A_2,\g_1) \leq \io(\A_2,\g_3)$.  However, this implies that $w_{\A_2}(a) + w_{\A_2}(c) \leq w_{\A_2}(c) + 1$ and so $w_{\A_2}(a) = 1$ and $\io(\A_2,\g) = \io(\A_2,\g')$.  By the above arguments, $w_{\A_1}(a) > 0$, and so
\[ \io(\A_1,\g') = w_{\A_1}(a) + w_{\A_1}(c) < 2w_{\A_1}(a) + w_{\A_1}(c) = \io(\A_1,\g).\]
This implies $\calc(\A,\n,\g') < \calc(\A,\n,\g)$, as desired.
\end{proof}
\end{lemma}

The next lemma uses our work from Section \ref{sec:digress} to determine how $\n_1$ meets $\A_1^+$ relative to $\g$ in $\Sigma_\A(\n,\g)$.

\begin{lemma}\label{exit}
For $i= 1,2$, the arc $\n_i$ intersects $\A_i^+$  between parallel arcs of $\gamma$  in the edge $a_+$.
\begin{proof}
Let $\A_* \subset \A_1$ denote the boundary arc of the octagon, say $O'$, which contains the intersection of $\n_1$ with $\A_1^-$, and observe by examination of $\Sigma_{\A}(\g)$ that in $O'$ the arc $\A_*$ is opposite an arc in $\A_2$.   See Figure \ref{fig:sec5_ExamplePair1}.  By Lemma \ref{hor}, the wave $\omega$ for $(\A,\g)$ is horizontal, and thus opposite pairs of vertical arcs in $\Sigma_{\g}(\A)$ contained in $\pd O'$ meet $\omega$ and are in the same curve $\A_1$.  This implies that $\A_*$ is a horizontal arc in one of $a_{\pm}$ in $G_{\g}(\A)$.  Now, we invoke Lemma \ref{alt}, which asserts that $(\A,\g)$ has alternating adjacency.  It follows that $\A_*$ also cobounds a rectangular component $R$ of $\Sigma \setminus \nu(\A \cup \g)$, and is opposite an arc of $\A_2$ in $R$.  Since the only such rectangles in $\Sigma_{\A}(\g)$ lie between parallel arcs in the edge $a_+$, the desired statement follows.  A parallel argument shows that the same is true for $\n_2$ and $\A_2^+$.
\end{proof}
\end{lemma}

\begin{figure}[h!]
\centering
\includegraphics[scale = .7]{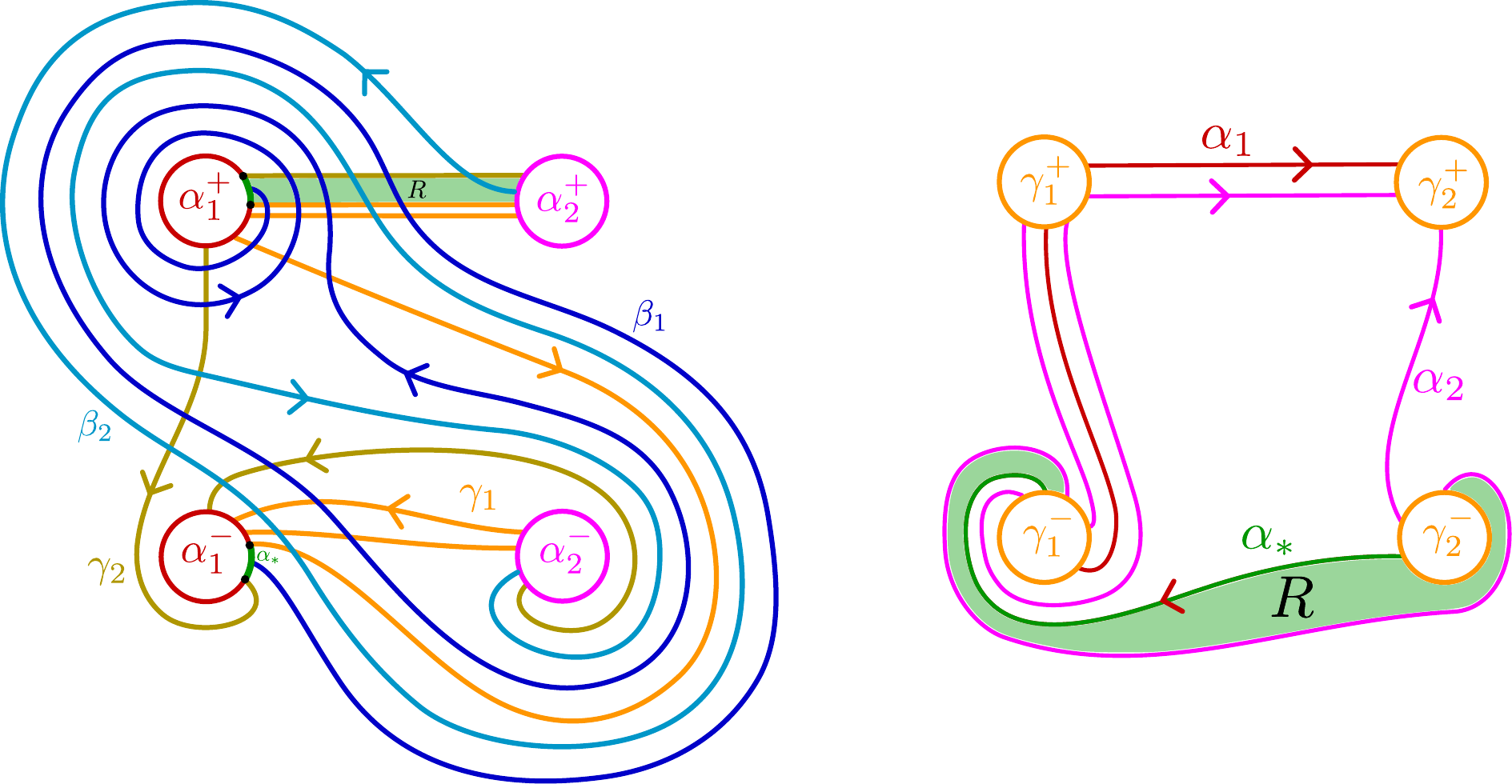}
\put(-290,-20){\large $\Sigma_\A(\n,\g)$}
\put(-80,-20){\large $\Sigma_\g(\A)$}
\caption{An example of $\Sigma_\A(\n,\g)$, along with the corresponding $\Sigma_\g(\A)$.  At left, $W_1=\frac{13}{5}$, $W_2=\frac{1}{3}$, and the slope of the $\n$--curves is $-\frac{3}{4}$.  Notice that  $(\n,\g)$ is a positive Heegaard diagram, according with Lemma \ref{sameint}.  We have assumed that there are no inessential intersections at $\A_1^-$, which prescribes that $\beta_1$ leaves this vertex as shown.  The existence of the rectangle $R$ on the right dictates that $\beta_1$ intersects $\A_1^+$ as shown, as in Lemma \ref{exit}.}
\label{fig:sec5_ExamplePair1}
\end{figure}

We note that if the weight of the edge $a_+$ in $\Sigma_{\g}(\A)$ is large, then $\n_1$ and $\n_2$ will not necessarily intersect $\A_1^+$ and $\A_2^+$ between the \emph{same} set of parallel arcs in the edge $a_+$ in $\Sigma_{\A}(\g)$ as in Figure \ref{fig:sec5_ExamplePair1}.  We will also keep track of how $\n$ and $\g$ intersect near the vertices $\A_i^+$, but we require the following lemma first.

\begin{lemma}\label{esssame}
All essential points of $\n \cap \g$ are coherently oriented.
\begin{proof}
Consider $\Sigma_{\A}(\n,\g)$.  Choose an orientation for $\g$ and note that since $(\A,\g)$ is a positive diagram, each arc corresponding to a single edge in $G_{\A}(\g)$ has the same orientation.  Thus, we may orient the edges of $G_{\A}(\g)$ so that oriented arcs originate in $\A_1^+$ and $\A_2^-$ and terminate in $\A_1^-$ and $\A_2^+$.  Now, we may orient $\n_1$ and $\n_2$ so that all essential points of intersection of $\n$ and the edge $a_+$ are coherently oriented.  The involution $J$ preserves $\n_1$ and $\n_2$ setwise while reversing their orientations; hence essential points of $\n \cap a_+$ and $\n \cap a_-$ are coherently oriented.  Since the slope $\frac{m}{n}$ of $\n$ is negative, it follows that essential points of $\n \cap c$ and $\n \cap a_-$ are coherently oriented as well.  Finally, the homeomorphism $\psi$ of $G_{\A}(\n,\g)$ swaps $\g$-arcs in the edge $c$ with those in the edge $d$ while preserving their orientations and maps $\n$-arcs of slope $-\frac{n+1}{n}$ to arcs of slope $-\frac{n+1}{n+2}$.  Thus, by applying $\psi$ to $G_{\A}(\n,\g)$, we may make a parallel argument that all essential points of $\n \cap d$ and $\n \cap a_-$ are coherently oriented.  We conclude that all essential intersections of $\n$ and $\g$ in $G_{\A}(\n,\g)$ are coherently oriented.  See Figure \ref{fig:sec5_ExamplePair1} and Figure \ref{fig:sec5_FirstWave}(a).
\end{proof}
\end{lemma}

Now, we define the \emph{winding of $\n_j$ relative to $\g$ at $\A_i$}, denoted $W_i(\n_j)$ (or simply $W_i$ when the relevant $\n$--curve is clear), to be the number of inessential intersections (counted with sign) of $\n$ and $\g$ that are vertices of triangles based in $\A_i^+$ divided by $\io(\A_i,\g)$ for $i=1,2$.  An inessential intersection is positive if it is coherently oriented with the essential points of intersection and negative otherwise.  In the next four lemmas, we place restrictions on the values of $W_1$ and $W_2$ by showing that certain combinations of windings imply that $(\A,\n,\g)$ is not minimal or $(\n,\g)$ does not contain a wave based in $\n$, which contradicts our prior assumptions.

\begin{lemma}\label{noninteger}
If $W_1$ or $W_2$ is an integer, then $(\A,\n,\g)$ is not minimal.
\begin{proof}
If the winding of $\n$ relative to $\g$ at $\A_1$ is an integer, then the two octagonal regions of $\Sigma \setminus \nu(\A \cup \g)$ are separated by a single arc contained in $\A_1$.  It follows that in the square graph $G_{\g}(\A)$, one of the edges has weight one and thus  $(\A,\n,\g)$ is not minimal by Lemma \ref{weightone}.
\end{proof}
\end{lemma}

\begin{lemma}\label{sameint}
If $W_i < 0$ for either $i=1$ or $i=2$, then $G_\n(\g)$ does not contain a wave.
\begin{proof}
Consider $\Sigma_{\A}(\n,\g)$ and note that there is a subarc of a $\g$-arc contained in the edge $a_+$ with endpoints in essential intersections of $\n_1$ and $\n_2$; thus $G_{\n}(\g)$ contains an edge connecting $\n_1^+$ to $\n_2^+$.  For any arc of $\g$ originating from $\A_1^+$, its first essential intersection is with $\n_2$.  Suppose that $\n_1 \cap \g$ contains an inessential point of intersection based in $\A_1^+$ which is oriented opposite the essential points of intersection.  Then there is an arc of $\g$ connecting an inessential point of intersection with $\n_1$ with an oppositely oriented essential point of intersection with $\n_2$.  It follows that $G_{\n}(\g)$ contains an edge connecting $\n_1^+$ to $\n_2^-$, and by the argument of Lemma \ref{sat}, $G_{\n}(\g)$ does not contain a wave. 

A similar argument shows that if $\n_2 \cap \g$ has an oppositely oriented inessential intersection point based in $\A_2^+$, then $G_{\n}(\g)$ again does not contain a wave.
\end{proof}
\end{lemma}

It follows that $W_i > 0$ for $i=1,2$; hence, $(\n,\g)$ is a positive diagram, and $G_{\g}(\n)$ is a square graph.

\begin{lemma}\label{nonewind}
If $0 < W_i < 1$ for both $i = 1,2$, then $G_{\n}(\g)$ does not contain a wave.
\begin{proof}
If $0 < W_1 < 1$, then there is an arc $\g'$ in $\gamma$ whose endpoints lie in $\A_1^+$ and $\n_2$.  See Figure \ref{fig:sec5_NoWinding}(a).  By assumption, there are no inessential intersections of $\n$ and $\g$ based at $\A_1^-$, and so every arc of $\gamma$ with an endpoint on $\A_1^-$ meets $\n_2$.  Thus, the endpoint of $\g'$ in $\A_1^+$ is the endpoint of another arc $\g''$ connecting $\A_1^-$ to $\n_2$, and in the graph $G_{\n}(\g)$, the arc $\g' \cup \g''$ is in an edge connecting $\n_2^+$ to $\n_2^-$. 

A parallel argument shows that if $0 < W_2< 1$, then $G_{\n}(\g)$ contains an arc in an edge connecting $\n_1^+$ to $\n_1^-$.  Together, these two arcs consitute a wave-busting pair, and by Lemma \ref{sat}, we have that $G_{\n}(\g)$ does not contain a wave.
\end{proof}
\end{lemma}


\begin{lemma}\label{bothwind}
If $W_i > 1$ for both $i = 1,2$, then $G_{\n}(\g)$ does not contain a wave.
\begin{proof}
If $W_1 > 1$, then there is an arc of $\g$ emanating from $\A_1^+$ that meets $\n$ in two consecutive inessential points of $\n_1 \cap \g$ which are coherently oriented.  See Figure \ref{fig:sec5_NoWinding}(b).  Such an arc contributes an edge connecting $\n_1^+$ to $\n_1^-$ in $G_{\n}(\g)$.  Similarly, if $W_2 > 1$, then $G_{\n}(\g)$ contains an edge connecting $\n_2^+$ and $\n_2^-$, and thus $G_{\n}(\g)$ contains a wave-busting pair.  We conclude that $G_{\n}(\g)$ does not contain a wave by Lemma \ref{sat}.
\end{proof}
\end{lemma}

\begin{figure}[h!]
\centering
\includegraphics[scale = 1]{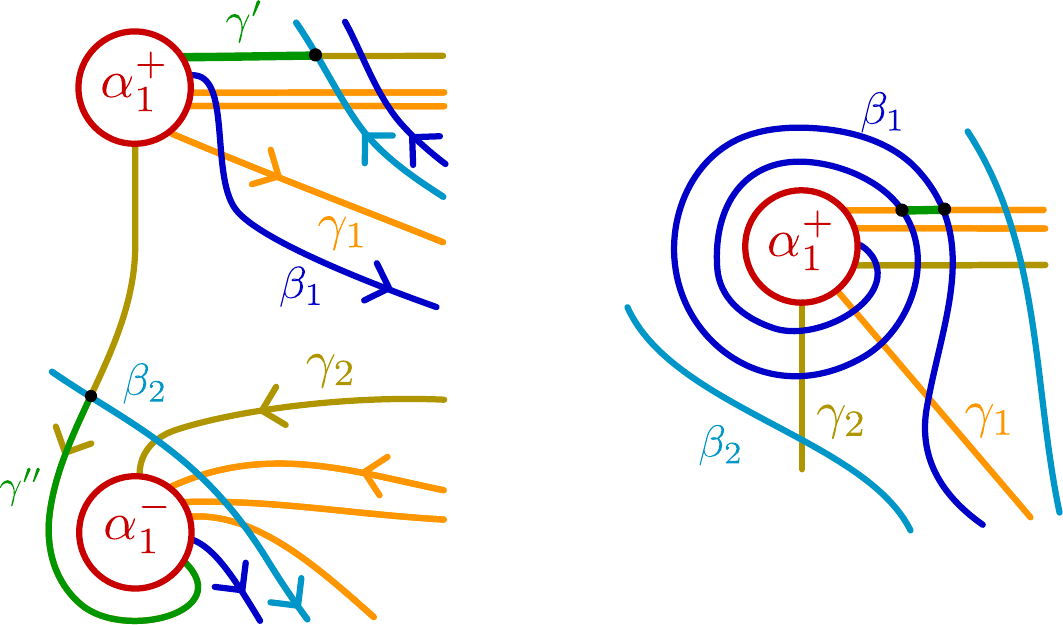}
\put(-240,-20){\large (a)}
\put(-60,-20){\large (b)}
\caption{The cases in which (a) $W_1<1$ and (b) $W_1>1$.}
\label{fig:sec5_NoWinding}
\end{figure}


Thus, we are left with two cases: $W_i > 1$ and $0 < W_j < 1$ for $\{i,j\} = \{1,2\}$.  Unfortunately, in both of these cases it is possible for $(\n,\g)$ to admit a wave based in $\n$, and so a more sophisticated argument is required -- especially since the result $\n'$ of wave surgery on $\n$ no longer satisfies $\io(\A,\n') = 2$.  Hence, we must appeal to a line of reasoning other than that involving minimal complexity $\calc(\A,\n,\g)$.  In the remaining two lemmas, we complete the proof of Proposition \ref{claim1} by using Lemma \ref{reducing} to determine that $(\n,\g)$ is not a Heegaard diagram for $S^3$.

\begin{lemma}\label{onewind}
If $W_1 > 1$ and $0<W_2 < 1$, then $(\n,\g)$ is not a Heegaard diagram for $S^3$.
\begin{proof}
First, we note that $(\n,\g)$ contains a wave $\omega$:  Since $W_1 > 1$, an octagonal component $O'$ of $\Sigma \setminus \nu(\A \cup \g)$ contains an arc $\n_*$ of $\n_1$, both of whose endpoints are inessential points of $\n \cap \g$, such that there is an arc $\omega$ which avoids $\n \cup \g$ in its interior, has one endpoint on $\n_*$, meets $\A_2^-$ in a single point contained in $O_*$, and continues through a rectangular component of $\Sigma \setminus \nu(\A \cup \g)$ to meet $\n_1$ in an oppositely oriented point.  See Figure \ref{fig:sec5_FirstWave}(a). 

\begin{figure}[h!]
\centering
\includegraphics[scale = .35]{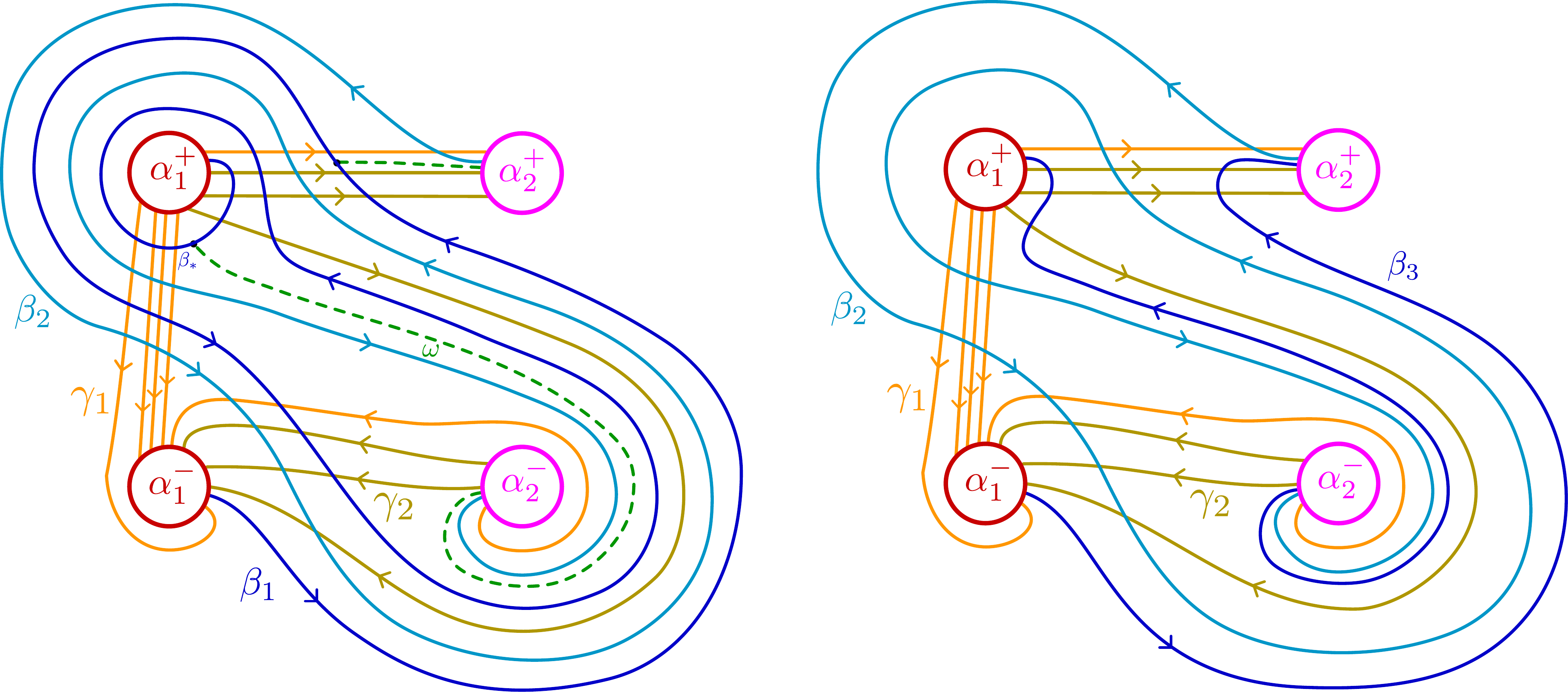}
\put(-330,-20){\large (a)}
\put(-100,-20){\large (b)}
\caption{(a) Before and (b) after the surgery along the $\n_1$--wave $\omega$ in $\Sigma\setminus\nu(\A\cup\g)$.}
\label{fig:sec5_FirstWave}
\end{figure}

As depicted in Figure \ref{fig:sec5_FirstWave}(b), we let $\n_3$ denote the new curve created by doing surgery on $\n_1$ along $\omega$.  The curve $\n_3$ intersects $G_{\A}(\n,\g)$ in two arcs, each having slope $-\frac{1}{1}$.  Further, $W_1(\n_3) = W_1(\n_1) - 1$ while $W_2(\n_3) = 1 - W_2(\n_2)$.  Let $\n' = \{\n_2,\n_3\}$.  We claim that $(\n',\g)$ does not contain a wave based in $\gamma$.  By Lemma \ref{weightone}, each edge $c$ and $d$ of $G_{\g}(\A)$ has weight at least two, which implies that $\Sigma \setminus \nu(\A \cup \g)$ contains two rectangles $R_1$ and $R_2$ such that $R_i$ has two opposite boundary components contained in $\gamma_i$.  See Figure \ref{fig:sec5_RectBust}.  Since the slope of $\n_2$ is $-\frac{n\pm1}{n}$, we have that $\n_2$ intersects each edge of $G_{\A}(\g)$ essentially, and thus $\n_2$ meets $R_1$ and $R_2$ in arcs $\n^1_*$ and $\n_2^*$.   In $G_{\g}(\n')$, the pair $\{\n^1_*,\n^2_*\}$ is a wave-busting pair, and thus $G_{\g}(\n')$ does not contain a wave by Lemma \ref{sat}.

\begin{figure}[h!]
\centering
\includegraphics[scale = .35]{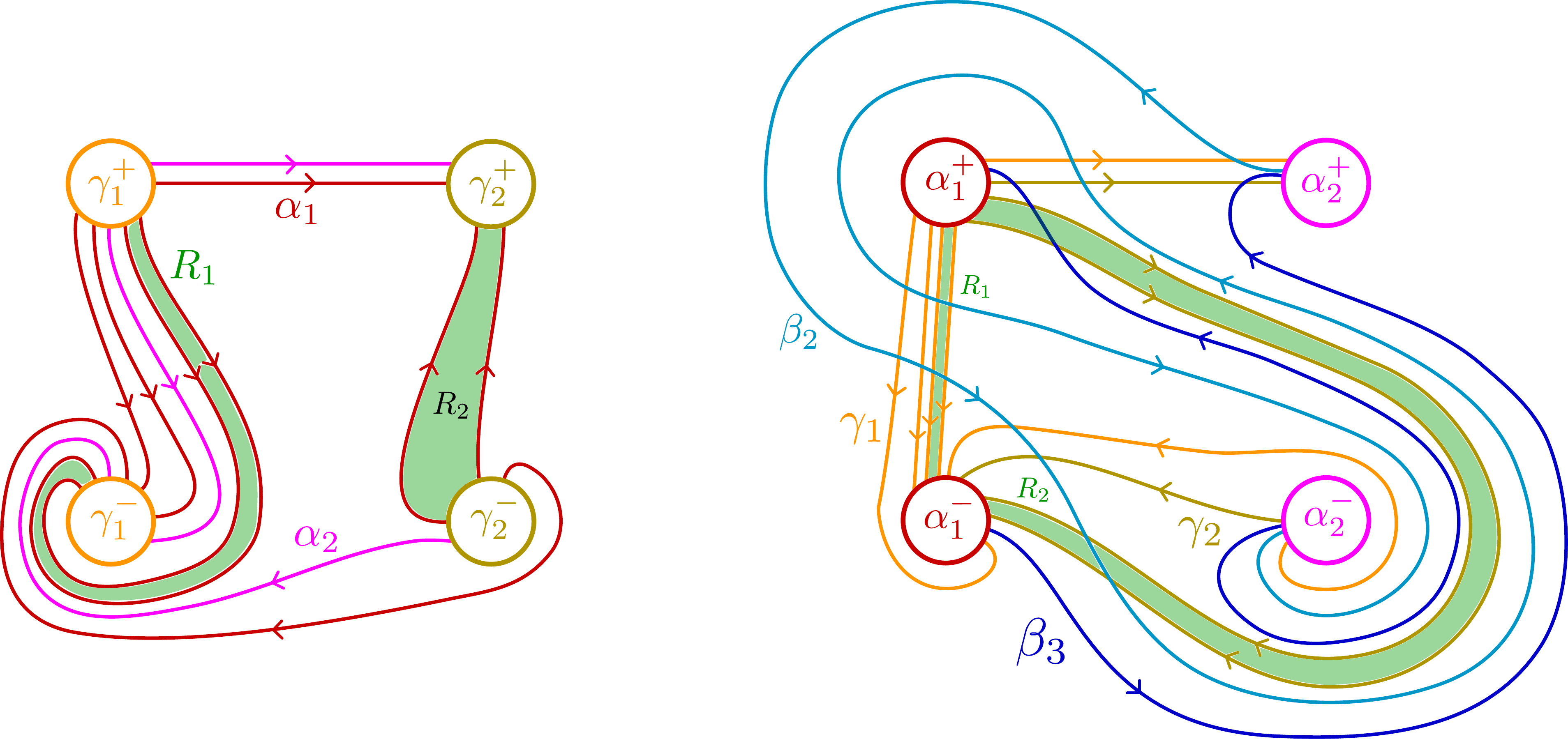}
\put(-130,-20){\large $\Sigma_\A(\n',\g)$}
\put(-330,-20){\large $\Sigma_\g(\A)$}
\caption{The fact that no edge in $G_\g(\A)$ has weight one gives rise to a pair of rectangles in $\Sigma_\g(\A)$ that in turn produce a wave-busting pair $(\n_*^1,\n_*^2)$ in $\Sigma_\A(\n',\g)$.  Each arc $\n_*^i$ comes from a component of $\n_2\cap R_i$.}
\label{fig:sec5_RectBust}
\end{figure}

Now, suppose that $W_1(\n_1) > 2$, so that $W_1(\n_3) > 1$.  In this case, as in the proof of Lemma \ref{bothwind}, there is an arc $\g_3^*$ of $\g$ which meets $\n_3$ in its coherently oriented endpoints and avoids $\n'$ in its interior.  In addition, $0 < W_2(\n_2) < 1$; hence there is an arc $\g_2^*$ contained in the edge $a_+$ which meets $\n_2$ in its coherently oriented endpoints and avoids $\n'$ in its interior.  See Figure \ref{fig:sec5_WindingBust}.  It follows that $\{\g_2^*,\g_3^*\}$ is a wave-busting pair and $G_{\n'}(\g)$ does not contain a wave by Lemma \ref{sat}.  In this case, we conclude that $(\n',\g)$, and therefore $(\n,\g)$, is not a Heegaard diagram for $S^3$ by Theorem \ref{wave}.

\begin{figure}[h!]
\centering
\includegraphics[scale = .5]{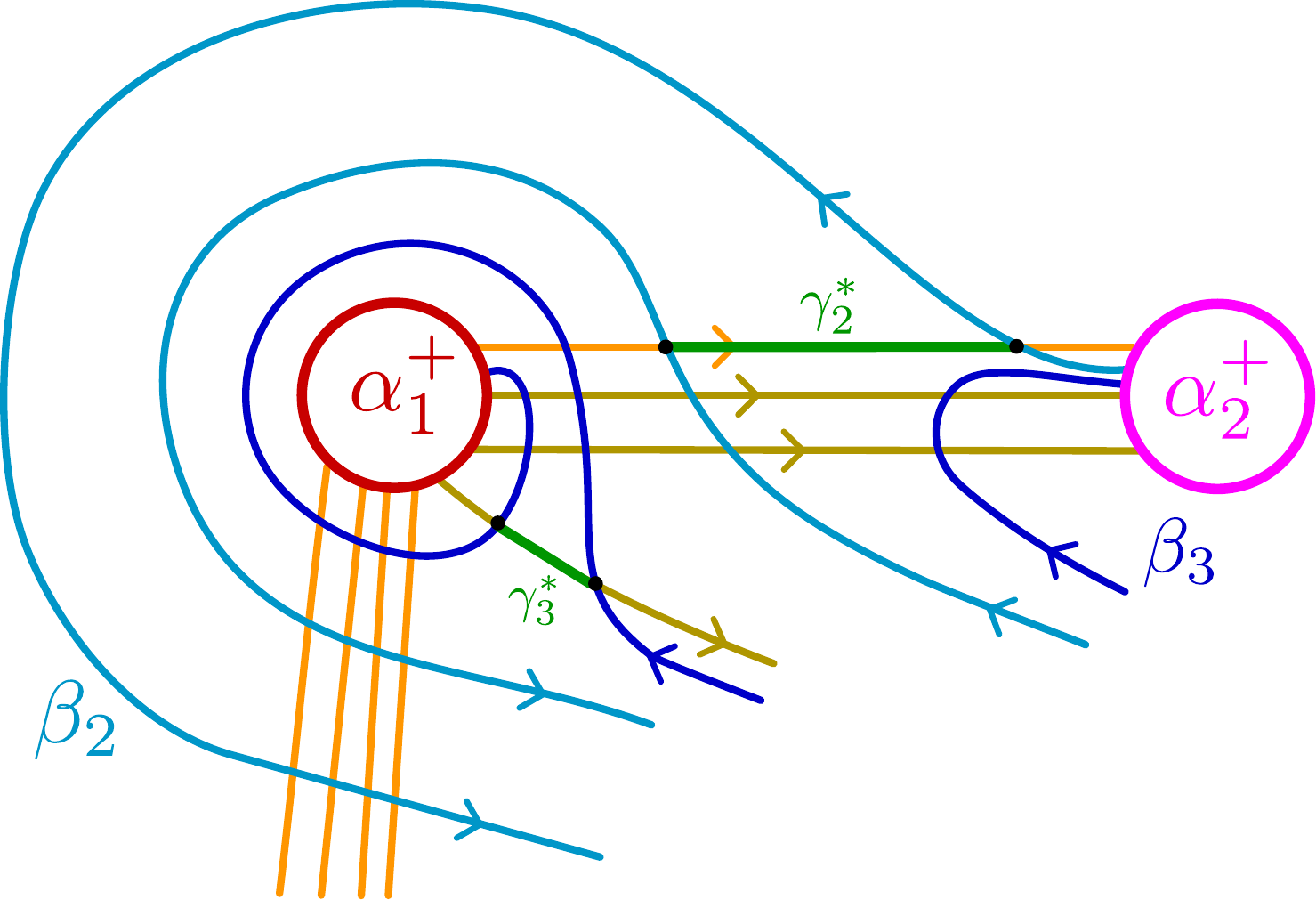}
\caption{A local view of the $\A_i^+$ in the case that $W_1(\n_3)>1$.  In this case, $\{\g_2^*,\g_3^*\}$ is a wave-busting pair of arcs.}
\label{fig:sec5_WindingBust}
\end{figure}

On the other hand, suppose that $1 < W_1(\n_1) < 2$, so that $0< W_1(\g_3) < 1$.  In this case the situation is considerably more complicated.  Suppose (by applying $\psi$ if necessary and noting that $\psi$ does not alter the values of $W_1$ and $W_2$) that the slope of $\n_2$ is $\frac{m}{n}=-\frac{|n|-1}{|n|}$.  By Lemma \ref{inter}, the curve $\n_2$ has exactly $k = \frac{|n|}{2}-1$ essential intersections with each arc in the edges $a_{\pm}$.  Consider the arc $\omega'$ connecting $\A_1^+$ to $\A_2^-$ in the octagon $O'$.  By the symmetry of $G_{\A}(\n,\g)$ under the involution $J$ and Lemma \ref{exit}, if we extend $\omega'$ through $\A_1$, it intersects $\A_1^-$ between parallel arcs corresponding to the edge $a_-$. See Figure \ref{fig:sec5_BigExample}. Thus, $\omega'$ may be extended from $\A_1^-$ and from $\A_2^+$ to an arc $\omega^*$ which avoids $\n_3$ and meets $\n_2$ in $2k$ points, where $k$ of these points arise from the $k$ essential intersections of $\n_2$ with $a_+$ and the other $k$ points arise from the intersections of $\n_2$ with $a_-$.  Moreover, each intersection point of the first type is naturally paired with an intersection point of the second type.  For this reason, we will label the points of intersection of $\omega^*$ with $\n_2$ as $p^+_1,\dots,p^+_{k},p^-_{k},\dots,p^-_1$. 

\begin{figure}[h!]
\centering
\includegraphics[scale = .45]{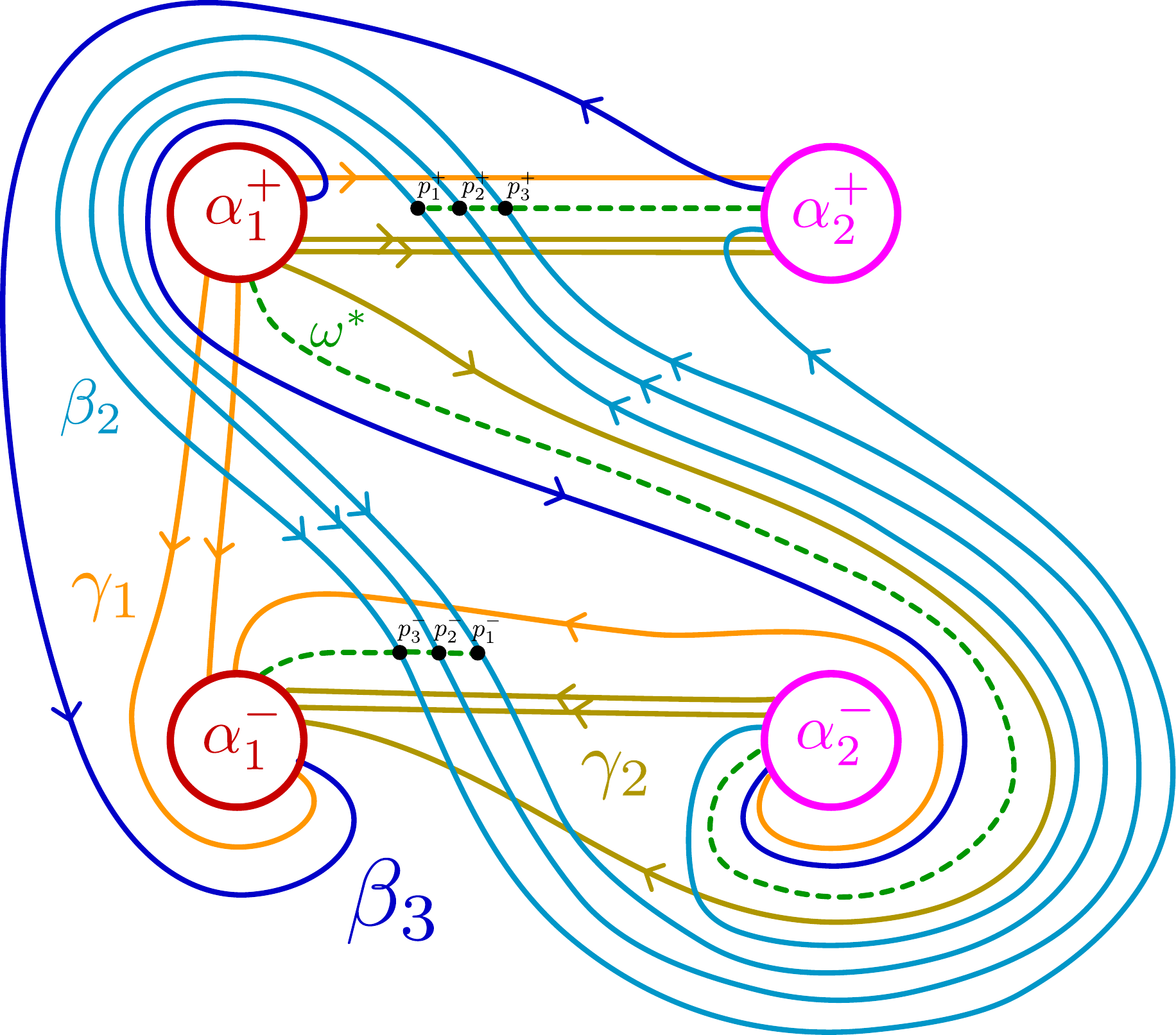}
\caption{An example of $\Sigma_\A(\n',\g)$ in the case that $W_1(\n_3)<1$, along with the nested wave sequence $\omega^*$. The slope of $\n_2$ is $-\frac{7}{8}$.}
\label{fig:sec5_BigExample}
\end{figure}

We call the arc $\omega^*$ a \emph{nested wave sequence}, because it gives rise to $k$ successive waves.  Let $\omega_i$ denote the subarc of $\omega^*$ connecting $p^+_i$ to $p^-_i$, so that $\omega^* = \omega_1$.  Let $\n^*_k = \n_2$ and let $\n'_k = \n' = \{\n^*_k,\n_3\}$.  Observe that $\omega_k$ is a wave for $(\n'_k,\g)$ based in $\n^*_k$, and let $\n'_{k - 1} = \{\n^*_{k-1},\n_3\}$ be the result of doing surgery on $\n^*_{k}$ along $\omega_{k}$.  Inductively, $\omega^*$ gives rise a to sequence of wave surgeries:  We will let $\n'_i = \{\n^*_i,\n_3\}$, where $\n^*_i$ is the result of doing surgery on $\n^*_{i+1}$ along the wave $\omega_{i+1}$ for the diagram $\n'_{i+1}$.  See Figure \ref{fig:sec5_WaveSequence}.

\begin{figure}[h!]
\centering
\includegraphics[scale = .38]{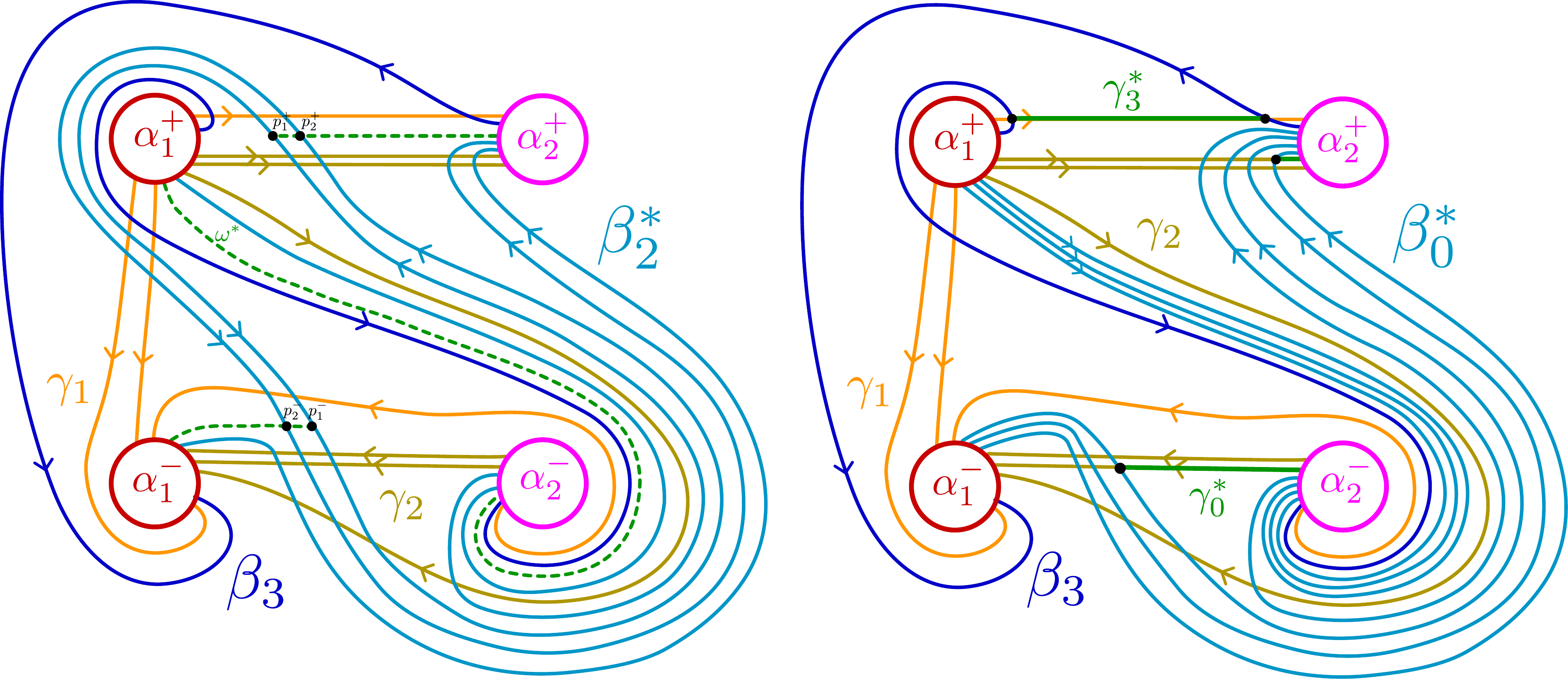}
\put(-320,-20){\large (a)}
\put(-110,-20){\large (b)}
\caption{The fat graph $\Sigma_\A(\n',\g)$ after (a) the first wave move in $\omega^*$ and (b) after all of the wave moves in $\omega^*$.  In (b), we see a pair of wave-busting arcs $\{\gamma_0^*,  \gamma_3^*\}$.}
\label{fig:sec5_WaveSequence}
\end{figure}

Observe that by construction, the curve $\n^*_0$ does not intersect the edge $a_+$ essentially; however, for $i \geq 1$, we have that $\n^*_i$ intersects all edges of $G_{\A}(\g)$ essentially.  As in the argument for the above case (shown in Figure \ref{fig:sec5_RectBust}), it follows that for $i \geq 1$, $\n^*_i$ intersects the rectangles $R_1$ and $R_2$ in a wave-busting pairs; thus $G_{\g}(\n'_i)$ does not contain a wave by Lemma \ref{sat}.  Lemmas \ref{12gon} and \ref{cgraph} then imply that $G_{\g}(\n'_i)$ is a square graph, and so $\n,\n'_{k},\dots,\n'_1$ is contained in the reducing sequence for $\n$ with respect to $\g$ guaranteed by Lemma \ref{reducing}.  Therefore, $G_{\g}(\n'_0)$ is either a square graph or a c-graph and as such $(\n_0',\g)$ contains a wave based in $\n_0'$. 

Now, each arc of $\n_0^*$ has nonzero winding relative to $\g$ at both $\A_1^-$ and $\A_2^+$, and so there is an arc $\g_0^*$ that originates at the innermost inessential point of $\n_0^* \cap \g$ based at $\A_2^+$, travels through $\A_2$, and extends from $\A_2^-$ along an arc in the edge $a_-$ to meet $\n_0^*$ and has coherently oriented endpoints. See Figure \ref{fig:sec5_WaveSequence}(b). In addition, by Lemma \ref{exit}, the curve $\n_3$ meets $\A_1^+$ in a point between parallel arcs in the edge $a_+$, and we have already established that $0<W_i(\n_3) <1$ for $i=1,2$.  Thus, there is an arc $\g^*_3$ in the edge $a_+$ that meets $\n_3$ in its coherently oriented endpoints.    The pair $(\g^*_0,\g^*_3)$ is a wave-busting pair for $G_{\n'_0}(\g)$, implying that $(\n_0',\g)$ does not contain a wave based in $\n_0'$ by Lemma \ref{sat}, contradicting our previous assumptions.  This completes the proof of the lemma.
\end{proof}
\end{lemma}

There is one final case to consider, which mirrors the proof of Lemma \ref{onewind} almost exactly, and so we give an abbreviated proof.

\begin{lemma}\label{twowind}
If $0 < W_1 < 1$ and $W_2 > 1$, then $(\n,\g)$ is not a Heegaard diagram for $S^3$.
\begin{proof}
The proof in this case is virutally identical to the proof of Lemma \ref{onewind}.  As above, $(\n,\g)$ contains a wave based in $\n$:   An arc connecting $\A_1^+$ to $\A_2^-$ extends to a wave $\omega$ based in $\n_2$.  Let $\n' = \{\n_1,\n_3\}$ be the result of surgery on $\n$ along $\omega$.  If $W_2 > 2$, then $(\n',\g)$ and thus $(\n,\g)$ is not a Heegaard diagram for $S^3$.  Otherwise, we find a nested wave sequence that gives rise to a partial reducing sequence $\n,\n'_k,\dots,\n'_1$ for $\n$ with respect to $\gamma$, implying that $(\n'_0,\g)$ contains a wave based in $\n'_0$, where $\n'_0 = \{\n_0^*,\n_3\}$.  The only distinction in this proof is that the roles of $\g_0^*$ and $\g_3^*$ from the proof of Lemma \ref{onewind} are reversed.  Now, $\g_0^*$ is an arc in $a_+$ connecting coherently oriented points of $\n_0^*$ whereas $\g_3^*$ runs over $\A_2$ to connect coherently points of $\n_3$. See Figure \ref{fig:84}. Of course, $\{\g_0^*,\g_3^*\}$ remains a wave-busting pair for $G_{\n'_0}(\g)$, a contradiction.  
\end{proof}
\end{lemma}

\begin{figure}[h!]
\centering
\includegraphics[scale = .25]{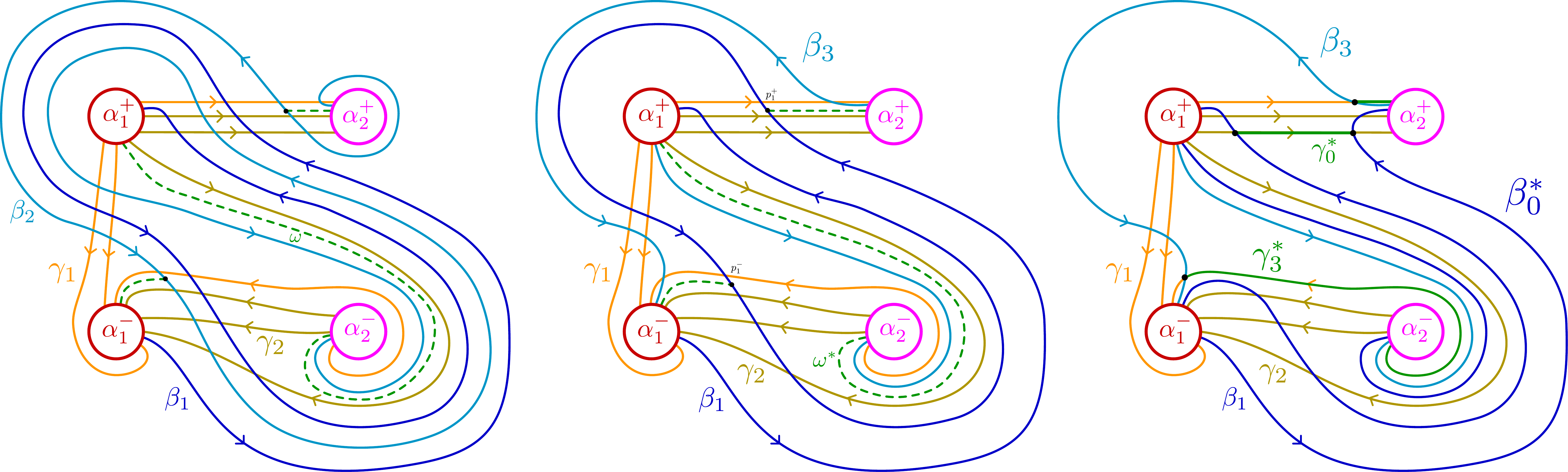}
\put(-360,-20){\large (a)}
\put(-215,-20){\large (b)}
\put(-70,-20){\large (c)}
\caption{The case in which $0<W_1<1$ and $W_2>1$.  The wave $\omega$ in (a), the nested wave sequence $\omega^*$ in (b), and the wave-busting pair $\{\g_0^*,\g_3^*\}$ in (c) appear proof of Lemma \ref{twowind}.}
\label{fig:84}
\end{figure}

We summarize the previous three sections.

\begin{proof}[Proof of Proposition \ref{claim1}]
Suppose that $(\A,\n,\g)$ is a minimal diagram for the genus two trisection $X = X_1 \cup X_2 \cup X_3$ which satisfies $\io(\A,\n) = 2$, and suppose by way of contradiction that $(\A,\g)$ is not the trivial diagram. By the minimality of $(\A,\n,\g)$, the diagram $(\n,\g)$ is also nontrivial.  By Theorem \ref{wave}, there are waves for both $(\A,\g)$ and $(\n,\g)$.  By Lemma \ref{gwave}, $(\A,\g)$ does not contain a wave based in $\g$, and so $G_{\A}(\g)$ is a Type III graph, and we parameterize the two arcs of $\n$ in $\Sigma_{\A}(\n,\g)$ with slope $\frac{m}{n}$.  By Lemma \ref{mone}, we have $m \neq 1$.  In addition, Lemma \ref{nogwave} asserts $(\n,\g)$ does not contain a wave based in $\g$; therefore, it must have a wave based in $\n$, and so $w_{\g}(b) = 0$ in $G_{\A}(\g)$ by Lemma \ref{bnotzero}. 

Lemma \ref{nothard} then implies that $\frac{m}{n}$ must be $-\frac{n \pm 1}{n}$, the case dealt with in this section.  Concerning the winding numbers $W_1$ and $W_2$ of $\n$ with respect to $\g$, Lemmas \ref{noninteger}, \ref{sameint}, \ref{nonewind}, and \ref{bothwind} yield that $W_i > 1$ and $0 < W_j < 1$ for $\{i,j\} = \{1,2\}$.  However, Lemmas \ref{onewind} and \ref{twowind} show that, in these two cases, $(\n,\g)$ is not a Heegaard diagram for $S^3$, a contradiction.  We conclude that $(\A,\g)$ must be the trivial diagram.
\end{proof}

\section{The case in which $\io(\A,\g) = 2$}\label{sec:main}

Proposition \ref{claim1} reveals that a minimal trisection diagram $(\A,\n,\g)$ for a $(2,0)$--trisection satisfies $\io(\A,\n) = \io(\A,\g) = 2$.  In this section, we show that, in addition, $\io(\n,\g)$ must also equal 2, which will complete the proof of the main theorem.  The strategy is to consider the graph $\Sigma_{\A}(\n,\g)$, in which each of curve in $\n \cup \g$ contributes a single arc, and the slopes of $\n$-arcs agree and the slopes of the $\g$-arcs agree.  We may suppose without loss of generality that the $\g$-arcs have slope $\frac{1}{0}$ and that the $\n$--arcs have slope $\frac{m}{n}$, where $m$ is odd and $n$ is even.  Although this choice is somewhat non-standard (it seems more natural to let slopes of $\g$-arcs vary) this convention allows us to import techniques from Section \ref{sec:hard_case} more consistently.  We will orient $\n$ and $\g$ so that oriented arcs in $\Sigma_{\A}(\n,\g)$ originate in $\A_i^-$ and terminate in $\A_i^+$, as shown in Figure \ref{fig:sec6_slopes}. \\

\begin{figure}[h!]
\centering
\includegraphics[scale = .75]{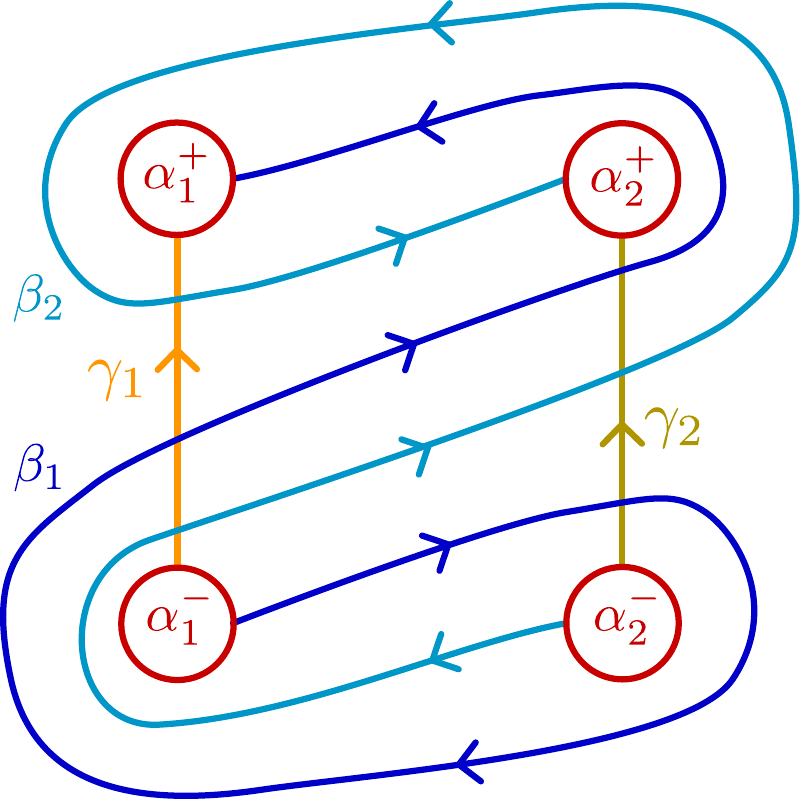}
\caption{An example of $\Sigma_\A(\n,\g)$ in which the $\n$--arcs have slope $\frac{3}{4}$, and the winding numbers $W_1$ and $W_2$ are both zero.}
\label{fig:sec6_slopes}
\end{figure}

Recall the homeomorphisms $\varphi$ and $\varphi^{-1}$ of $\Sigma_\A$ defined in Section \ref{sec:prelims} and given by matrices
\[ M_{\varphi} = \begin{pmatrix}
1 & 1 \\ 0 & 1 \end{pmatrix} \qquad \text{ and } \qquad 
M_{\varphi^{-1}} = \begin{pmatrix}
1 & -1 \\ 0 & 1 \end{pmatrix}.\]
Note that these homeomorphisms preserve $\g$ setwise but exchange $\A_2^+$ and $\A_2^-$, which results in a reversal of orientations of $\n_2$ and $\g_2$.  When we apply such a homeomorphism, we will always re-label vertices and re-orient the curves in $\n$ and $\g$ to match the conventions given above.  See Figure \ref{fig:sec6_automorphisms}.

\begin{lemma}\label{smallslope}
The slope $\frac{m}{n}$ for $\n$ may be chosen so that $-\frac{1}{2} < \frac{m}{n} \leq \frac{1}{2}$.
\begin{proof}
Observe that $\varphi^{\pm 1}$ preserves $\g$ and takes the slope $\frac{m}{n}$ of $\n$ to $\frac{m\pm n}{n}$.  Since there exists an integer $m'$ such that $-\frac{n}{2} < m' \leq \frac{n}{2}$ and $m' \equiv m \, \, \text{mod}(n)$, we may map $\n$-arcs to arcs of slope $\frac{m'}{n}$ by repeated applications of either $\varphi$ or $\varphi^{-1}$, as desired.
\end{proof}
\end{lemma}

\begin{figure}[h!]
\centering
\includegraphics[scale = .55]{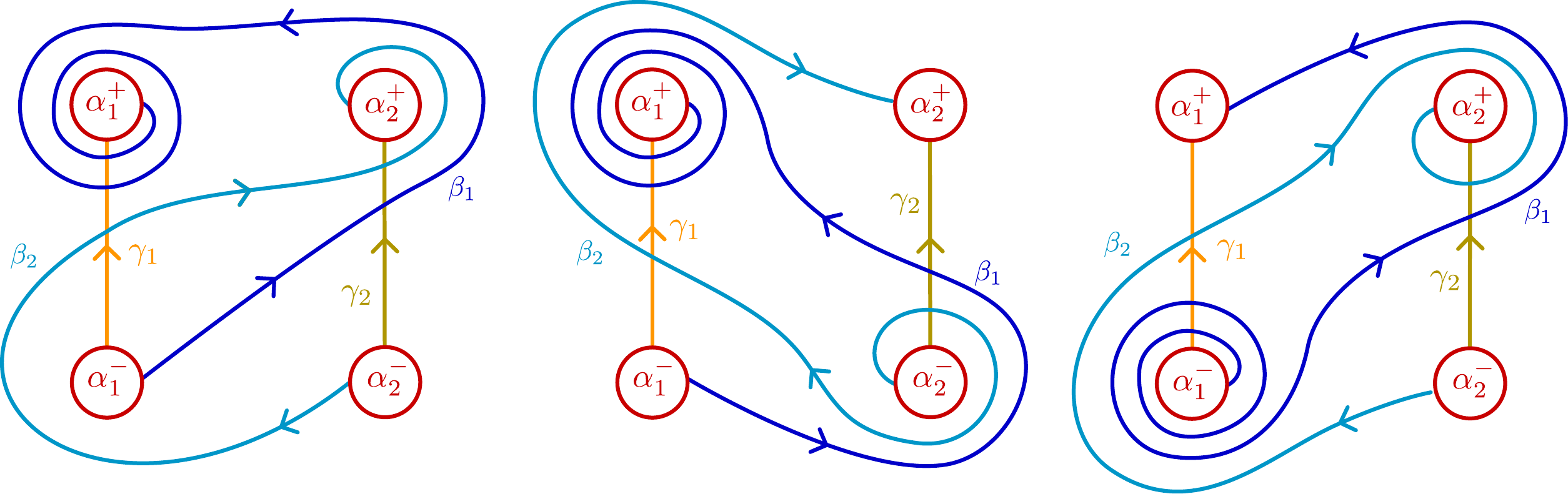}
\put(-345,-25){\large$G$}
\put(-220,-25){\large$\varphi^{-1}(G)$}
\put(-95,-25){\large$H\circ\varphi^{-1}(G)$}
\caption{The effects of the homeomorphisms $\varphi^{-1}$ and $H \circ \varphi^{-1}$ on an example of $\Sigma_\A(\n,\g)$.}
\label{fig:sec6_automorphisms}
\end{figure}

As in previous sections, we will systematically rule out various configurations of $\n$ and $\g$ in $\Sigma_\A$ until the only remaining cases are the ones in which $\io(\n,\g) = 2$, shown in Figure \ref{fig:StdDiags}.

\begin{lemma}\label{bigm}
If $m > 1$, then $(\n,\g)$ is not a Heegaard splitting of $S^3$.
\begin{proof}
We will show that there are wave-busting arcs contained in both $\n$ and $\g$, implying that $(\n,\g)$ does not contain a wave.  By applying the horizontal reflection $H$ if necessary, we may suppose without loss of generality that $-\frac{1}{2} < \frac{m}{n} < - \frac{1}{|n|}$.  Then, $\frac{m}{n} \neq -\frac{n\pm 1}{n}$ and by hypothesis, $m > 1$.  Thus, by an argument identical to that of Lemma \ref{nothard}, there exists a wave-busting arc $\g^* \subset \g_1$ for $(\n,\g)$ whose endpoints $q_1$ and $q_2$ are oppositely oriented points in $\n_1 \cap \g_1$.  However, in this situation, $q_1$ and $q_2$ also cobound an arc $\n^* \subset \n_1$ which meets $\g_2$ in a single point, so that $\n^*$ is a wave-busting arc for $(\n,\g)$.  See Figure \ref{fig:sec6_winding}(a).  By Lemma \ref{sat}, neither $G_{\n}(\g)$ nor $G_{\g}(\n)$ contains a wave, and thus $(\n,\g)$ is not a Heegaard diagram for $S^3$ by Theorem \ref{wave}.
\end{proof}
\end{lemma}

\begin{figure}[h!]
\centering
\includegraphics[scale = .8]{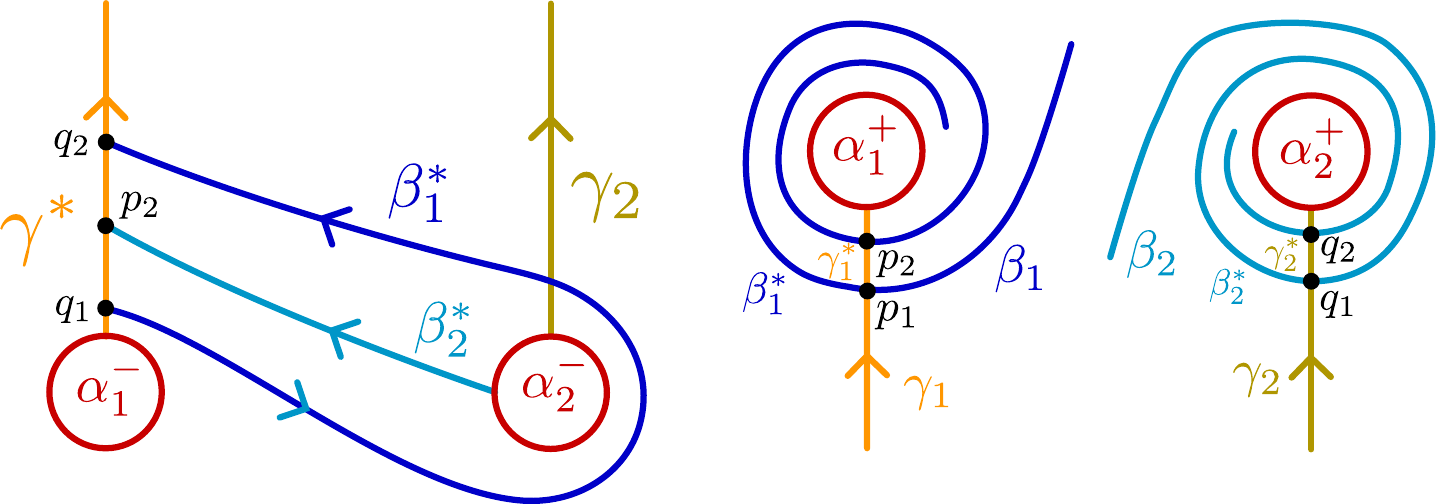}
\put(-260,-15){\large (a)}
\put(-90,-15){\large (b)}
\caption{In (a), we see the wave-busting arcs $\n_1^*$ and $\g^*$ that exist when $m>1$. In (b), we see the wave-busting pairs $\{\n_1^*, \n_2^*\}$ and $\{\g_1^*, \g_2^*\}$ that exists when both windings are large.}
\label{fig:sec6_winding}
\end{figure}

We will assume for the remainder of the section that $m = 1$.  As in Section \ref{sec:hard_case}, for simplicity we will isotope $\n$ so that all inessential intersections of $\n$ and $\g$  are based in $\A_1^+$ and $\A_2^+$, and we define the winding $W_i$ of $\n_i$ relative to $\g_i$ at $\A_i$ as above.  The main difference in this case is that $\io(\A_i,\g) = 1$, so $W_i$ is a (signed) integer.

\begin{lemma}\label{muchwind}
If $|W_i| > 1$ for both $i=1,2$, then $(\n,\g)$ is not a Heegaard splitting of $S^3$.
\begin{proof}
If $|W_i| > 1$, then there is an arc $\g_i^* \subset \g_i$ which meets $\n_i$ in its coherently oriented endpoints and avoids $\n$ in its interior.  See Figure \ref{fig:sec6_winding}(b).  Thus, $(\g_1^*,\g_2^*)$ is a wave-busting pair and $G_{\n}(\g)$ does not contain a wave by Lemma \ref{sat}.  Let $p_1,p_2$ denote the endpoints of $\g_1^*$ and $q_1,q_2$ denote the endpoints of $\g_2^*$.  Note that there are arcs $\n_1^* \subset \n_1$ and $\n_2^* \subset \n_2$ which have endpoints on $p_1,p_2$ and $q_1,q_2$, respectively, and which avoid $\g$ in their interiors.  Thus, $(\n_1^*,\n_2^*)$ is also wave-busting pair, and we conclude that $(\n,\g)$ does not contain a wave and, as such, is not a Heegaard splitting of $S^3$.
\end{proof}
\end{lemma}

Recall from Section \ref{sec:start} the definition of the $0$--replacement for a set of curves $\delta$ such that $\Sigma_{\A}(\delta)$ contains two arcs of slope $\frac{1}{n}$.

\begin{lemma}\label{replace}
Suppose $W_2 = 0$.  If $\n_3$ and $\g_3$ are the $0$--replacements for $\n$ and $\g$, respectively, then $\io(\n_2,\g_3) = \io(\n_3,\g_2) = 1$.  In addition, if $W_1(\n_1) \geq 2$, then $\io(\n_3,\g_3) = W_1(\n_1)-2$.  
\begin{proof}
The fact that $\io(\n_2,\g_3) = \io(\n_3,\g_2) = 1$ is clear from Figure \ref{fig:sec6_even_odd}.  Let $k = W_1(\n_1) \geq 2$, and suppose first that $k$ is even.  Isotope $\n_1$ so that $\frac{k}{2}$ of the inessential intersections of $\n_1$ and $\g_1$ are based at $\A_1^+$ and $\frac{k}{2}$ are based at $\A_1^-$, and so that $\n_1 \cap \A_1 = \g_3 \cap \A_1$.  Then $\n_3$, isotoped to remain disjoint from $\n_1$, has precisely $\frac{k}{2}-1$ essential intersections with $\g_3$ based at $\A_1^+$ and precisely $\frac{k}{2} - 1$ based at $\A_1^-$, yielding $\io(\n_3,\g_3) = k-2$.  See Figure \ref{fig:sec6_even_odd}(a).

If $k$ is odd, isotope $\n_1$ so that one inessential intersection of $\n_1$ and $\g_1$ is contained in $\A_1$, $\frac{k-1}{2}$ are based at $\A_1^+$, and $\frac{k-1}{2}$ are based at $\A_1^-$.  Then $\n_3$ may be isotoped so that one inessential intersection of $\n_3$ and $\g_3$ is contained in $\A_1$; otherwise, $\n_3$ and $\g_3$ have $\frac{k-1}{2}  -1$ inessential intersections based at $\A_1^+$ and $\frac{k-1}{2} - 1$ based at $\A_1^-$.  Thus, in this case as well, $\io(\n_3,\g_3) = k-2$.  See Figure \ref{fig:sec6_even_odd}(b)
\end{proof}
\end{lemma}

\begin{figure}[h!]
\centering
\includegraphics[scale = .75]{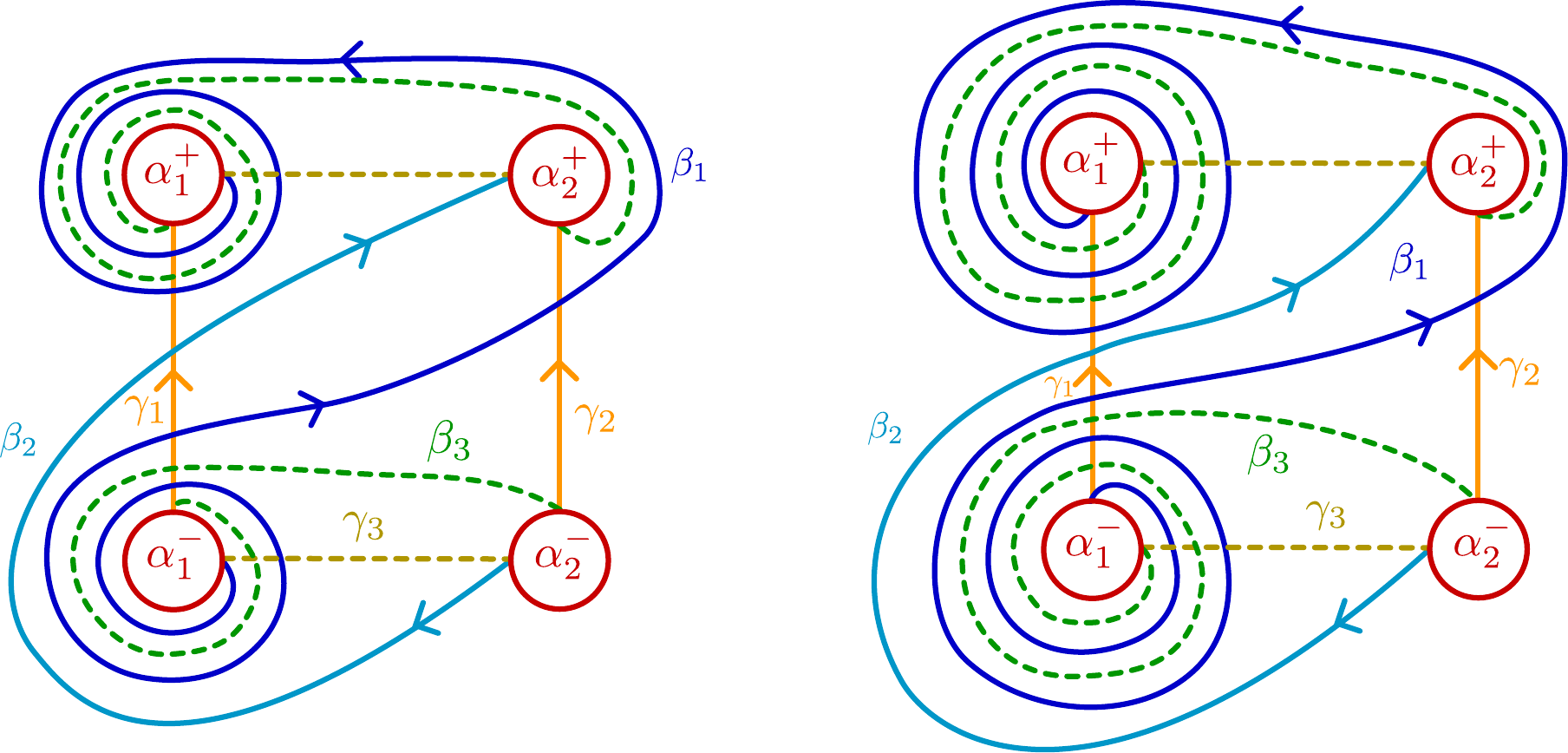}
\put(-325,-15){\large (a)}
\put(-105,-15){\large (b)}
\caption{The $0$--replacements for $\n$ and $\g$ in the case that $W_2=0$ and $W_1$ is (a) even and (b) odd. As shown, the $\n$--arcs have slope $\frac{1}{2}$, and we have (a) $W_1(\n_1)=4$ and (b) $W_1(\n_1)=5$.  The curve $\A_3$ is suppressed for clarity (cf. Figure \ref{fig:0-replacements}).}
\label{fig:sec6_even_odd}
\end{figure}

Finally, we use the techniques developed up to this point to prove that if $(\A,\n,\g)$ is a minimal trisection diagram, then each Heegaard diagram $(\A,\n)$, $(\A,\g)$, and $(\n,\g)$ is standard, from which the main theorem follows.

\begin{theorem}\label{claim2}
If $(\A,\n,\g)$ is a minimal diagram for a $(2,0)$--trisection $X = X_1 \cup X_2 \cup X_3$, then $\io(\A,\n) = \io(\A,\g) = \io(\n,\g) = 2$ and $(\A,\n,\g)$ is homeomorphic to one of the standard diagrams pictured in Figure \ref{fig:StdDiags}.
\begin{proof}
By Proposition \ref{claim1}, if $(\A,\n,\g)$ is minimal, then $\io(\A,\n) = \io(\A,\g) = 2$.  It only remains to show that $\io(\n,\g) = 2$.  We may assume that in $G_{\A}(\n,\g)$, the two $\g$-arcs have slope $\frac{1}{0}$, and by Lemmas \ref{smallslope} and \ref{bigm}, the two $\n$-arcs have slope $\frac{1}{n}$ for some even $n$.  Additionally, by Lemma \ref{muchwind}, we may assume without loss of generality that $|W_2| \leq 1$.

First, suppose that  $|n| \geq 4$.  If $n < 0$, we may apply the horizontal reflection $H$; hence, we assume that $n > 0$.  By Lemma \ref{inter}, we have $\n_i \cap \g_i$ contains $\frac{n}{2}-1$ essential intersections, and, for $i \neq j$, $\n_i \cap \g_j$ contains $\frac{n}{2}$ essential intersections.  Choose orientations on $\n$ and $\g$ so that all essential intersections are coherently oriented.  Since $n \geq 4$, each of $\n_1$ and $\n_2$ intersects both $\g_1$ and $\g_2$ essentially.  This implies that $G_{\n}(\g)$ contains an edge connecting $\n_1^+$ to $\n_2^+$ and $G_{\g}(\n)$ contains an edge connecting $\g_1^+$ to $\g_2^+$.  If $\n_1 \cap \g_1$ contains an inessential point of intersection which is oriented opposite the essential intersections, then there is an arc $\n_1^* \subset \n_1$ connecting an inessential point of $\n_1 \cap \g_1$ and an essential point of $\n_1 \cap \g_2$.  Thus, $G_{\g}(\n)$ contains an edge connecting $\g_2^+$ to $\g_1^-$, and, by the arguments in Lemma \ref{sat}, $G_{\g}(\n)$ does not contain a wave.  Similarly, $\g_1$ contains an arc connecting an inessential point of $\n_1 \cap \g_1$ to an essential point of $\n_2 \cap \g_1$, thus $G_{\n}(\g)$ does not contain a wave by the same reasoning.  A parallel argument shows that if $\n_2 \cap \g_2$ contains an inessential intersection oriented opposite the essential intersections, then neither $G_{\g}(\n)$ nor $G_{\n}(\g)$ contains a wave, so by Theorem \ref{wave}, $(\n,\g)$ is not a Heegaard diagram for $S^3$.

Therefore, we may assume that all points of $\n \cap \g$ are coherently oriented, so that $W_i \geq 0$ for $i=1,2$ and $(\n,\g)$ is a positive Heegaard diagram.  If $W_2 = 1$, then $\io(\n_2,\g_1) = \io(\n_2,\g_2) = \frac{n}{2} > 1$.  But this implies that $\frac{n}{2}$ divides $\text{det}(M(\n,\g))$, so that $M(\n,\g) \neq S^3$ by Lemma \ref{determ}, a contradiction.  It follows that $W_2 = 0$.

A similar argument shows that $W_1 \neq 1$.  If $W_1 \geq 2$, let $a$ be a seam of slope $\frac{0}{1}$, let $\A_3 = \Delta_a$, and let $\n_3$, and $\g_3$ be the $0$--replacements for $\n$ and $\g$, respectively, as in Figure \ref{fig:sec6_even_odd}.  In addition, let $\A' = \{\A_1,\A_3\}$, $\n' = \{\n_2,\n_3\}$, and $\g' = \{\g_2,\g_3\}$, so that $\io(\A',\n') = \io(\A',\g') = 2$.  Using Lemma \ref{replace}, we compute
\begin{eqnarray*}
\io(\n',\g') &=& \io(\n_2,\g_2) + \io(\n_2,\g_3) + \io(\n_3,\g_2) + \io(\n_3,\g_3) \\
&=& \left(\frac{n}{2} - 1\right) +  1 + 1 + W_1(\n_1) - 2 \\
&=& \frac{n}{2} + W_1(\n_1) - 1.
\end{eqnarray*}
In addition, we have
\begin{eqnarray*}
\io(\n,\g) &=& \io(\n_1,\g_1) + \io(\n_1,\g_2) + \io(\n_2,\g_1) + \io(\n_2,\g_2) \\
&=& \left(\frac{n}{2} - 1 + W_1(\n_1)\right) +  \frac{n}{2} + \frac{n}{2} + \left(\frac{n}{2} - 1\right) \\
&=& 2n + W_1(\n_1) - 2.
\end{eqnarray*}
By the minimality of $(\A,\n,\g)$, we have
\begin{equation}\label{smallwind}
2n + W_1(\n_1) - 2 \leq \frac{n}{2} + W_1(\n_1) - 1
\end{equation}
and thus $n \leq \frac{2}{3}$, a contradiction.

The only remaining possibility in this case is that $W_1 = 0$.  However, this implies that the intersection matrix of $(\n,\g)$ is
\[ M(\n,\g) = \begin{pmatrix}
\frac{n}{2} - 1 & \frac{n}{2} \\
\frac{n}{2} & \frac{n}{2} - 1
\end{pmatrix}.\]
By Lemma \ref{determ}, we have $|\text{det}(M(\n,\g))| = |n-1| = 1$, and so $n = 0$ or $n=2$, contradicting the assumption that $n \geq 4$.

Now suppose that $n=2$.  In this case, intersection matrix for $(\n,\g)$ is
$$ M(\n,\g) =  \begin{pmatrix}
W_1 & 1 \\
1 & W_2 \end{pmatrix}.$$
It follows from Lemma \ref{determ} that $W_1W_2 = 2$ or $W_1W_2 = 0$.  We may observe that if $W_1,W_2 \leq 0$, we may replace $\n$ and $\g$ with curves having the same slopes in $\Sigma_\A$ and nonnegative winding by applying $H \circ \varphi^{-1}$ to $\Sigma_\A$.   See Figure \ref{fig:sec6_automorphisms}.  Thus, suppose first that $W_2 = 1$, so that $W_1 = 2$.  In this case, there is a sequence of handle slides that reduces complexity.  To see this, let $a$ be a seam of slope $\frac{0}{1}$, let $\A_3 = \Delta_a$, and let $\n_3$, and $\g_3$ be the $0$--replacements for $\n$ and $\g$, respectively, as in Figure \ref{fig:0-replacements}(a).  In addition, let $\A' = \{\A_1,\A_3\}$, $\n' = \{\n_2,\n_3\}$, and $\g' = \{\g_2,\g_3\}$.  Then, $\io(\A',\n') = \io(\A',\g') = 2$, but $\io(\n',\g')=\io(\n,\g)-3=2$.  It follows that the $\mathcal C(\A',\n'\g')<\mathcal C(\A,\n,\g)$, a contradiction.


\begin{figure}[h!]
\centering
\includegraphics[scale = .70]{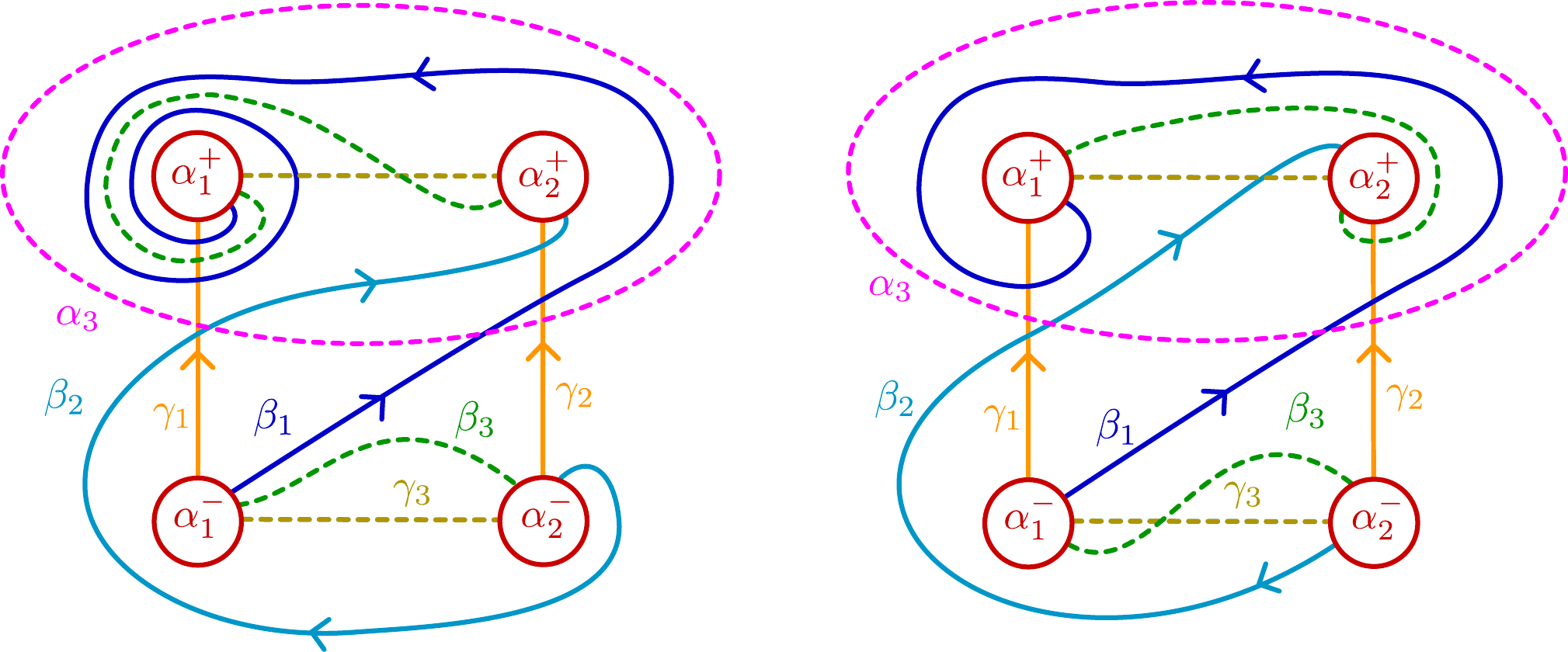}
\put(-320,-15){\large (a)}
\put(-103,-15){\large (b)}
\caption{The 0--replacements in the case where  the slope of the $\n$--arcs is $\frac{1}{2}$ and (a) $W_1=2$ and $W_2=1$ and (b) $W_1=1$ and $W_2=0$.}
\label{fig:0-replacements}
\end{figure}

In the other remaining case, suppose that $W_2 = 0$.  If $W_1 \geq 2$, we note that Inequality (\ref{smallwind}) is true in this case as well, and thus $n \leq \frac{2}{3}$ by above arguments, a contradiction.  If $W_1 = 1$, then as above we perform handle slides in $\A$, $\n$, and $\g$ to reduce complexity:  Let $a$ be a seam of slope $\frac{0}{1}$, let $\A_3 = \Delta_a$, and let $\n_3$, and $\g_3$ be the $0$--replacements for $\n$ and $\g$, respectively, as in Figure \ref{fig:0-replacements}(b).  In this case, the handle slides are slightly different than above; let $\A' = \{\A_2,\A_3\}$, $\n' = \{\n_1,\n_3\}$, and $\g' = \{\g_1,\g_3\}$.  Then $\io(\A',\n') =\io(\A',\g') = 2$, but $\io(\n',\g') = \io(\n, \g) - 1 = 2$, contradicting that $(\A,\n,\g)$ is minimal.  The only remaining case when $n = 2$ is that $W_1 = W_2 = 0$; thus $(\A,\n,\g)$ is the standard diagram in Figure \ref{fig:StdDiags}(a) (with $\n$ and $\g$ reversed). 


Finally, suppose that $n = 0$, so that both the $\n$ and the $\g$ arcs have the same slope.  In this case
$$ M(\n,\g) =  \begin{pmatrix}
W_1 & 0 \\
0 & W_2 \end{pmatrix}$$
and thus $W_1,W_2 = \pm 1$ by Lemma \ref{determ}.  It follows that $(\A,\n,\g)$ is homeomorphic to one of the standard diagrams in Figures \ref{fig:StdDiags}(b) and \ref{fig:StdDiags}(c). (If $W_1=W_2=-1$, then apply $H$ to recover Figure \ref{fig:StdDiags}(b).)
\end{proof}
\end{theorem}

\section{Cosmetic surgery, trisections, and exotic 4--manifolds}\label{sec:surgery}

In the final section, we complete the proof of Theorem \ref{thm:main} by proving Proposition \ref{21class}.  We also explain and prove Corollary \ref{cor:surgery}.

Let $L=K_1\cup\cdots\cup K_n$ be an $n$--component link in a compact 3--manifold $Y$ and let $T_i$ denote $\pd \overline{\nu(K_i)}$ in the exterior $Y_L$ of $L$, where $Y_L = Y \setminus \nu(L)$.  Suppose $\vec\rho=(\rho_1,\ldots, \rho_n)$ is a framing on $L$; that is, each $\rho_i$ is a boundary slope in $T_i$.  Let $Y_{\vec\rho}(L)$ denote the 3--manifold obtained by Dehn surgery on $L$ with slope $\vec\rho$.  We will discuss surgeries having the property that $Y_{\vec\rho}(L) \cong Y$.  Clearly, $Y_{(1/0,\dots,1/0)}(L) = Y$, and so this case is completely uninteresting.  However, surgeries with the property $Y_{\vec\rho}(L)\cong Y$ and $\rho_i\not=\frac{1}{0}$ for all $i$ are quite uncommon, and we call such a surgery \emph{cosmetic}.

Now, suppose that $X$ is a closed 4--manifold admitting a $(g,k)$--trisection.  Then $X$ has a handle decomposition consisting of one 0--handle, $k$ 1--handles, $g-k$ 2--handles, $k$ 3--handles, and one 4--handle.  Let $X^{(i)}$ denote the union of the handles of index at most $i$.  By \cite{laudenbach-poenaru}, there is a unique way to attach the union of 3--handles and the 4--handle to the boundary of $X^{(2)}$, so $\partial X^{(1)}\cong\partial X^{(2)}\cong \#^k(S^1\times S^2)$. Let $L$ denote the union of the attaching circles of the 2--handles.  Then $\partial X^{(2)}$ is obtained from surgery on $L$ in $\partial X^{(1)}$, and thus $L$ is a $(g-k)$--component link in $\#^k(S^1\times S^2)$ with a cosmetic surgery.

We use this set-up to conclude our analysis of genus two trisections.

\subsection{(2,1)--trisections are standard}\ 

If $X$ admits a (2,1)--trisection, then $X$ has a decomposition with one 1--handle, one 2--handle, and one 3--handle.  The 2--handle, call it $h$, is attached along a knot $L$ in $S^1\times S^2$, the framing of $h$ gives a cosmetic surgery slope for $L$ in $S^1 \X S^2$.  However, by work of Gabai, $S^1\times S^2$ admits no nontrivial cosmetic surgeries \cite{gabai-I}.  More specifically, we have the following theorem.

\begin{theorem}\cite{gabai-I}\label{cosm}
Suppose $L$ is a knot in $S^2 \X S^1$ with a cosmetic surgery.  Then $L$ is a $(\pm 1)$--framed unknot.
\end{theorem}

A Heegaard splitting of a compact 3--manifold $Y$ with boundary is a decomposition of $Y$ as $C_1\cup_{\Sigma} C_2$, where $C_1$ and $C_2$ are compression bodies.  We will require the following generalization of Haken's Lemma.

\begin{lemma}\cite{haken}\label{haken}
Let $Y$ be a compact 3--manifold with a Heegaard splitting $Y = C_1 \cup_{\Sigma} C_2$, and suppose that $D$ is a properly embedded disk in $Y$.  Then there exists a disk $D' \subset Y$ such that $\pd D' = \pd D$ and $D' \cap \Sigma$ is a single simple closed curve.
\end{lemma}

For two distinct trisections (of arbitrary genus) $X' = X'_1 \cup X'_2 \cup X'_3$ and $X''= X_1'' \cup X_2'' \cup X_3''$, there is a natural trisection of the connected sum $X' \# X''$.  The easiest way to see this is to take the connected sum of trisection diagrams $(\Sigma',\A',\n',\g')$ and $(\Sigma'',\A'',\n'',\g'')$, which yields an induced trisection diagram $(\Sigma' \# \Sigma'',\A' \cup \A'',\n' \cup \n'',\g' \cup \g'')$ for the natural trisection of $X' \# X''$.  Note that the induced diagram has the property that there is a separating curve in $\Sigma = \Sigma' \# \Sigma''$ which bounds a compressing disk in each of the three 3--dimensional handlebodies determined by the three sets of attaching curves.  In this case we call the trisection \emph{reducible}.  Conversely, if $X = X_1 \cup X_2 \cup X_3$ is a reducible trisection, then it can be written as the connected sum of trisections of 4--manifolds $X'$ and $X''$ such that $X = X' \# X''$.  For more details, see \cite{gay-kirby:trisections}.

In the following proposition, we show that every $(2,1)$--trisection is reducible, and, as such, can be written as the conected sum of a $(1,0)$-- and a $(1,1)$--trisection.  Since the only $X'$ admitting a $(1,0)$--trisection are $\CP^2$ and $\overline{\CP}^2$, the only $X''$ admitting a $(1,1)$--trisection is $S^1 \X S^3$, and the genus one trisections of these manifolds are unique up to diffeomorphism \cite{gay-kirby:trisections}, this completes the classification of $(2,1)$--trisections.

\begin{proposition}\label{21class}
Every $(2,1)$--trisection is reducible.
\begin{proof}
Let $X = X_1 \cup X_2 \cup X_3$ be a $(2,1)$--trisection with trisection surface $\Sigma$.  As described above, $X$ admits a handle decomposition with one 1--handle, one 2--handle which we will call $h$, and one 3--handle.  Let $Y = \pd X_1 = \partial X^{(1)} \cong S^2 \X S^1$.  By Lemma 13 of \cite{gay-kirby:trisections}, there is a trisection diagram $(\A,\n,\g)$ for the trisection so that $(\A,\n)$ is a Heegaard diagram for $Y \cong S^2 \X S^1$, the Heegaard diagram $(\A,\g)$ satisfies $\io(\A_1,\g_1) = 1$ and $\A_2 = \g_2$, and $h$ is attached to $Y$ along $\g_1$ with framing given by the surface framing of $\g_1$ in $\Sigma$.

Let $H_{12} = X_1 \cap X_2$ and $H_{13} = X_1 \cap X_3$ be the 3--dimensional handlebodies determined by $\A$ and $\n$, respectively, so that $H_{23} = X_2 \cap X_3$ is the handlebody determined by $\g$.  We will show that there is a separating curve $\delta \subset \Sigma$ which bounds a disk in each of $H_{12}$, $H_{13}$, and $H_{23}$.  By Theorem \ref{cosm}, $\g_1$ bounds a disk in $Y = H_{12} \cup_{\Sigma} H_{13}$, and $\g_1$ has surface framing $\pm 1$ in $\Sigma$.  Since $\io(\A_1,\g_1) = 1$ and $\A_1$ bounds a disk in $H_{12}$, a positive or negative Dehn twist of $\g_1$ about $\A_1$ yields a curve $\n_*$ which has surface framing zero and is isotopic in $H_{12}$ to $\gamma_1$.  Note that $\n_*$ is also isotopic to a core of $H_{12}$ and that $\io(\A_1,\n_*) = \io(\g_1,\n_*) = 1$.

Since $\n_*$ is a 0--framed unknot in $\Sigma$, there is an embedded disk $D \subset Y$ such that $\pd D = \n_*$ and a collar neighborhood of $\pd D$ is disjoint from $\Sigma$.  Let $D'$ be the image of $D$ under an isotopy which pushes $\n_*$ slightly into the interior of $H_{12}$, and let $\n' = \pd D'$.  Then $\n'$ is a core of $H_{12}$ and $C = H_{12} \setminus \nu(\n')$ is a compression body.   In addition, let $T = \pd \overline{\nu(\n')}$, so that $D \setminus \nu(\n')$ is a compressing disk for $Y \setminus \nu(\n')$ with boundary in $T$.  Then $Y \setminus \nu(\n') = C \cup_{\Sigma} H_{13}$ is a Heegaard splitting satisfying the hypotheses of Lemma \ref{haken}.  It follows that there is a disk $D_0$ properly embedded in $Y \setminus \nu(\n')$ such that $\pd D_0 = D' \cap T$ and such that $D_0$ intersects $\Sigma$ in a single simple closed curve $\n_0$.  This implies that $\n_0$ bounds a disk in $H_{13}$. 

The compression body $C$ may be viewed as the union of product neighborhood $T \X I$, where $T = T \X \{0\}$, and a three--dimensional 1--handle attached along $T_1 = T \X \{1\}$ whose cocore is a disk bounded by $\A_2$ in $H_{12}$.  By construction, there are annuli $A' \subset D' \cap C$ and $A_0 \subset D_0 \cap C$ such the boundary of $A'$ in $T$ agrees with the boundary of $A_0$ in $T$, while $\pd A' \cap \Sigma = \n_*$ and $\pd A_0 \cap \Sigma  = \n_0$.  Moreover, since $T$ is incompressible in $C$, these annuli are incompressible.  By a standard innermost disk and outermost arc argument, each annulus may be isotoped into $T \X I$ so that it is vertical with respect to the product structure $T \X I$.  It follows that, after isotopy, $\n_*$ and $\n_0$ are disjoint and thus isotopic in $T_1$, although (depending on where the 1--handle is attached) $\n_*$ and $\n_0$ may not be isotopic in $\Sigma = \pd_+ C$.  See Figure \ref{fig:ThreeFrames}(a).

\begin{figure}[h!]
\centering
\includegraphics[scale = .55]{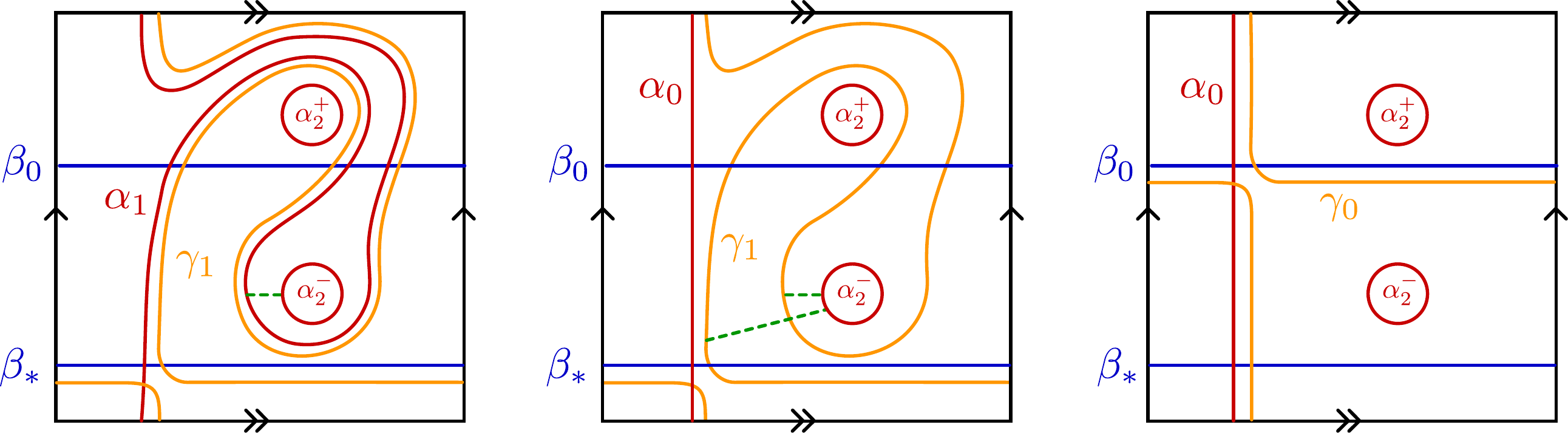}
\put(-353,-20){\large (a)}
\put(-210,-20){\large (b)}
\put(-66,-20){\large (c)}
\caption{Three pictures of $T_1\cup$(1--handle) showing wave moves taking $\A_1$ and $\g_1$ to $\A_0$ and $\g_0$.}
\label{fig:ThreeFrames}
\end{figure}

Since $\n_*$ is the result of Dehn twisting $\g_1$ around $\A_1$, we have $\nu(\A_1 \cup \n_*) = \nu(\A_1 \cup \g_1)$.  The same may not be true for $\A_1$ and $\g_1$ with respect to $\n_0$; in fact it is possible that $\io(\A_1,\n_0),\io(\g_1,\n_0) > 1$.  See Figure \ref{fig:ThreeFrames}(a).  However, there are curves $\A_0$ and $\g_0$ isotopic to $\A_1$ and $\g_1$ in $T_1$ such that $\io(\A_0,\n_0) = \io(\A_0,\g_0) = \io(\n_0,\g_0) = 1$ and $\n_0$ is the result of Dehn twisting $\g_0$ around $\A_0$.  Moreover, by virtue of this isotopy, there is a sequence of handle slides over $\A_2$ taking $\A_1$ to $\A_0$.  Likewise, there is a sequence of handle slides over $\g_2 = \A_2$ taking $\g_1$ to $\g_0$, as in Figures \ref{fig:ThreeFrames}(b) and \ref{fig:ThreeFrames}(c).  We conclude that $\A_0$, $\n_0$, and $\g_0$ bound disks in the handlebodies determined by $\A$, $\n$, and $\g$.

Let $\delta = \pd \overline{\nu(\A_0 \cup \g_0)}$.  Since $\A_0$ bounds a disk $\Delta \subset H_{12}$, we have $\delta$ bounds a disk in $H_{12}$ consisting of two copies of $\Delta$ banded together along $\g_0$.  A similar argument shows that $\g_0$ bounding a disk in $H_{23}$ implies that $\delta$ bounds a disk in $H_{23}$ as well.  Finally, we note that $\delta$ is also equal to $\pd \overline{\nu(\A_0\cup \n_0)}$.  Thus, since $\n_0$ bounds a disk in $H_{13}$, the separating curve $\delta$ also bounds a disk in $H_{13}$.  It follows that the trisection is reducible, completing the proof.
\end{proof}
\end{proposition}

\subsection{Cosmetic 3--sphere surgeries and exotic 4--manifolds}\ 

Although neither $S^3$ nor $S^2 \X S^1$ admits a cosmetic surgery on a knot $L$, the situation is quite different when $L$ is a 2--component link, leading to the following natural question.

\begin{question}\label{question:cosmetic}
Which $2$--component links in 3--sphere admit a nontrivial cosmetic Dehn surgery?
\end{question}

Properties of these links are studied in \cite{mos} (in which they are called \emph{reflexive}), and Ochiai classifies 2--bridge links with cosmetic surgeries in \cite{ochiai:cosmetic}.

Suppose that $L \subset S^3$ is such a link, so that $L(\rho_1,\rho_2)\cong S^3$.  If $\rho_1,\rho_2\in\Z$, then there is an associated closed 4--manifold $X$ obtained by first attaching $(\rho_i)$--framed 4--dimensional 2--handles to $B^4$ along $L\subset S^3=\partial B^4$, and next capping off the resulting 3--sphere with a 4--dimensional 4--handle.  Since $X$ is simply-connected, a theorem of Whitehead tells us that the homotopy type of $X$ is determined by the symmetric, bilinear intersection form $Q_X$.  See \cite{gs} for complete details.  Since $X$ is closed, $Q_X$ is unimodular, and since $b_2(X)=2$, there are only three choices:
$$Q_X\in\left\{\begin{pmatrix} 1 & 0 \\ 0 & 1 \end{pmatrix}, \begin{pmatrix} 1 & 0 \\ 0 & -1 \end{pmatrix}, \begin{pmatrix} 0 & 1 \\ 1 & 0 \end{pmatrix}\right\}.$$
By Freedman \cite{freedman, freedman-quinn}, it follows that $X$ is homeomorphic to either $S^2\times S^2$, $\CP^2\#\CP^2$, or $\CP^2\#\overline\CP^2$.  Thus, one way to look for an exotic simply-connected 4--manifold $X$ with $b_2(X) = 2$ is to limit the search to those $X$ arising from a cosmetic surgery on a 2--component link $L \subset S^3$.


In this setting, we can apply the results of Theorem \ref{thm:main} to place restrictions on such a link.

\begin{lemma}\label{cosmeticlink}
Let $(\Sigma,\A,\n)$ be a genus two Heegaard diagram for $S^3$, and let $\g$ be a cut system for $\Sigma$ with surface framing $\vec{\rho}$.  Then $(\Sigma,\A,\n,\g)$ is a $(2,0)$-trisection diagram if and only if $L(\vec\rho) \cong S^3$.
\begin{proof}
Let $S^3 = H_\A \cup_{\Sigma} H_{\n}$ be the Heegaard splitting determined by $(\Sigma,\A,\n)$, and note that we may construct $L(\vec\rho)$ by first gluing two 3--dimensional 2--handles to $H_{\A}$ along $\g$ to get a compact 3--manifold $H_\A(\g)$ and gluing two 3--dimensional 2--handles to $H_\n$ along $\g$ to get $H_{\n}(\g)$, followed by attaching $H_{\A}(\g)$ to $H_{\n}(\g)$ along their common 2-sphere boundary.  In addition, $H_{\A}(\g)$ is a 3--ball if and only if $(\Sigma,\A,\n)$ is a Heegaard diagram for $S^3$, and a similar statement holds for $H_{\n}(\g)$.  Finally, $L(\vec\rho) = H_{\A}(\g) \cup H_{\n}(\g) \cong S^3$ if and only if both $H_{\A}(\g)$ and $H_{\n}(\g)$ are 3--balls, completing the proof.
\end{proof}
\end{lemma}

\begin{proof}[Proof of Corollary \ref{cor:surgery}]
Suppose that $L$ is a two-component link contained in a genus two Heegaard surface $\Sigma$ for $S^3$, and $L$ admits a surface-framed cosmetic Dehn surgery.  Let $(\Sigma,\A,\n)$ be a diagram for the Heegaard splitting of $S^3$ given by $\Sigma$.  Then $(\Sigma,\A,\n,L)$ is a $(2,0)$--trisection diagram by Lemma \ref{cosmeticlink}, so by Theorem \ref{thm:main} there is a series of handle slides on $\A$, $\n$, and $L$ taking $(\A,\n,L)$ to one of the standard diagrams pictured in Figure \ref{fig:StdDiags}.  It follows that there is a series of handle slides on $L$ contained in the surface $\Sigma$ that converts $L$ to a $\pm 1$-framed unlink or a $0$-framed Hopf link.
\end{proof}



There is much active research devoted to finding exotic simply connected 4--manifolds with small $b_2$.  This corollary can be viewed as a lower bound of sorts on the complexity of link surgery descriptions of such manifolds.  In particular, we see that such an exotic manifold cannot be obtained by surface-framed surgery on a two-component link with genus two bridge number zero.

Finally, we remark that this corollary may be compared to the Generalized Property R Conjecture for two-component links, which states that if a 2--component link $L$ has a 0--framed $\#^2(S^1 \X S^2)$ surgery, then there is a series of handle slides converting $L$ to the 0--framed unlink.  See \cite{gst} for further details.  For both the corollary and the conjecture, the statements involve taking a link $L$ with a specified surgery and converting it to an obvious canonical example or examples with the same surgery via handle slides.

\bibliographystyle{amsalpha}
\bibliography{G2Biblio.bib}

\end{document}